\documentclass[3p,times]{elsarticle}

\usepackage{bbm}
\usepackage{bm}
\usepackage{color}
\usepackage{xcolor}

\usepackage{graphicx}
\usepackage{epstopdf}
\usepackage{psfrag}

\usepackage{amssymb}
\usepackage{amsmath}
\usepackage{dsfont}
\usepackage{rotating}
\usepackage{pgf,tikz}
\usetikzlibrary{arrows}

\definecolor{qqttzz}{rgb}{0,0.2,0.6}
\definecolor{ffqqqq}{rgb}{1,0,0}
\definecolor{qqwuqq}{rgb}{0,0.39,0}
\definecolor{zzttqq}{rgb}{0.6,0.2,0}
\definecolor{qqqqff}{rgb}{0,0,1}
\definecolor{ttttqq}{rgb}{0.2,0.2,0}

\definecolor{qqwwtt}{rgb}{0,0.4,0.2}
\definecolor{ubqqys}{rgb}{0.29,0,0.51}

\definecolor{wwttqq}{rgb}{0.4,0.2,0}
\definecolor{uuuuuu}{rgb}{0.27,0.27,0.27}
\definecolor{qqzzff}{rgb}{0,0.6,1}
\definecolor{xdxdff}{rgb}{0.49,0.49,1}
\definecolor{ccwwqq}{rgb}{0.8,0.4,0}

\definecolor{ttqqqq}{rgb}{0.2,0,0}

\definecolor{qqzzcc}{rgb}{0,0.6,0.8}

\newcommand{\R}{\mathbb{R}}

\newcommand{\x}{\chi}

\newcommand{\eref}[1]{$(\ref{#1})$}
\newcommand{\p}{\mathbb{\wp}}
\newcommand{\np}{p}
\newcommand{\vv}{\vec{v}}

\newcommand{\nextud}{\vec{n}_{ij}}
\newcommand{\nstd}{\vec{n}_{j}}
\newcommand{\etah}{ \hat{\bm{p}}}
\newcommand{\bphi}{ \bm{\phi}}
\newcommand{\bpsi}{ \bm{\psi}}

\newcommand{\vvh}{\hat{\bm{\vec{v}}}}

\newcommand{\Fvh}{\widehat{\bm{F\vec{v}}}}

\newcommand{\nv}{\vec{n}}
\newcommand{\B}{\mathcal{B}}
\newcommand{\TF}{F}
\newcommand{\TT}{\mbox{\boldmath$T$}}
\newcommand{\QQ}{\mbox{\boldmath$R$}}

\newcommand{\D}{\bm{\mathcal{D}}}
\newcommand{\Q}{\bm{\mathcal{Q}}}
\newcommand{\RM}{\bm{\mathcal{R}}}
\newcommand{\LM}{\bm{\mathcal{L}}}

\newcommand{\Mpsi}{\bm{M}}

\newtheorem{proof}{Proof}
\newtheorem{Lemma}{Lemma}

\newcommand{\A}{\mathcal{A}}

\newcommand{\diff}[2]{\frac{\partial {#1} }{\partial {#2} } }

\journal{Applied Mathematics and Computation}

\begin{document}

\begin{frontmatter}



\title{A staggered semi-implicit discontinuous Galerkin method for the two dimensional incompressible Navier-Stokes equations}


\author[1]{Maurizio Tavelli\fnref{label1}}
\author[2]{Michael Dumbser \corref{corr1} \fnref{label2}}
\address[1]{Department of Mathematics, University of Trento, \\ Via Sommarive 14, I-38050 Trento, Italy}
\address[2]{Laboratory of Applied Mathematics, Department of Civil, Environmental and Mechanical Engineering,
          					   University of Trento, Via Mesiano 77, I-38123 Trento, Italy}
												

\fntext[label1]{\tt m.tavelli@unitn.it (M.~Tavelli)}
\fntext[label2]{\tt michael.dumbser@unitn.it (M.~Dumbser)}


\begin{abstract}
In this paper we propose a new spatially high order accurate semi-implicit discontinuous Galerkin (DG) method for the solution of the two dimensional
incompressible Navier-Stokes equations on staggered unstructured curved meshes. While the discrete pressure is defined on the primal grid, the discrete
velocity vector field is defined on an edge-based dual grid. The flexibility of high order DG methods on curved unstructured meshes allows to discretize
even complex physical domains on rather coarse grids.

Formal substitution of the discrete momentum equation into the discrete continuity equation yields one sparse block four-diagonal linear equation system
for only one scalar unknown, namely the pressure. The method is computationally efficient, since the resulting system is not only very sparse but also
symmetric and positive definite for appropriate boundary conditions.
Furthermore, all the volume and surface integrals needed by the scheme presented in this paper depend only on the geometry
and the polynomial degree of the basis and test functions and can therefore be precomputed and stored in a preprocessor stage, which leads to savings in
terms of computational effort for the time evolution part. In this way also the extension to a fully curved isoparametric approach becomes natural
and affects only the preprocessing step. The method is validated for polynomial degrees up to $p=3$ by solving some typical numerical test problems and
comparing the numerical results with available analytical solutions or other numerical and experimental reference data.
\end{abstract}


\end{frontmatter}



\section{Introduction}
The main difficulty in the numerical solution of the incompressible Navier-Stokes equations lies in the pressure Poisson equation
and the associated linear equation system to be solved on the discrete level. This is closely related to the elliptic nature of these
equations, where boundary conditions affect instantly the solution everywhere inside the domain.

While finite difference schemes for the incompressible Navier-Stokes equations are well-established for several decades now
\cite{markerandcell,patankarspalding,patankar,vanKan}, as well as continuous finite element methods
\cite{TaylorHood,SUPG,SUPG2,Fortin,Verfuerth,Rannacher1,Rannacher3},
the development of high order discontinuous Galerkin (DG) finite element methods
for the incompressible Navier-Stokes equations is still a very active topic of ongoing research.

Several high order DG methods for the incompressible Navier-Stokes equations have been recently presented in literature, see for example
\cite{Bassi2007,Shahbazi2007,Ferrer2011,Nguyen2011,Rhebergen2012,Rhebergen2013,Crivellini2013,KleinKummerOberlack2013}, or the work of Bassi et al.
\cite{Bassi2006} based on the technique of artificial compressibility, originally introduced by Chorin in \cite{chorin1,chorin2}.


In this paper we propose a new, spatially high order accurate semi-implicit DG finite element scheme that is based on the general ideas of
\cite{DumbserCasulli,2DSIUSW}, following the philosophy of semi-implicit staggered finite difference schemes, which have been successfully
used in the past for the solution of the incompressible Navier-Stokes equations \cite{markerandcell,patankarspalding,patankar,vanKan} and the
free surface shallow water and Navier-Stokes equations, see
\cite{HirtNichols,CasulliCattani,CasulliCheng1992,CasulliWalters2000,WaltersCasulli1998,Casulli1999}.
Very recent developments in the field of such semi-implicit finite difference schemes for the free surface Navier-Stokes equations can
be found in \cite{Casulli2009,CasulliStelling2011,CasulliVOF}, together with their theoretical analysis presented in
\cite{BrugnanoCasulli,BrugnanoCasulli2,CasulliZanolli2012}.

In our semi-implicit staggered DG scheme, the discrete pressure is defined on the control volumes of the primal triangular mesh, while the discrete
velocity vector is defined on an edge-based, quadrilateral dual mesh. Thus, the usual orthogonality condition on the grid that applies to staggered
finite difference schemes which only use the edge-normal velocity component is not necessary here. The nonlinear convective terms are discretized
explicitly in time, using a classical RKDG scheme \cite{cbs4,CBS-convection-diffusion,CBS-convection-dominated} based on the local Lax-Friedrichs
(Rusanov) flux \cite{Rusanov:1961a}, while the viscous terms are discretized implicitly using a fractional step method. The DG discretization of the
viscous fluxes is based on the formulation of Gassner et al. \cite{MunzDiffusionFlux}, who obtained the viscous numerical flux from the
solution of the Generalized Riemann Problem (GRP) of the diffusion equation. The solution of the GRP has first been used to construct numerical
methods for hyperbolic conservation laws by Ben Artzi and Falcovich \cite{Artzi} and by Toro and Titarev \cite{toro4,titarevtoro}.
The discrete momentum equation is then inserted into the discrete continuity equation in order to
obtain the discrete form of the pressure Poisson equation. The chosen dual grid used here is taken as the one used in
\cite{Bermudez1998,USFORCE,StaggeredDG,2DSIUSW,Vazquez2014}, which leads to a sparse block four-diagonal system for the scalar pressure. Once the new
pressure field is known, the velocity vector field can subsequently be updated directly.  Very recently, an accurate and efficient pressure-based
hybrid finite volume / finite element solver using staggered unstructured meshes has been proposed in \cite{Vazquez2014}.


Other staggered DG schemes have been used in \cite{StaggeredDG,StaggeredDG2,StaggeredDG3,StaggeredDGCE1,StaggeredDGCE2,CentralDG1,CentralDG2}.
However, to our knowledge, none of these schemes has ever been applied to the incompressible Navier-Stokes equations. To our knowledge, a staggered
DG scheme has been proposed only for the Stokes system so far, see \cite{Kim2013}.
For alternative semi-implicit DG schemes on collocated grids see \cite{TumoloBonaventuraRestelli,GiraldoRestelli,Dolejsi1,Dolejsi2,Dolejsi3}.

The rest of the paper is organized as follows: in Section \ref{sec_1} the numerical method is described in detail, while in Secion \ref{sec.tests}
a set of numerical test problems is solved in order to study the accuracy of the presented approach. Some concluding remarks are given in Section
\ref{sec.concl}.

\section{DG scheme for the 2D incompressible Navier-Stokes equations}
\label{sec_1}
\subsection{Governing equations}
The two dimensional incompressible Navier-Stokes equations and the continuity equation are given by
\begin{eqnarray}
    \frac{\partial \vec{v}}{\partial t}+\nabla \cdot \mathbf{F}_c + \nabla p=\nu \Delta \vec{v} \label{eq:CS_2_2_0}, \\
    \nabla \cdot \vec{v}=0 \label{eq:CS_2},
\end{eqnarray}
where $p=P/\rho$ indicates the normalized fluid pressure; $P$ is the physical pressure and $\rho$ is the constant fluid density; $\nu$ is the kinematic viscosity  coefficient; $\vec{v}=(u,v)$ is the velocity vector, where $u$ and $v$ are the velocity components in the $x$ and $y$ direction, respectively;
$\mathbf{F}_c=\vec{v} \otimes \vec{v}$ is the flux tensor of the nonlinear convective terms, namely:
$$ \mathbf{F}_c=\left(\begin{array}{cc} uu & uv \\ vu & vv \end{array} \right). $$

The viscosity term is first written as $\nu \Delta \vec{v}=\nabla \cdot (\nu \nabla \vec{v})$ and then grouped with the nonlinear convective term.
So Eq. \eref{eq:CS_2_2_0} becomes
\begin{equation}
	\frac{\partial \vec{v}}{\partial t}+\nabla \cdot \mathbf{\TF} + \nabla p=0
\label{eq:CS_2_2},
\end{equation}
where $\mathbf{\TF}=\mathbf{\TF}(\vec{v},\nabla \vec{v})=\mathbf{F}_c(\vec{v})-\nu \nabla \vec{v}$ is a nonlinear tensor that depends on the velocity and its gradient,
see e.g. \cite{MunzDiffusionFlux,ADERNSE}. We further use the abbreviation $L(\vec v) = \diff{}{t}\vec{v}+\nabla \cdot \mathbf{\TF}$.

\subsection{Unstructured grid}
In this paper we use the same general unstructured staggered mesh proposed in \cite{2DSIUSW}. In this section we briefly summarize the grid construction and the main notation.
The computational domain is covered with a set of $N_i$ non-overlapping triangles $\TT_i$ with $i=1 \ldots N_i$. By denoting with $N_j$ the total number of edges, the $j-$th edge will be called $\Gamma_j$. $\B(\Omega)$ denotes the set of indices $j$ corresponding to boundary edges.
The three edges of each triangle $\TT_i$ constitute the set $S_i$ defined by $S_i=\{j \in [1,N_j] \,\, | \,\, \Gamma_j \mbox{ is an edge of }\TT_i \}$. For every $j\in [1\ldots N_j]-\B(\Omega)$ there exist two triangles $i_1$ and $i_2$ that share $\Gamma_j$. It is possible to assign arbitrarily a left and a right triangle called $\ell(j)$ and $r(j)$, respectively. The standard positive direction is assumed to be from left to right. Let $\nv_{j}$ denote the unit normal
vector defined on the edge $j$ and oriented with respect to the positive direction from left to right. For every triangular element $i$ and edge $j \in S_i$,
the neighbor triangle of element $\TT_i$ at edge $\Gamma_j$ is denoted by $\p(i,j)$.
\par For every $j\in [1, N_j]-\B(\Omega)$ the quadrilateral element associated to $j$ is called $\QQ_j$ and it is defined, in general, by the two centers of gravity of $\ell(j)$ and $r(j)$ and the two terminal nodes of $\Gamma_j$, see also \cite{Bermudez1998,USFORCE, 2DSIUSW}. We denote by $\TT_{i,j}=\QQ_j \cap \TT_i$ the intersection element for every $i$ and $j \in S_i$. Figure $\ref{fig.1}$ summarizes the notation used here, the main triangular and the dual quadrilateral meshes.
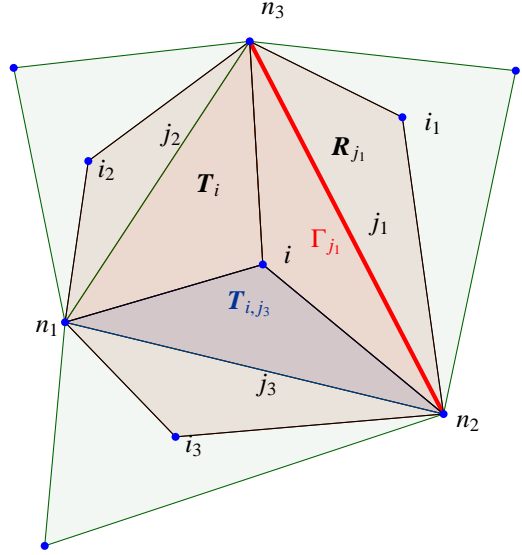
\begin{figure}[ht]
    \begin{center}
    \begin{tikzpicture}[line cap=round,line join=round,>=triangle 45,x=0.6373937677053826cm,y=0.6177884615384613cm]
\clip(2.11,-8.53) rectangle (16.23,3.95);
\fill[color=zzttqq,fill=zzttqq,fill opacity=0.1] (5.19,-3.03) -- (9,3) -- (13,-5) -- cycle;
\fill[color=qqwuqq,fill=qqwuqq,fill opacity=0.05] (9,3) -- (14.49,2.37) -- (13,-5) -- cycle;
\fill[color=qqwuqq,fill=qqwuqq,fill opacity=0.05] (9,3) -- (4.13,2.43) -- (5.19,-3.03) -- cycle;
\fill[color=qqwuqq,fill=qqwuqq,fill opacity=0.05] (5.19,-3.03) -- (4.77,-7.83) -- (13,-5) -- cycle;
\fill[color=zzttqq,fill=zzttqq,fill opacity=0.1] (9.27,-1.79) -- (9,3) -- (5.67,0.43) -- (5.19,-3.03) -- cycle;
\fill[color=zzttqq,fill=zzttqq,fill opacity=0.1] (9.27,-1.79) -- (9,3) -- (12.15,1.37) -- (13,-5) -- cycle;
\fill[color=zzttqq,fill=zzttqq,fill opacity=0.1] (5.19,-3.03) -- (7.47,-5.49) -- (13,-5) -- (9.27,-1.79) -- cycle;
\fill[color=qqttzz,fill=qqttzz,fill opacity=0.1] (13,-5) -- (5.19,-3.03) -- (9.27,-1.79) -- cycle;
\draw [color=zzttqq] (5.19,-3.03)-- (9,3);
\draw [color=zzttqq] (9,3)-- (13,-5);
\draw [color=zzttqq] (13,-5)-- (5.19,-3.03);
\draw [color=qqwuqq] (9,3)-- (14.49,2.37);
\draw [color=qqwuqq] (14.49,2.37)-- (13,-5);
\draw [color=qqwuqq] (13,-5)-- (9,3);
\draw [color=qqwuqq] (9,3)-- (4.13,2.43);
\draw [color=qqwuqq] (4.13,2.43)-- (5.19,-3.03);
\draw [color=qqwuqq] (5.19,-3.03)-- (9,3);
\draw [color=qqwuqq] (5.19,-3.03)-- (4.77,-7.83);
\draw [color=qqwuqq] (4.77,-7.83)-- (13,-5);
\draw [color=qqwuqq] (13,-5)-- (5.19,-3.03);
\draw (9.51,-1.19) node[anchor=north west] {$i$};
\draw (12.41,1.67) node[anchor=north west] {$i_1$};
\draw (5.67,0.67) node[anchor=north west] {$i_2$};
\draw (7.47,-5.25) node[anchor=north west] {$i_3$};
\draw (11.25,-0.47) node[anchor=north west] {$j_1$};
\draw (6.95,1.43) node[anchor=north west] {$j_2$};
\draw (8.93,-3.93) node[anchor=north west] {$j_3$};
\draw (4.39,-2.79) node[anchor=north west] {$n_1$};
\draw (13.07,-4.83) node[anchor=north west] {$n_2$};
\draw (9.05,4.03) node[anchor=north west] {$n_3$};
\draw (7.71,0.35) node[anchor=north west] {$\TT_i$};
\draw [color=zzttqq] (9.27,-1.79)-- (9,3);
\draw [color=zzttqq] (9,3)-- (5.67,0.43);
\draw [color=zzttqq] (5.67,0.43)-- (5.19,-3.03);
\draw [color=zzttqq] (5.19,-3.03)-- (9.27,-1.79);
\draw [color=zzttqq] (9.27,-1.79)-- (9,3);
\draw [color=zzttqq] (9,3)-- (12.15,1.37);
\draw [color=zzttqq] (12.15,1.37)-- (13,-5);
\draw [color=zzttqq] (13,-5)-- (9.27,-1.79);
\draw [color=zzttqq] (5.19,-3.03)-- (7.47,-5.49);
\draw [color=zzttqq] (7.47,-5.49)-- (13,-5);
\draw [color=zzttqq] (13,-5)-- (9.27,-1.79);
\draw [color=zzttqq] (9.27,-1.79)-- (5.19,-3.03);
\draw (10.49,1.15) node[anchor=north west] {$\QQ_{j_1}$};
\draw [color=ffqqqq](10.07,-0.83) node[anchor=north west] {$\Gamma_{j_1}$};
\draw [line width=1.6pt,color=ffqqqq] (9,3)-- (13,-5);
\draw [color=qqttzz] (13,-5)-- (5.19,-3.03);
\draw [color=qqttzz] (5.19,-3.03)-- (9.27,-1.79);
\draw [color=qqttzz] (9.27,-1.79)-- (13,-5);
\draw [color=qqttzz](8.35,-2.17) node[anchor=north west] {$\TT_{i,j_3}$};
\draw (5.19,-3.03)-- (5.67,0.43);
\draw (5.67,0.43)-- (9,3);
\draw (9,3)-- (9.27,-1.79);
\draw (9.27,-1.79)-- (5.19,-3.03);
\draw (9,3)-- (12.15,1.37);
\draw (12.15,1.37)-- (13,-5);
\draw (13,-5)-- (9.27,-1.79);
\draw (13,-5)-- (7.47,-5.49);
\draw (7.47,-5.49)-- (5.19,-3.03);
\begin{scriptsize}
\fill [color=qqqqff] (5.19,-3.03) circle (1.5pt);
\fill [color=qqqqff] (9,3) circle (1.5pt);
\fill [color=qqqqff] (13,-5) circle (1.5pt);
\fill [color=qqqqff] (9.27,-1.79) circle (1.5pt);
\fill [color=qqqqff] (14.49,2.37) circle (1.5pt);
\fill [color=qqqqff] (4.13,2.43) circle (1.5pt);
\fill [color=qqqqff] (4.77,-7.83) circle (1.5pt);
\fill [color=qqqqff] (12.15,1.37) circle (1.5pt);
\fill [color=qqqqff] (5.67,0.43) circle (1.5pt);
\fill [color=qqqqff] (7.47,-5.49) circle (1.5pt);
\end{scriptsize}
\end{tikzpicture} 
    \caption{Example of a triangular mesh element with its three neighbors and the associated staggered edge-based dual control volumes, together with the notation
    used throughout the paper.}
    \label{fig.1}
		\end{center}
\end{figure}
According to \cite{2DSIUSW}, we often call the mesh of triangular elements $\{\TT_i \}_{i \in [1, N_i]}$ the \textit{main grid} or the \textit{primal grid} and  the quadrilateral grid $\{\QQ_j \}_{j \in [1, N_j]}$ is termed the \textit{dual grid}.

On the dual grid we define the same quantities as for the main grid, briefly: $N_l$ is the total amount of edges of $\QQ_j$; $\Gamma_l$ indicates the $l$-th edge; $\forall j$, the set of edges $l$ of $j$ is indicated with $S_j$; $\forall l$, $\ell_{jl}(l)$ and $r_{jl}(l)$ are the left and the right quadrilateral element, respectively; $\forall l$, $\nv_{l}$ is the standard normal vector defined on $l$ and assumed positive with respect to the standard orientation on $l$ (defined, as above, from the left to the right).
Finally, each triangle $\TT_i$ is defined starting from an arbitrary node and oriented in counter-clockwise direction. Similarly, each quadrilateral element $\QQ_j$ is defined
starting from $\ell(j)$ and oriented in counter-clockwise direction.


\subsection{Basis functions}
\label{sec_bf}
According to \cite{2DSIUSW} we proceed as follows: we first construct the polynomial basis up to a generic polynomial degree $p$ on some reference triangular and quadrilateral elements. In order to do this we take $T_{std}=\{(\xi,\gamma) \in \R^{2,+} \,\, | \,\, \gamma\leq1-\xi \vee 0 \leq \xi \leq 1 \}$ as the reference triangle and the unit square as the reference quadrilateral element $R_{std}=[0,1]^2$. Using the standard nodal approach of conforming continuous finite elements,  we obtain $N_\phi=\frac{(p+1)(p+2)}{2}$ basis functions $\{\phi_k \}_{k \in [1,N_\phi]}$ on $T_{std}$ and $N_{\psi}=(p+1)^2$ basis functions on $R_{std}$.
The connection between reference and physical space is performed by the maps $\TT_i \stackrel{T_i}{\longrightarrow} T_{std}$ for every $i =1 \ldots N_i$; $\QQ_j \stackrel{T_j}{\longrightarrow}R_{std}$ for every $j =1 \ldots N_j$ and its inverse, called $\TT_i \stackrel{T_i^{-1}}{\longleftarrow} T_{std}$ and $\QQ_j \stackrel{T_j^{-1}}{\longleftarrow}R_{std}$, respectively. The maps from physical coordinates to reference coordinates can be constructed following a classical
sub-parametric or a complete iso-parametric approach.

\subsection{Semi-Implicit DG scheme}
\label{sec_semi_imp_dg}
We define the spaces of piecewise polynomials used on the main grid and the dual grid as follows,
\begin{equation}
	V_h^m=\{ \phi \,\, : \,\, \phi |_{\TT_i} \in \mathds{P}^p(\TT_i), \forall i\in [1, N_i]  \}, \qquad \textnormal{ and } \qquad
	V_h^d=\{ \psi \,\, : \,\, \psi |_{\QQ_j} \in \mathds{Q}^p(\QQ_j), \forall j\in [1, N_j]-\B(\Omega)  \},
\label{eq:rev1}
\end{equation}
where $\mathds{P}^p(\TT_i)$ is the space of polynomials of degree at most $p$ on $\TT_i$, while $\mathds{Q}^p(\QQ_j)$ is the space of
tensor products of one-dimensional polynomials of degree at most $p$ on $\QQ_j$.

The discrete pressure $p_h$ is defined on the main grid while the discrete
velocity vector field $\vec{v}_h$ is defined on the dual grid, namely $p_h \in V_h^m$ and $\vec{v}_h \in V_h^d$ for each component of the velocity vector. \par

The numerical solution of \eref{eq:CS_2}-\eref{eq:CS_2_2} is represented by piecewise polynomials and written in terms of the basis functions on the primary and the dual grid as
\begin{equation}
	p_i(x,y,t)=\sum\limits_{l=1}^{N_\phi} \phi_l^{(i)}(x,y)\hat{p}_{l,i}(t)=:\bphi^{(i)}(x,y)\etah_i(t),
\label{eq:D_1}
\end{equation}
\begin{equation}
	\vec{v}_j(x,y,t)=\sum\limits_{l=1}^{N_\psi} \psi_l^{(j)}(x,y) \hat{\vec{v}}_{l,j}(t)=:\bpsi^{(j)}(x,y)\vvh_j(t),
\label{eq:D_3}
\end{equation}
where the vector of basis functions $\bphi(x,y)$ and $\bpsi(x,y)$ are generated from $\bphi(\xi,\gamma)$ on $\bpsi(\xi,\gamma)$ on $R_{std}$, respectively. Formally $\bphi^{(i)}(x,y)=\bphi(T_i(x,y))$ for $i=1 \ldots N_i$ and $\bpsi^{(j)}(x,y)=\bpsi(T_j(x,y))$ for every $j=1 \ldots N_j$.

\begin{figure}[!htbp]
    \begin{center}
    \begin{tikzpicture}[line cap=round,line join=round,>=triangle 45,x=0.36630036630036633cm,y=0.3386004514672688cm]
\clip(-2.22,-4.94) rectangle (14.16,3.92);
\fill[color=qqttzz,fill=qqttzz,fill opacity=0.1] (2.5,-0.42) -- (5.42,3.74) -- (10.46,-3.06) -- (5.54,-4.86) -- cycle;
\fill[color=zzttqq,fill=zzttqq,fill opacity=0.1] (5.54,-4.86) -- (13.96,-3.58) -- (5.42,3.74) -- cycle;
\fill[color=ttttqq,fill=ttttqq,fill opacity=0.1] (5.42,3.74) -- (-1.82,-0.36) -- (5.54,-4.86) -- cycle;
\fill[color=zzttqq,fill=zzttqq,fill opacity=0.5] (5.54,-4.86) -- (10.46,-3.06) -- (5.42,3.74) -- cycle;
\fill[color=ttttqq,fill=ttttqq,fill opacity=0.5] (5.42,3.74) -- (2.5,-0.42) -- (5.54,-4.86) -- cycle;
\draw [color=qqttzz] (2.5,-0.42)-- (5.42,3.74);
\draw [color=qqttzz] (5.42,3.74)-- (10.46,-3.06);
\draw [color=qqttzz] (10.46,-3.06)-- (5.54,-4.86);
\draw [color=qqttzz] (5.54,-4.86)-- (2.5,-0.42);
\draw (5.54,-4.86)-- (5.42,3.74);
\draw (0.44,0.74) node[anchor=north west] {$\ell(j)$};
\draw (11.12,-2.7) node[anchor=north west] {$r(j)$};
\draw (6.26,4.22) node[anchor=north west] {$j$};
\draw [color=zzttqq] (5.54,-4.86)-- (13.96,-3.58);
\draw [color=zzttqq] (13.96,-3.58)-- (5.42,3.74);
\draw [color=zzttqq] (5.42,3.74)-- (5.54,-4.86);
\draw [color=ttttqq] (5.42,3.74)-- (-1.82,-0.36);
\draw [color=ttttqq] (-1.82,-0.36)-- (5.54,-4.86);
\draw [color=ttttqq] (5.54,-4.86)-- (5.42,3.74);
\draw (9.78,-0.84) node[anchor=north west] {$\eta_{r(j)}$};
\draw (1.18,-0.82) node[anchor=north west] {$\eta_{\ell(j)}$};
\draw [color=zzttqq] (5.54,-4.86)-- (10.46,-3.06);
\draw [color=zzttqq] (10.46,-3.06)-- (5.42,3.74);
\draw [color=zzttqq] (5.42,3.74)-- (5.54,-4.86);
\draw [color=ttttqq] (5.42,3.74)-- (2.5,-0.42);
\draw [color=ttttqq] (2.5,-0.42)-- (5.54,-4.86);
\draw [color=ttttqq] (5.54,-4.86)-- (5.42,3.74);
\draw [line width=3.6pt,dash pattern=on 1pt off 1pt,color=ffqqqq] (5.42,3.74)-- (5.54,-4.86);
\begin{scriptsize}
\fill [color=qqqqff] (2.5,-0.42) circle (1.5pt);
\fill [color=qqqqff] (5.42,3.74) circle (1.5pt);
\fill [color=qqqqff] (10.46,-3.06) circle (1.5pt);
\fill [color=qqqqff] (5.54,-4.86) circle (1.5pt);
\end{scriptsize}
\end{tikzpicture} \quad
    \begin{tikzpicture}[line cap=round,line join=round,>=triangle 45,x=0.37149166187873495cm,y=0.3423930029581111cm]
\clip(2.89,-6.84) rectangle (16.35,3.67);
\fill[color=zzttqq,fill=zzttqq,fill opacity=0.1] (3.82,-4.82) -- (7.48,3.06) -- (16,-4) -- cycle;
\fill[color=qqwwtt,fill=qqwwtt,fill opacity=0.2] (7.48,3.06) -- (8.54,-1.5) -- (16,-4) -- (15.16,0.18) -- cycle;
\fill[color=qqttzz,fill=qqttzz,fill opacity=0.1] (7.48,3.06) -- (8.54,-1.5) -- (3.82,-4.82) -- (4.9,2.32) -- cycle;
\fill[color=ubqqys,fill=ubqqys,fill opacity=0.1] (3.82,-4.82) -- (8.54,-1.5) -- (16,-4) -- (7.24,-6.18) -- cycle;
\fill[color=qqwwtt,fill=qqwwtt,fill opacity=0.5] (7.48,3.06) -- (8.54,-1.5) -- (16,-4) -- cycle;
\fill[color=qqttzz,fill=qqttzz,fill opacity=0.5] (7.48,3.06) -- (3.82,-4.82) -- (8.54,-1.5) -- cycle;
\fill[color=ubqqys,fill=ubqqys,fill opacity=0.5] (3.82,-4.82) -- (16,-4) -- (8.54,-1.5) -- cycle;
\draw [color=zzttqq] (3.82,-4.82)-- (7.48,3.06);
\draw [color=zzttqq] (7.48,3.06)-- (16,-4);
\draw [color=zzttqq] (16,-4)-- (3.82,-4.82);
\draw (7.48,3.06)-- (8.54,-1.5);
\draw (8.54,-1.5)-- (3.82,-4.82);
\draw (8.54,-1.5)-- (16,-4);
\draw [color=qqwwtt] (7.48,3.06)-- (8.54,-1.5);
\draw [color=qqwwtt] (8.54,-1.5)-- (16,-4);
\draw [color=qqwwtt] (16,-4)-- (15.16,0.18);
\draw [color=qqwwtt] (15.16,0.18)-- (7.48,3.06);
\draw [color=qqttzz] (7.48,3.06)-- (8.54,-1.5);
\draw [color=qqttzz] (8.54,-1.5)-- (3.82,-4.82);
\draw [color=qqttzz] (3.82,-4.82)-- (4.9,2.32);
\draw [color=qqttzz] (4.9,2.32)-- (7.48,3.06);
\draw [color=ubqqys] (3.82,-4.82)-- (8.54,-1.5);
\draw [color=ubqqys] (8.54,-1.5)-- (16,-4);
\draw [color=ubqqys] (16,-4)-- (7.24,-6.18);
\draw [color=ubqqys] (7.24,-6.18)-- (3.82,-4.82);
\draw [line width=3.6pt,dash pattern=on 1pt off 1pt,color=ffqqqq] (7.48,3.06)-- (8.54,-1.5);
\draw [line width=3.6pt,dash pattern=on 1pt off 1pt,color=ffqqqq] (8.54,-1.5)-- (3.82,-4.82);
\draw [line width=3.6pt,dash pattern=on 1pt off 1pt,color=ffqqqq] (8.54,-1.5)-- (16,-4);
\draw (8.88,-0.23) node[anchor=north west] {$i$};
\draw [color=qqwwtt] (7.48,3.06)-- (8.54,-1.5);
\draw [color=qqwwtt] (8.54,-1.5)-- (16,-4);
\draw [color=qqwwtt] (16,-4)-- (7.48,3.06);
\draw (11.95,1.42) node[anchor=north west] {$\TT_ {i,j_1}$};
\draw (3.48,1.81) node[anchor=north west] {$\TT_ {i,j_2}$};
\draw (8.03,-4.49) node[anchor=north west] {$\TT_ {i,j_3}$};
\draw [color=qqttzz] (7.48,3.06)-- (3.82,-4.82);
\draw [color=qqttzz] (3.82,-4.82)-- (8.54,-1.5);
\draw [color=qqttzz] (8.54,-1.5)-- (7.48,3.06);
\draw [color=ubqqys] (3.82,-4.82)-- (16,-4);
\draw [color=ubqqys] (16,-4)-- (8.54,-1.5);
\draw [color=ubqqys] (8.54,-1.5)-- (3.82,-4.82);
\begin{scriptsize}
\fill [color=qqqqff] (3.82,-4.82) circle (1.5pt);
\fill [color=qqqqff] (7.48,3.06) circle (1.5pt);
\fill [color=qqqqff] (16,-4) circle (1.5pt);
\fill [color=qqqqff] (8.54,-1.5) circle (1.5pt);
\end{scriptsize}
\end{tikzpicture}
    \caption{Jumps of $\np$ on the main grid (left) and of $\vec v$ on the dual grid (right)}
    \label{fig.2}
    \end{center}
\end{figure}
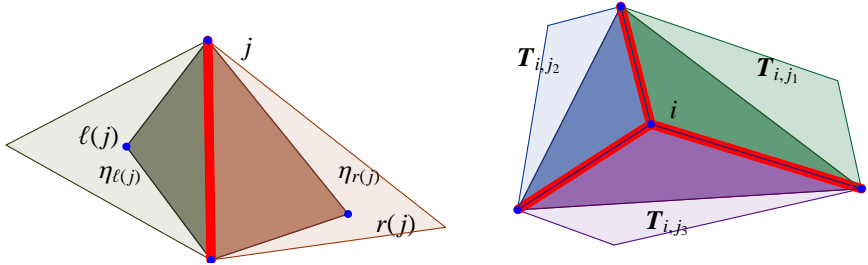

A weak formulation of equation \eref{eq:CS_2} is obtained by multiplying it by $\bphi$ and integrating over a control volume
$\TT_i$, for every $k=1\ldots N_\phi$. The resulting weak formulation of \eref{eq:CS_2} reads
\begin{equation}
\int\limits_{\TT_i}{\phi_k^{(i)} \nabla \cdot \vec{v} dx dy}=0.
\label{eq:CS_4}
\end{equation}
Similarly, multiplication of the momentum equation \eref{eq:CS_2_2} by $\bpsi$ and integrating over a control volume $\QQ_j$ one obtains, componentwise,
\begin{equation}
\int\limits_{\QQ_j}{\psi_k^{(j)}\left( \diff{\vec{v}}{t}+\nabla \cdot \mathbf{\TF} \right)  dx dy}+\int\limits_{\QQ_j}{\psi_k^{(j)} \nabla \np \, dx dy}=0,
\label{eq:CS_5}
\end{equation}
for every $j=1 \ldots N_j$ and $k=1 \ldots N_\psi$. Using integration by parts Eq. \eref{eq:CS_4} yields
\begin{equation}
\oint\limits_{\partial \TT_i}{\phi_k^{(i)} \vec{v} \cdot \nv_{i} \, ds}-\int\limits_{\TT_i}{\nabla \phi_k^{(i)} \cdot \vec{v} \, dx dy}  =0,
\label{eq:CS_6}
\end{equation}
where $\nv_{i}$ indicates the outward pointing unit normal vector.
The discrete pressure $p_h$ in general presents a discontinuity on $\Gamma_j$ and also the discrete velocity field $\vec{v}_h$
jumps on the edges of $\QQ_j$ (see Figure $\ref{fig.2}$). Hence, equations \eref{eq:CS_5} and \eref{eq:CS_6} have to be split as follows:

\begin{equation}
\sum\limits_{j \in S_i}\left( \int\limits_{\Gamma_j}{\phi_k^{(i)} \vec{v}_j \cdot \nextud \, ds}-\int\limits_{\TT_{i,j}}{\nabla \phi_k^{(i)} \cdot \vec{v}_j \, dx dy}  \right)=0,
\label{eq:CS_8}
\end{equation}
and
\begin{eqnarray}
\int\limits_{\QQ_j}{\psi_k^{(j)}\left( \diff{\vec{v_j}}{t}+\nabla \cdot \mathbf{\TF}_j \right)  dx dy}
+\hspace{-3mm} \int\limits_{\TT_{\ell(j),j}}{\psi_k^{(j)} \nabla \np_{\ell(j)} dx dy} \nonumber
+\hspace{-3mm} \int\limits_{\TT_{r(j),j}}{\psi_k^{(j)} \nabla \np_{r(j)} \, dx dy} + \nonumber 
\int\limits_{\Gamma_j}{\psi_k^{(j)} \left(\np_{r(j)}-\np_{\ell(j)}\right) \nstd \, ds}=0,
\label{eq:CS_9}
\end{eqnarray}
where $\nextud=\nv_{i}|_{\Gamma_j}$.
Definitions \eref{eq:D_1} and \eref{eq:D_3} allow to rewrite the above equations by splitting the spatial and temporal variables, namely
\begin{eqnarray}
\sum\limits_{j \in S_i}\left(\int\limits_{\Gamma_j}{\phi_k^{(i)}\psi_l^{(j)} \nextud ds}\, \vec{\hat{v}}_{l,j}-\int\limits_{\TT_{i,j}}{\nabla \phi_k^{(i)}\psi_l^{(j)}dx dy} \, \vec{\hat{v}}_{l,j} \right)=0,
\label{eq:CS_10}
\end{eqnarray}
and
\begin{eqnarray}
\int\limits_{\QQ_j}{\psi_k^{(j)} \psi_l^{(j)}dx dy} L_h(\vec{\hat{v}}_{l,j})
+\int\limits_{\TT_{\ell(j),j}}{\psi_k^{(j)} \nabla \phi_{l}^{(\ell(j))}  dx dy} \,  \, \hat \np_{l,\ell(j)}
+\int\limits_{\TT_{r(j),j}}{   \psi_k^{(j)} \nabla \phi_{l}^{(r(j))}     dx dy} \,  \, \hat \np_{l,r(j)}    \nonumber \\
 +\int\limits_{\Gamma_j}{\psi_k^{(j)} \phi_{l}^{(r(j))}    \nstd ds} \,  \hat \np_{l,r(j)}
 -\int\limits_{\Gamma_j}{\psi_k^{(j)} \phi_{l}^{(\ell(j))} \nstd ds} \,  \hat \np_{l,\ell(j)}=0, \nonumber \\
\label{eq:CS_11}
\end{eqnarray}
where we have used the standard summation convention for the repeated index $l$. $L_h$ is an appropriate discretization of the operator $L$ and will
be given later. For every $i$ and $j$, Eqs. \eref{eq:CS_10}-\eref{eq:CS_11} are written in a compact matrix form such as
\begin{eqnarray}
    \sum\limits_{j \in S_i}\D_{i,j}\vvh_j=0 \label{eq:CS_12},
\end{eqnarray}
and
\begin{eqnarray}
    \Mpsi_j L_h(\vvh_j) + \RM_j \etah_{r(j)}- \LM_j \etah_{\ell(j)} =0, \label{eq:CS_12_1}
\end{eqnarray}
respectively, where:
\begin{equation}
	\Mpsi_j=\int\limits_{\QQ_j}{\psi_k^{(j)}\psi_l^{(j)}  dx dy},
\label{eq:MD_2}
\end{equation}

\begin{equation}
	\D_{i,j}=\int\limits_{\Gamma_j}{\phi_k^{(i)}\psi_l^{(j)}\nextud ds}-\int\limits_{\TT_{i,j}}{\nabla \phi_k^{(i)}\psi_l^{(j)}dx dy},
\label{eq:MD_3}
\end{equation}

\begin{equation}
	\RM_{j}=\int\limits_{\Gamma_j}{\psi_k^{(j)} \phi_{l}^{(r(j))}\nstd ds}+\int\limits_{\TT_{r(j),j}}{\psi_k^{(j)} \nabla \phi_{l}^{(r(j))}  dx dy},
\label{eq:MD_4}
\end{equation}

\begin{equation}
	\LM_{j}=\int\limits_{\Gamma_j}{\psi_k^{(j)} \phi_{l}^{(\ell(j))}\nstd ds}-\int\limits_{\TT_{\ell(j),j}}{\psi_k^{(j)} \nabla \phi_{l}^{(\ell(j))}  dx dy}.
\label{eq:MD_5}
\end{equation}
\par According to \cite{2DSIUSW} the action of tensors $\LM$ and $\RM$ can be generalized by introducing the new tensor $\Q_{i,j}$, defined as
\begin{equation}
	\Q_{i,j}=\int\limits_{\TT_{i,j}}{\psi_k^{(j)} \nabla \phi_{l}^{(i)}  dx dy}-\int\limits_{\Gamma_j}{\psi_k^{(j)} \phi_{l}^{(i)}\sigma_{i,j} \nstd ds},
\label{eq:MD_6}
\end{equation}
where $\sigma_{i,j}$ is a sign function defined by
\begin{equation}
	\sigma_{i,j}=\frac{r(j)-2i+\ell(j)}{r(j)-\ell(j)}.
\label{eq:SD_1}
\end{equation}
In this way $\Q_{\ell(j),j}=-\LM_j$ and $\Q_{r(j),j}=\RM_j$, and then Eq. \eref{eq:CS_12_1} becomes in terms of $\Q$
\begin{equation}
	\Mpsi_j L_h(\vvh_j) + \Q_{r(j),j}  \etah_{r(j)}+ \Q_{\ell(j),j} \etah_{\ell(j)} =0,
\label{eq:CS_12_2}
\end{equation}
or, equivalently,
\begin{equation}
\Mpsi_j L_h(\vvh_j) + \Q_{i,j} \etah_{i}+ \Q_{\p(i,j),j} \etah_{\p(i,j)} =0.
\label{eq:CS_13}
\end{equation}

We discretize the velocity in  Eq. \eref{eq:CS_12} implicitly and the pressure in Eq. \eref{eq:CS_12_1} semi-implicitly by using the theta method in time, namely
\begin{equation}
\left\{
\begin{array}{l}
	\sum\limits_{j \in S_i}\D_{i,j}\vvh_j ^{{n+1}}=0,	\\
	\Mpsi_j	\frac{\vvh_j^{n+1}-\Fvh_j^{n}}{\Delta t} + \Q_{r(j),j} \etah_{r(j)}^{{n+\theta}}+ \Q_{\ell(j),j} \etah_{\ell(j)} ^{{ n+\theta}} =0,
\end{array}
\right.
\label{eq:CS_14}
\end{equation}
where $\etah^{{ n+\theta}}=\theta \etah^{{ n+1}}+(1-\theta)\etah^{{ n}}$; and $\theta$ is an implicitness factor to be taken in the range $\theta \in [\frac{1}{2},1]$, see e.g. \cite{CasulliCattani}.
Discretizing Eqs. \eref{eq:CS_14} as described above and using the formulation \eref{eq:CS_13}, we get for every $i$ and $j \in S_i$
\begin{equation}
	\sum\limits_{j \in S_i}\D_{i,j}\vvh_j^{n+1}=0,
\label{eq:CS_15}
\end{equation}
\begin{equation}
	\Mpsi_j	\frac{\vvh_j^{n+1}-\Fvh_j^{n}}{\Delta t}+ \theta \left( \Q_{i,j} \etah_{i}^{n+1}+  \Q_{\p(i,j),j} \etah_{\p(i,j)}^{n+1} \right) 
+(1-\theta)\left( \Q_{i,j} \etah_{i}^{n}+  \Q_{\p(i,j),j} \etah_{\p(i,j)}^{n} \right)=0,
\label{eq:CS_16}
\end{equation}
where $\Fvh_j^{n}$ is an appropriate discretization of the nonlinear convective and viscous terms.
The details for the computation of $\Fvh_j^{n}$  will be presented later in Section $\ref{sub_sec_nonlinearconv}$.
Formal substitution of the momentum equation \eref{eq:CS_16} into the continuity equation \eref{eq:CS_15}, see also \cite{CasulliCheng1992,DumbserCasulli}, yields
\begin{equation}
- \theta \Delta t \sum\limits_{j\in S_i}\D_{i,j}\Mpsi_j^{-1}\Q_{i,j} \etah_i^{n+1} 
- \theta \Delta t \sum\limits_{j\in S_i} \D_{i,j}\Mpsi_j ^{-1}\Q_{\p(i,j),j} \etah_{\p(i,j)}^{n+1}=\mathbf{b}_i^n,
\label{eq:CS_19}
\end{equation}
where
\begin{equation}
	\mathbf{b}_i^n = - \sum\limits_{j \in S_i} \D_{i,j} \Fvh_j^n  
     +(1-\theta) \Delta t \sum\limits_{j \in S_i} \D_{i,j} \left(\Mpsi_j \right)^{-1} \left( \Q_{i,j} \etah_{i}^{n}+ \Q_{\p(i,j),j} \etah^{n}_{\p(i,j),j} \right), 
\label{eq:CS_20}
\end{equation}
groups all the known terms at time $t^n$.

Eq. \eref{eq:CS_19} represents a block four-diagonal system for the new pressure $\etah_i^{n+1}$. It can be interpreted as the discrete form of the pressure
Poisson equation of the incompressible Navier-Stokes equations. Once the new pressure field is known, the velocity field can be readily updated from the
momentum equation, Eq. \eref{eq:CS_16}. We emphasize that in the present algorithm, the only unknown is the \textit{scalar} pressure $p_h$.

It remains to complete the system by introducing the boundary conditions. In order to do this observe how, for $i \in [1, N_i]$ and $j \in S_i \cap \B(\Omega)$,  the boundary element $\QQ_j=\TT_{i,j}$ is a \textit{triangular element} and \textit{not} a quadrilateral element. The basis functions to be used are the one generated on $T_{std}$. In  this way the matrices $\Mpsi_j, \D_{i,j}, \Q_{i,j}$ defined in \eref{eq:MD_2}, \eref{eq:MD_3} and \eref{eq:MD_6}, have to be modified for boundary elements.
\newline
For every $j \in S_i \cap \B(\Omega)$
\begin{eqnarray}
	\vec{v}_j&=&\sum\limits_{l=1}^{N_\phi}{\phi}_l \hat{\vec{v}}_{l,j},
\label{eq:70}
\end{eqnarray}
where the $\phi_l$ are the basis functions on the reference triangle $T_{std}$.
The matrices can be recomputed for $j \in S_i \cap \B(\Omega)$ and will be called $\D_{i,j}^\partial, \Q_{i,j}^\partial$.

Equations \eref{eq:CS_19}-\eref{eq:CS_20} are consequently computed with the triangular boundary elements and one so obtains
\begin{eqnarray}
\theta \Delta t \left[- \sum\limits_{j\in S_i\cap \B(\Omega)}\D_{i,j}^{\partial}\Mpsi_j^{-1}\Q_{i,j}^{\partial} -\sum\limits_{j\in S_i-\B(\Omega)}\D_{i,j}\Mpsi_j^{-1}\Q_{i,j} \right] \etah_i^{n+1} 
-\theta \Delta t \sum\limits_{j\in S_i-\B(\Omega)} \D_{i,j}\Mpsi_j^{-1}\Q_{\p(i,j),j} \etah_{\p(i,j)}^{n+1}=\tilde{\mathbf{b}}_i^n,
\label{eq:82}
\end{eqnarray}
where now the vector of known terms is
\begin{eqnarray}
	\tilde{\mathbf{b}}_i^n & = &- \sum\limits_{j \in S_i-\B(\Omega)} \D_{i,j} \Fvh_j^n
  + \sum\limits_{j \in S_i\cap\B(\Omega)} \D_{i,j}^{\partial} \Fvh_j^n \nonumber \\
&& +(1-\theta) \Delta t \sum\limits_{j\in S_i\cap \B(\Omega)}\D_{i,j}^{\partial}\Mpsi_j^{-1}\Q_{i,j}^{\partial}\etah_{i}^{n}  
+(1-\theta) \Delta t \sum\limits_{ j \in S_i-\B(\Omega)} \D_{i,j} \Mpsi_j^{-1} \left( \Q_{i,j} \etah_{i}^{n}+ \Q_{\p(i,j),j} \etah_{\p(i,j)}^{n} \right). \nonumber \\
\label{eq:83}
\end{eqnarray}

As implied by Eq. \eref{eq:82}, the stencil of the present scheme only involves the $i-$th element and its direct Neumann neighbors.
Thus, since $ \#S_i=3$, the system described by \eref{eq:82} is a block-four-diagonal one. As we will show later, the system is symmetric and positive definite for
appropriate boundary conditions, hence it can be efficiently solved by using a matrix-free implementation of the conjugate gradient algorithm \cite{cgmethod}.
Once the new pressure has been computed, the new velocity field can be readily updated from Eq. \eref{eq:CS_16} for every $j \notin \B(\Omega)$:
\begin{equation}
\vvh_j^{n+1}= \Fvh_j^n-\theta \Delta t \Mpsi_j^{-1} \left( \Q_{i,j} \etah_{i}^{n+1}+ \Q_{\p(i,j),j} \etah_{\p(i,j)}^{n+1} \right) 
-(1-\theta) \Delta t \Mpsi_j ^{-1} \left( \Q_{i,j} \etah_{i}^{n}+ \Q_{\p(i,j),j} \etah_{\p(i,j)}^{n} \right). 
\label{eq:83_2}
\end{equation}
The above equations \eref{eq:82},\eref{eq:83} and \eref{eq:83_2} can be modified for $j \in \B(\Omega)$ according to the type of boundary conditions (velocity
or pressure boundary condition).
Note that all the matrices used in the above algorithm can be precomputed once and forall for a given mesh and polynomial degree $p$.

\subsection{Nonlinear convection-diffusion}
\label{sub_sec_nonlinearconv}
In problems where the convective term and the viscosity can be neglected we can take $\Fvh_j^n=\vvh_j^n$ in Eq. \eref{eq:CS_20}. Otherwise, an explicit cell-centered RKDG method \cite{cbs4} on the dual mesh is used in this paper for the discretization of the nonlinear convective terms. The viscosity contribution is  discretized implicitly using a fractional step method, in order to avoid additional restrictions on the time step $\Delta t$.
The semi-discrete DG scheme for the nonlinear convection-diffusion terms on the dual mesh is given by
\begin{equation}
\int\limits_{\QQ_j} \psi_k \frac{d}{dt} \vec{v}_h \, dx dy + \int\limits_{\partial \QQ_j}{\psi_k \mathbf{G}_h \cdot \vec{n} \, \,  ds} - \int\limits_{\QQ_j}{\nabla \psi_k \cdot \mathbf{F}(\vec{v}_h,\nabla \vec{v}_h) dx dy}  = 0,
\label{eq:59}
\end{equation}
\noindent and the numerical flux for both, the convective and the viscous contribution, is given by \cite{Rusanov:1961a,MunzDiffusionFlux,ADERNSE} as
\begin{equation}
	\mathbf{G}_h \cdot \vec{n} = \frac{1}{2}\left(\mathbf{F}(\vec{v}_h^{\,+},\nabla \vec{v}_h^{\,+}) + \mathbf{F}(\vec{v}_h^{\,-},\nabla \vec{v}_h^{\,-}) \right)\cdot \vec{n} -\frac{1}{2}s_{\max} \left( \vec{v}_h^{\,+} - \vec{v}_h^{\,-} \right),
\label{eq:61}
\end{equation}
%
\noindent with
\begin{equation}
s_{\max} = 2 \, \max( |\vec v_h^{\,-} \cdot \vec n|, |\vec v_h^{\,+} \cdot \vec n| ) + \frac{2 \nu}{h^+ +h^-}\frac{2p+1}{\sqrt{\frac{\pi}{2}}},
\end{equation}
which contains the maximum eigenvalue of the Jacobian matrix of the purely convective transport operator $\mathbf{F}_c$ in normal direction,
see \cite{DumbserCasulli}, and the
stabilization term for the viscous flux, see \cite{ADERNSE,MunzDiffusionFlux}. Furthermore, the $\vec{v}_h^\pm$ and $\nabla \vec{v}_h^\pm$ denote the
velocity vectors and their gradients, extrapolated to the boundary of $\QQ_j$ from within the element $\QQ_j$ and from the neighbor element, respectively.
$h^+$ and $h^-$ are the maximum radii of the inscribed circle in $\QQ_j$ and the neighbor element, respectively.
A classical third order accurate TVD Runge-Kutta method is used for time integration of the nonlinear convective terms, see e.g. \cite{shuosher1,shu2,cbs4}, since
the explicit discretization of higher order DG schemes with a simple first order Euler method in time would lead to a linearly unstable scheme.
The above method requires that the time step size is restricted by a CFL-type restriction for DG schemes, namely:
\begin{equation}
    \Delta t = \frac{\textnormal{CFL}}{2p+1}\cdot \frac{h_{min}}{2|\vec{v}_{max}|},
\label{eq:CFLC}
\end{equation}
where $h_{min}$ is the smallest incircle diameter; $\textnormal{CFL}<0.5$; and $\vec{v}_{max}$ is the maximum convective speed.
Furthermore, the time step of the global semi-implicit scheme is \textit{not} affected by the local time step used for the time integration of the
convective terms if a local time stepping / subcycling approach is employed, see \cite{CasulliZanolli,TavelliDumbserCasulli}.

Implicit discretization of the viscous contribution $\nabla \vec{v}$ in \eref{eq:59} with a fractional step method involves a block five-diagonal system that can be efficiently
solved using the GMRES algorithm \cite{GMRES}. The solution of this system is not necessary in problems where the viscous terms are small and can be integrated
explicitly in time.
The stability of the method is linked to the nonlinear convective term, so the method is stable under condition \eref{eq:CFLC}.

\subsection{Extension to curved elements}
\label{sec_2.2}
The maps used to switch between reference and physical space can be defined using a simple sub-parametric vertex based approach or a fully isoparametric approach. In the first case only the vertices of the elements $\TT_i$ and $\QQ_j$ are required to map the physical element into the reference one and vice versa. In this simple case an explicit expression for the maps $T_i$, $T_i^{-1}$ and $T_j^{-1}$ can be computed while for the map $T_j$ we use the Newton method (see e.g. \cite{2DSIUSW}). A simple extension to the complete isoparametric case requires to store more information about each element, namely we need to know the coordinates
of the nodes $\{(X,Y)_k^i\}_{k=1,N_\phi}$ for each triangular element $i$ and $\{(X,Y)_l^j\}_{l=1,N_\psi}$ for each quadrilateral element $j$. The inverse maps $T_i^{-1}$ and $T_j^{-1}$ are defined by using the same basis functions $\phi_k$ and $\psi_k$ used for representing the discrete solution of the PDE, i.e. we have
\begin{equation}
x=\sum \limits_k^{N_\phi} \, \phi_k \, X^i_k, \qquad
y=\sum \limits_k^{N_\phi} \, \phi_k \, Y^i_k,
\label{eqn.iso.tri}
\end{equation}
and
\begin{equation}
x=\sum \limits_k^{N_\psi} \, \psi_k \, X^j_k, \qquad
y=\sum \limits_k^{N_\psi} \, \psi_k \, Y^j_k,
\label{eqn.iso.quad}
\end{equation}
for triangles and quadrilateral elements, respectively. In this case the maps $T_i$ and $T_j$ become nonlinear and so the Newton method has to be used for both.  Also the Jacobian and the normal
vectors are not, in general, constant through the element and the edges, respectively. The main advantage of this approach is that now the edges become curved and so the computational domain can
better approximate the physical one. It is important to observe how this approach affects only the preprocessing step.

\subsection{Remarks on the main system and further improvements}
In this section we will show how the main system for the computation of the pressure, developed in Section \ref{sec_semi_imp_dg} results symmetric and, in general, positive semi-definite. These results allows to use very fast methods to solve the system such as the conjugate gradient method with a significant gain in terms of computational time. In order to do this observe how, from the definitions \eref{eq:MD_3} and \eref{eq:MD_6}, we can further generalize the action of $\D$ in terms of $\Q$ such as $\D=-\Q^\top$ since
\begin{eqnarray}
-\Q_{i,j}^\top &=&-\left(\int\limits_{\Omega_{i,j}}{\psi_k^{(j)} \nabla \phi_{l}^{(i)}  dx dy}-\int\limits_{\Gamma_j}{\psi_k^{(j)} \phi_{l}^{(i)}\sigma_{i,j} \nstd ds}\right)^\top \nonumber \\
&=&-\int\limits_{\Omega_{i,j}}{\psi_l^{(j)} \nabla \phi_{k}^{(i)}  dx dy}+\int\limits_{\Gamma_j}{\psi_l^{(j)} \phi_{k}^{(i)}\sigma_{i,j} \nstd ds}
= \D_{i,j}
\end{eqnarray}
and if $i=\ell(j)$, $\nextud$ coincides with $\nstd$, else, it is $-\nstd$, $\forall i, j \in S_i$.
Consequently, the main system \eref{eq:CS_19} can be written as
\begin{eqnarray}
\A: \theta \Delta t \sum\limits_{j\in S_i}\Q_{i,j}^\top \left(\Mpsi_j \right)^{-1}\Q_{i,j} \etah_i^{n+1}
 +\theta \Delta t \sum\limits_{j\in S_i} \Q_{i,j}^\top\left(\Mpsi_j \right)^{-1}\Q_{\p(i,j),j} \etah_{\p(i,j)}^{n+1}=\mathbf{b}_i^n, \nonumber
\label{sys.NOBE}
\end{eqnarray}
that we will call in the following $\A$. If we do not introduce any boundary conditions, we have the following
\begin{Lemma}
Without any boundary conditions the system $\mathcal{A}$ is singular.
\end{Lemma}
\begin{proof}
Let $p_h \in V_h^m$, in order to show that $\mathcal{A}$ is singular we investigate the kernel of the linear operator $\mathcal{A}$. Since $det \mathcal{A} \neq 0 \Leftrightarrow Ker \mathcal{A}=\{0 \}$, we would like to show that the kernel does not contain only the zero.
A weak formulation of $\nabla p_h$ over $\Omega_j$ is given by $ \Q_{\ell(j),j}p_{\ell(j)}+\Q_{r(j),j}p_{r(j)} $, then we have the identity
\begin{equation}
    \Q_{\ell(j),j}\etah_{\ell(j)}+\Q_{r(j),j}\etah_{r(j)}\equiv 0 \Leftrightarrow \nabla p|_{\Omega_j}=0
    \label{proof1.1}
\end{equation}
We are looking for a $p_h \neq 0$ such that $\mathcal{A}p_h=0$. For a fixed $i \in [1,N_i]$,
\begin{eqnarray}
- \theta \Delta t \sum\limits_{j\in S_i}\D_{i,j}\left(\Mpsi_j \right)^{-1}\Q_{i,j} \etah_i^{n+1}- \theta \Delta t \sum\limits_{j\in S_i} \D_{i,j}\left(\Mpsi_j \right)^{-1}\Q_{\p(i,j),j} \etah_{\p(i,j)}^{n+1} &=& 0 \nonumber \\
- \theta \Delta t \sum\limits_{j\in S_i}\D_{i,j}\left(\Mpsi_j \right)^{-1}\left[\Q_{i,j} \etah_i^{n+1}  + \Q_{\p(i,j),j}\etah_{\p(i,j)}^{n+1} \right] &=& 0 \nonumber \\
- \theta \Delta t \sum\limits_{j\in S_i}\D_{i,j}\left(\Mpsi_j \right)^{-1}\left[\Q_{\ell(j),j} \etah_{\ell(j)}^{n+1}  + \Q_{r(j),j}\etah_{r(j)}^{n+1} \right] &=& 0
\label{proof1.2}
\end{eqnarray}
Hence, if $p=constant$, the left side of \eref{proof1.2} vanishes and then $\{p_i \equiv c \,\, \forall i, c \in \R\} \subset ker\mathcal{A} $.

\end{proof}

This represents a natural result since the incompressible NS equations depend only on the gradient of the pressure and not directly on the pressure. Once we have an exact solution for the pressure $p_e$,  then every solution of the kind $p_e+c$ with $c \in \R$ is also a solution. If we introduce the boundary conditions and we specify the pressure in at least one point (i.e. in at least one degree of
freedom), this is equivalent to choose the constant $c$ and the system becomes non-singular. The following results state that the developed system has several important properties such  as the symmetry and, in general, positive semi-definiteness:
\begin{Lemma}[Symmetry]
    The system matrix of $\A$ is symmetric.
\end{Lemma}
\begin{proof}
    In the following we denote with $(i,k)$ the $k-th$ degree of freedom of the $i-th$  element. For the symmetry of $\A$ we have to verify that $(i,k)$ act on $(\tilde{i},\tilde{k})$ as $(\tilde{i},\tilde{k})$ act on $(i,k)$. If $i=\tilde{i}$, the action is described by
$\sum\limits_{j\in S_i}\Q_{i,j}^\top \left(\Mpsi_j \right)^{-1}\Q_{i,j}$ that is trivially  symmetric since $\Mpsi_j=\Mpsi_j^\top$ is symmetric. If $\tilde{i} \not\in \p(i,S_i)$ the two actions are zero so it is also trivially verified. Remains the case $\tilde{i} \in \p(i,S_i)$. In this case, the actions of the right element on  the left one and vice versa are, respectively, $\Q_{\ell(j),j}^\top\Mpsi_j^{-1}\Q_{r(j),j}$ and $\Q_{r(j),j}^\top\Mpsi_j^{-1}\Q_{\ell(j),j}$. A simple computation leads to
\begin{eqnarray}
  \Mpsi_j^{-1}\Q_{r(j),j}(k,l)=\sum_{\xi=1}^{N_\psi} \Mpsi_j^{-1}(k,\xi) \Q_{r(j),j}(\xi,l) \qquad \forall k=1 \ldots N_\psi \,\, , \,\, l=1 \ldots N_\phi \nonumber
\end{eqnarray}
and then $\forall k=1 \ldots N_\phi \,\, , \,\, l=1 \ldots N_\phi$,
\begin{eqnarray}
  \Q_{\ell(j),j}^\top \Mpsi_j^{-1}\Q_{r(j),j}(k,l)&=&\sum_{\gamma=1}^{N_\psi} \Q_{\ell(j),j}^\top(k,\gamma) \left(\Mpsi_j^{-1}\Q_{r(j),j}\right)(\gamma,l) \nonumber \\
  &=& \sum_{\gamma=1}^{N_\psi} \Q_{\ell(j),j}^\top(k,\gamma) \sum_{\xi=1}^{N_\psi} \Mpsi_j^{-1}(\gamma,\xi) \Q_{r(j),j}(\xi,l) \nonumber \\
  &=& \sum_{\gamma,\xi=1}^{N_\psi} \Q_{\ell(j),j}^\top(k,\gamma) \Mpsi_j^{-1}(\gamma,\xi) \Q_{r(j),j}(\xi,l) \nonumber \\
  &=& \sum_{\gamma,\xi=1}^{N_\psi} \Q_{\ell(j),j}(\gamma,k) \Mpsi_j^{-1}(\gamma,\xi) \Q_{r(j),j}^\top(l,\xi) \nonumber \\
  &=& \sum_{\gamma,\xi=1}^{N_\psi} \Q_{r(j),j}^\top(l,\xi) \Mpsi_j^{-1}(\xi,\gamma) \Q_{\ell(j),j}(\gamma,k)  \nonumber \\
  &=& \Q_{r(j),j}^\top \Mpsi_j^{-1}\Q_{\ell(j),j}(k,l)
\end{eqnarray}
\end{proof}

\begin{Lemma}
The matrix $\mathcal{A}$ is positive semi-definite, i.e. $x^\top A x \geq 0 \,\,\,\, \forall x \in \R^{N_i \cdot N_{\phi}}$
\end{Lemma}

\begin{proof}
We do the computation directly.
$x^\top A x= \sum_i{(x^\top A x)_i}$ and
\begin{eqnarray}
	(x^\top A x)_i &=& x_i\sum_{j \in S_i} \Q_{i,j}^\top M_j^{-1}\Q_{i,j} x_i + x_i\sum_{j \in S_i} \Q_{i,j}^\top M_j^{-1}\Q_{\p(i,j),j} x_{\p(i,j)} \nonumber \\
								&=& \sum_{j \in S_i} \left(M_j^{-\frac{1}{2}} \Q_{i,j} x_i \right)^\top \left(M_j^{-\frac{1}{2}} \Q_{i,j} x_i \right) \nonumber \\
								&&+ \sum_{j \in S_i} \left(M_j^{-\frac{1}{2}} \Q_{i,j} x_i \right)^\top \left(M_j^{-\frac{1}{2}} \Q_{\p(i,j),j} x_{\p(i,j)} \right) \nonumber
\end{eqnarray}
where we used that $M_j$ is symmetric and positive definite, hence $M_j^{-1}$ is symmetric and positive definite and then exists the so called square operator, namely $\exists M_j^{-\frac{1}{2}}$ such that $M_j^{-1}=\left( M_j^{-\frac{1}{2}} \right)^\top \left( M_j^{-\frac{1}{2}} \right)$. By defining $T_{i,j}:= M_j^{-\frac{1}{2}} \Q_{i,j}$ we obtain
\begin{eqnarray}
	(x^\top A x)_i &=& \sum_{j \in S_i} \left(T_{i,j} x_i \right)^\top \left(T_{i,j} x_i \right) + \sum_{j \in S_i} \left(T_{i,j} x_i \right)^\top \left(T_{\p(i,j),j} x_{\p(i,j)} \right)
\end{eqnarray}
and consequently
\begin{eqnarray}
	x^\top A x= \sum_{i=1}^{N_i} \sum_{j \in S_i} \left(T_{i,j} x_i \right)^\top \left(T_{i,j} x_i \right) + \sum_{i=1}^{N_i} \sum_{j \in S_i} \left(T_{i,j} x_i \right)^\top \left(T_{\p(i,j),j} x_{\p(i,j)} \right)
\end{eqnarray}

Remark that the double summation $\sum_{i=1}^{N_i} \sum_{j \in S_i}$ sum every element $i$ and edge $j$. From the edge point of view, every edge gives two contributions, one given when $i=\ell(j)$ and one when $i=r(j)$. The double summation can be consequently inverted as follows:
\begin{eqnarray}
	\sum_{i=1}^{N_i} \sum_{j \in S_i} \left(T_{i,j} x_i \right)^\top \left(T_{i,j} x_i \right) &=& \sum_{j=1}^{N_j} \left(T_{\ell(j),j} x_{\ell(j)} \right)^\top \left(T_{\ell(j),j} x_{\ell(j)} \right) \nonumber \\
	&&+\sum_{j=1}^{N_j} \left(T_{r(j),j} x_{r(j)} \right)^\top \left(T_{r(j),j} x_{r(j)} \right) \nonumber \\	
	\sum_{i=1}^{N_i} \sum_{j \in S_i} \left(T_{i,j} x_i \right)^\top \left(T_{\p(i,j),j} x_{\p(i,j)} \right) &=& \sum_{j=1}^{N_j} \left(T_{\ell(j),j} x_{\ell(j)} \right)^\top \left(T_{r(j),j} x_{r(j)} \right) \nonumber \\
	&&+\sum_{j=1}^{N_j} \left(T_{r(j),j} x_{r(j)} \right)^\top \left(T_{\ell(j),j} x_{\ell(j)} \right)
\end{eqnarray}
and then, by recompose everything
\begin{eqnarray}
	x^\top A x &=& \sum_{j=1}^{N_j} \left[  \left(T_{\ell(j),j} x_{\ell(j)} \right)^\top \left(T_{\ell(j),j} x_{\ell(j)} \right)+ \left(T_{r(j),j} x_{r(j)} \right)^\top \left(T_{r(j),j} x_{r(j)} \right) \right. \nonumber \\
	&&	\left. \left(T_{\ell(j),j} x_{\ell(j)} \right)^\top \left(T_{r(j),j} x_{r(j)} \right)+ \left(T_{r(j),j} x_{r(j)} \right)^\top \left(T_{\ell(j),j} x_{\ell(j)}\right) \right] \nonumber \\
		&=& \sum_{j=1}^{N_j} \left( T_{\ell(j),j}x_{\ell(j)}+T_{r(j),j}x_{r(j)} \right)^\top \left( T_{\ell(j),j}x_{\ell(j)}+T_{r(j),j}x_{r(j)} \right) \nonumber \\
		&=& \sum_{j=1}^{N_j} \left[
		\left(
		\begin{array}{cc}
			T_{\ell(j),j} & 0 \\
			0 & T_{r(j),j}
		\end{array}
		\right)
		\cdot
		\left(
		\begin{array}{c}
			x_{\ell(j)} \\
			x_{r(j)}
		\end{array}
		\right)		
		 \right]^\top
		 \left[
		\left(
		\begin{array}{cc}
			T_{\ell(j),j} & 0 \\
			0 & T_{r(j),j}
		\end{array}
		\right)
		\cdot
		\left(
		\begin{array}{c}
			x_{\ell(j)} \\
			x_{r(j)}
		\end{array}
		\right)		
		 \right] \nonumber \\
		 &=& \sum_{j=1}^{N_j} (x_{\ell(j)} , x_{r(j)})
		 		\left(
				\begin{array}{cc}
					T_{\ell(j),j} & 0 \\
					0 & T_{r(j),j}
				\end{array}
				\right)^\top
				\left(
				\begin{array}{cc}
					T_{\ell(j),j} & 0 \\
					0 & T_{r(j),j}
				\end{array}
				\right)
				\left(
				\begin{array}{c}
					x_{\ell(j)} \\
					x_{r(j)}
				\end{array}
				\right)	\nonumber \\
				&=& \sum_{j=1}^{N_j}\vec{x}_j^\top \mathcal{T}^\top \mathcal{T} \vec{x}_j
				\label{eq_proof3}
\end{eqnarray}
And, since $\tilde{\mathcal{T}}:=\mathcal{T}^\top \mathcal{T}$ is a positive semi-definite matrix by construction, $\vec{x}_j^\top \tilde{\mathcal{T}} \vec{x}_j \geq 0$ and then $x^\top A x=\sum_j\vec{x}_j^\top \tilde{\mathcal{T}} \vec{x}_j \geq0$
\end{proof}

We introduce now the boundary elements and, in particular,
    $$\D_{i,j}^\partial = \int\limits_{\Gamma_j}{\phi_k^{(i)}\psi_l^{\partial(j)}\nextud ds}-\int\limits_{\TT_{i,j}}{\nabla \phi_k^{(i)}\psi_l^{\partial(j)}dx dy}$$
    and
    $$ \Q_{i,j}^\partial=\int\limits_{\TT_{i,j}}{\psi_k^{\partial(j)} \nabla \phi_{l}^{(i)}  dx dy}-\int\limits_{\Gamma_j}{\psi_k^{\partial(j)} \phi_{l}^{(i)}\sigma_{i,j} \nstd ds}. $$

    Then it is still true that $\D_{i,j}^\partial=-\left(\Q_{i,j}^\partial \right)^\top$ and the complete system $\tilde{\A}$ can be written as $\tilde{\A}= \A+\B$ where

    \begin{eqnarray}
    \B: \qquad && \theta \Delta t \sum\limits_{j\in S_i\cap \B(\Omega)}\left(\Q_{i,j}^{\partial}\right)^\top \Mpsi_j^{-1}\Q_{i,j}^{\partial} \etah_i^{n+1} \nonumber \\
     \tilde{\A}: \qquad && \theta \Delta t \left[ \sum\limits_{j\in S_i\cap \B(\Omega)}\left(\Q_{i,j}^{\partial}\right)^\top \Mpsi_j^{-1}\Q_{i,j}^{\partial} +\sum\limits_{j\in S_i-\B(\Omega)}\Q_{i,j}^\top\Mpsi_j^{-1}\Q_{i,j} \right] \etah_i^{n+1} \nonumber \\
       && +\theta\Delta t \sum\limits_{j\in S_i-\B(\Omega)} \Q_{i,j}^\top\Mpsi_j^{-1}\Q_{\p(i,j),j} \etah_{\p(i,j)}^{n+1} \nonumber
    \end{eqnarray}

It is easy to check that $\B$ is symmetric and at least positive semi-definite.
\par
We have to introduce now some types of boundary conditions in order to show that, if the pressure is specified on the boundary, the complete system $\tilde{\A}$ is positive definite.

{Let us rewrite $x^\top \B x$ by including the external contribution and in the form of the Eq. \eref{eq_proof3}, namely
\begin{eqnarray}
	x^\top \B x = \sum_{j=1}^{N_j} \left( T^\partial_{\ell(j),j}x_{\ell(j)}+\left[T^\partial x \right]_{ext,j} \right)^\top \left( T^\partial_{\ell(j),j}x_{\ell(j)}+\left[T^\partial x \right]_{ext,j} \right) \nonumber
\label{eq:ppd1}
\end{eqnarray}
where $T^\partial_{i,j}=M_j^{-\frac{1}{2}}\Q^\partial_{i,j}$ and $\left[T^\partial x \right]_{ext,j}$ is a known external contribution that depends on the boundary conditions. In particular, if the pressure is specified at the boundary, then $T^\partial_{ext,j}=T^\partial_{\ell(j),j}$ and
$\left[T^\partial x \right]_{ext,j}$ is a known quantity that in general is part of the known right hand side vector. Since the external pressure is specified, then $T^\partial_{\ell(j),j}x_{\ell(j)}+\left[T^\partial x \right]_{ext,j}=0 \Leftrightarrow x_{\ell(j)} \equiv x_{ext,j}$. We take now $x^\top \B x=0$ that implicitly fixes $x_{ext}=0$. In this way $x_{\ell(j)}=0 \,\, \forall j \in \B(\Omega)$.
Using the same reasoning on the matrix $\A$ we can conclude that $x \equiv 0$, and hence $\tilde{\A}$ is positive definite in this case.
A possible way to specify the velocity at the boundary is to neglect the jump contribution for the pressure at the boundary or equivalent, taken $x_{ext,j}=x_{\ell(j)} \,\,\,\, \forall j \in \B(\Omega)$. It is easy to check that if we have only this type of boundary conditions then $x^\top \tilde{\A} x=0 $ for every $x$ constant, and then the matrix $\tilde{\A}$ is only positive semi-definite.
}

\section{Numerical test problems}
\label{sec.tests}
\subsection{Convergence test}
We consider a smooth steady state problem in order to measure the order of accuracy of the proposed method. For this purpose, the Navier-Stokes
equations are first rewritten in cylindrical coordinates ($r$ and $\varphi$), with $r^2=x^2+y^2$,
$\tan \varphi = x/y$, the radial velocity component $u_r$ and the angular velocity component $u_\varphi$.
In order to derive an analytical solution we suppose a steady vortex-type flow with angular symmetry, i.e. $\partial / \partial t = 0$, $\partial / \partial \varphi = 0$ and $u_r=0$. With these assumptions, the continuity equation is automatically satisfied and the system of incompressible Navier-Stokes equations reduces to
\begin{equation}
\left\{
    \begin{array}{l}
        \diff{p}{r}=\frac{u_\varphi^2}{r}, \\
        r \frac{\partial^2 u_\varphi}{\partial r^2}+\diff{u_\varphi}{r}-\frac{u_\varphi}{r}=0.
    \end{array}
\right.
\label{eq:CT_1}
\end{equation}
One can now recognize in the second equation of \eref{eq:CT_1} a classical second order Cauchy Euler equation and so obtain two solutions for $u_\varphi$, namely:
\begin{equation}
    u_\varphi=c_1 r,
    \label{eq:CT_2_1}
\end{equation}
\begin{equation}
    u_\varphi=\frac{c_1}{r} ,
    \label{eq:CT_2_2}
\end{equation}
for every $c_1 \in \R$. The corresponding pressures read
\begin{equation}
    p=\frac{c_1^2 r^2}{2}+c_2,
    \label{eq:CT_3_1}
\end{equation}
\begin{equation}
    p=-2\frac{c_1^2}{r^2}+c_2.
    \label{eq:CT_3_2}
\end{equation}
respectively. In this section we set the boundary conditions in order to obtain the non-trivial solution \eref{eq:CT_2_2}-\eref{eq:CT_3_2}. Due to the singularity of $u_\varphi$ for $r=0$, let $\Omega=C(5)-C(1)$ where $C(r)=\{(x,y) \in \R^2 \,\, | \,\, \sqrt{x^2+y^2} \leq r\}$. As initial condition we impose Eqs. \eref{eq:CT_2_2}-\eref{eq:CT_3_2} with $c_1=u_\varphi(1)=2$ and $c_2=0$. The exact velocity is imposed at the internal boundary while exact pressure is specified at the external circle. The proposed algorithm is validated for several polynomial degrees $p$ using successively refined grids. The chosen parameters for the numerical simulations are $t_{end}=0.75$; $\theta=1$; $\nu=10^{-5}$; the time step $\Delta t$ is taken according to the CFL time restriction for the explicit discretization of the nonlinear convective term \eref{eq:CFLC}. The $L_2$ error  between the analytical and the numerical solution is computed as
\begin{equation}
\epsilon(p)= \sqrt{\int\limits_\Omega (p_h-p_{e})^2 dx dy}, \qquad
\epsilon(\vv)= \sqrt{\int\limits_\Omega (\vec{v}_h-\vec{v}_{e})^2 dx dy },
\end{equation}
for the pressure and for the velocity vector field, respectively, where the subscript $h$ indicates the numerical solution and $e$ denotes the exact solution.
\begin{table}[!htb]
\begin{center}
\begin{tabular}{|c|c|c|c|c|c|c|c|c|}
\hline
$N_i$ & \multicolumn{4}{|c|}{$p=0$} & \multicolumn{4}{c|}{$p=1$} \\
  & \multicolumn{1}{|c}{$\epsilon(p)$} & \multicolumn{1}{c}{$\epsilon(\vec{v})$} & \multicolumn{1}{c}{$\mathcal{O} (p)$} & \multicolumn{1}{c|}{$\mathcal{O} (\vec{v})$} & \multicolumn{1}{|c}{$\epsilon(p)$} & \multicolumn{1}{c}{$\epsilon(\vec{v})$} & \multicolumn{1}{c}{$\mathcal{O} (p)$} & \multicolumn{1}{c|}{$\mathcal{O} (\vec{v})$} \\
  \hline \hline
    124 &  7.902E-01 & 1.095E-00 & - & -        &  3.944E-01   &  4.311E-01         &  -     &   -    \\
    496 &  5.026E-01 & 7.086E-01  & 0.7 & 0.6   &  8.830E-02   & 1.221E-01      & 2.2  &  1.8 \\
    1984 &  2.982E-01 & 4.502E-01 & 0.8 & 0.7   &  2.325E-02   & 3.299E-02      & 1.9  &  1.9 \\
    7936 &  1.659E-01 & 2.797E-01 & 0.8 & 0.7   &  6.207E-03   & 8.725E-03      & 1.9  &  1.9 \\
    31744 &  8.797E-02 & 1.714E-01 & 0.9 & 0.7  &  1.615E-03   & 2.318E-03  & 1.9 &  1.9 \\
  \hline
\end{tabular}
\end{center}
\caption{Numerical convergence results for $p=0$ and $p=1$.}
\label{tab:1}
\end{table}

\begin{table}[!htb]
\begin{center}
\begin{tabular}{|c|c|c|c|c|c|c|c|c|}
\hline
$N_i$ & \multicolumn{4}{|c|}{$p=2$} & \multicolumn{4}{c|}{$p=3$} \\
  & \multicolumn{1}{|c}{$\epsilon(p)$} & \multicolumn{1}{c}{$\epsilon(\vec{v})$} & \multicolumn{1}{c}{$\mathcal{O} (p)$} & \multicolumn{1}{c|}{$\mathcal{O} (\vec{v})$} & \multicolumn{1}{|c}{$\epsilon(p)$} & \multicolumn{1}{c}{$\epsilon(\vec{v})$} & \multicolumn{1}{c}{$\mathcal{O} (p)$} & \multicolumn{1}{c|}{$\mathcal{O} (\vec{v})$} \\
  \hline \hline
    124 &  9.366E-02 & 1.990E-01 & - & -            &  4.346E-02    &  9.317E-02     &  -     &   -    \\
    496 &   1.054E-02 & 3.069E-02 & 3.2 & 2.7       &  2.966E-03    & 8.027E-03         & 3.9  &  3.5 \\
    1984 &  1.193E-03 & 3.686E-03 & 3.1 & 3.1       &  1.783E-04    & 7.153E-04         & 4.1  &  3.5 \\
    7936 & 1.438E-04  & 4.425E-04 & 3.1 & 3.1       &  1.313E-05    & 5.997E-05         & 3.8  & 3.6    \\
  \hline
\end{tabular}
\end{center}
\caption{Numerical convergence results for $p=2$ and $p=3$.}
\label{tab:2}
\end{table}

Tables $\ref{tab:1}$ and $\ref{tab:2}$ show the $L_2$ convergence rates for successive refinements of the grid, where $\mathcal{O}(p)$ and $\mathcal{O} (\vec{v})$ represent the order of accuracy achieved  for the pressure and the velocity field, respectively. The optimal convergence is reached up to $p=2$  while for $p=3$ the observable order of accuracy for the velocity vector field is closer to $p+\frac{1}{2}$ rather then $p+1$.

\subsection{Womersley profiles}
In this section the proposed algorithm is verified against the exact solution for an oscillating flow in a rigid tube of length $L$. The unsteady flow is driven by  a sinusoidal pressure gradient on the boundaries
\begin{equation}
	\frac{p_{out}(t)-p_{inlet}(t)}{L}=\frac{\tilde{p}}{\rho}e^{i \omega t},
\label{eq:W_1}
\end{equation}
where $\tilde{p}$ is the amplitude of the pressure gradient; $\rho$ is the fluid density; $\omega$ is the frequency of the oscillation; $i$ indicates the imaginary unit; $p_{inlet}$ and $p_{out}$ are the inlet and outlet pressures, respectively. The analytical solution was derived by Womersley in \cite{Womersley1995}.
According to \cite{Womersley1995,FambriDumbserCasulli} no convective contribution is considered.
By imposing Eq. \eref{eq:W_1} at the tube ends, the resulting unsteady velocity field is uniform in the axial direction and is given by
\begin{equation}
	u_{e}(x,y,t)=\frac{\tilde{p}}{\rho}\frac{1}{i \omega}\left[ 1- \frac{J_0\left(\alpha \zeta i^{\frac{3}{2}}\right)}{J_0\left(\alpha i^{\frac{3}{2}}\right)} \right]e^{i \omega t}\,\, ; \,\, v_{e}(x,y,t)=0,
\label{eq:W_2}
\end{equation}
where $\zeta=2y/D$ is the dimensionless radial coordinate; $D$ is the diameter of the tube; $\alpha=\frac{D}{2}\sqrt{\frac{\omega}{\nu}}$ is a constant; and $J_0$ is the zero-th order Bessel function of the first kind. For the present test we take $\Omega=[-0.5, 1]\times [-0.2, 0.2]$; $\tilde{p}=1000$; $\rho=1000$; $\omega=2 \pi$; $\theta=0.6$; and $\nu=8.94\times 10^{-4}$. The computational domain $\Omega$ is covered with a total number of $N_i=98$ triangles and the time step size is chosen as $\Delta t=0.01$. The numerical results for $p=3$ are shown in Fig. $\ref{fig.W.1}$ for several times at $x=0.1$. A good agreement between exact and
numerical solution can be observed.
\begin{figure}[ht]
    \begin{center}
    \includegraphics[width=0.6\textwidth]{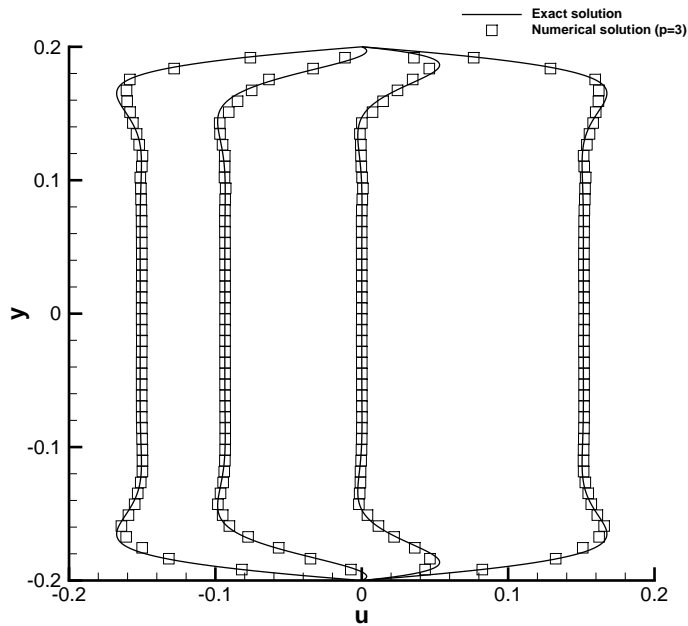}
    \caption{Comparison between the exact and the numerical solution for the Womersley profiles
		at times $t=1.7$, $t=1.9$, $t=2.0$, $t=1.2$, respectively, from left to right.}
    \label{fig.W.1}
		\end{center}
\end{figure}

\subsection{Blasius boundary layer}
Another classical test problem concerns the Blasius boundary layer.
For the particular case of laminar stationary flow over a flat plate, a solution of Prandtl's boundary layer equations was found by Blasius in \cite{Blasius1908} and is determined by the solution of a third-order non-linear ODE, namely:
\begin{eqnarray}
\left\{
\begin{array}{l}
    f'''+ff''=0 \\
    f(0)=0 \\
    f'(0)=0 \\
    \lim_{\xi \rightarrow \infty} f'(\xi)=1
\end{array}
\right.
\end{eqnarray}
where $\xi=y \sqrt{\frac{u_{\infty}}{2\nu x}}$ is the Blasius coordinate; $f'=\frac{u}{u_\infty}$; and $u_\infty$ is the farfield velocity. The reference solution is computed here using a tenth-order DG ODE solver, see e.g. \cite{ADERNSE}, together with a classical shooting method.
In order to obtain the Blasius velocity profile in our simulations we consider a steady flow over a a wedge-shaped object. As a result of the viscosity,
a boundary layer appears along the obstacle. For the present test, we consider $\Omega=[0,1]\times [-0.25, 0.25]$ and a wedge shape object with upper edge corresponding to the segment $x=[0,1]$. An initially uniform flow $u(x,y,0)=u_\infty=1$ , $v(x,y,0)=0$ and $p(x,y,0)=1$ is imposed as initial condition, while an inflow boundary is imposed on the left and outflow boundary conditions are imposed on the other edges of the external box. Finally, no-slip wall boundary conditions are considered over the wedge shape object. We cover $\Omega$ with a total amount of $N_i=278$ triangles and use $\theta=1$ and $p=3$. The resulting  Blasius velocity profile is shown in Figure $\ref{fig.B.1}$ while the profile with respect to the Blasius coordinate $\xi$ is shown in Figure $\ref{fig.B.2}$ in order to verify whether the obtained solution is self-similar with respect to $\xi$. A comparison between the numerical results presented here and the reference solution is depicted in Figure $\ref{fig.B.3}$ for $x=0.4$ and $x=0.6$.
\begin{figure}[ht]
    \begin{center}
    \includegraphics[width=0.8\textwidth]{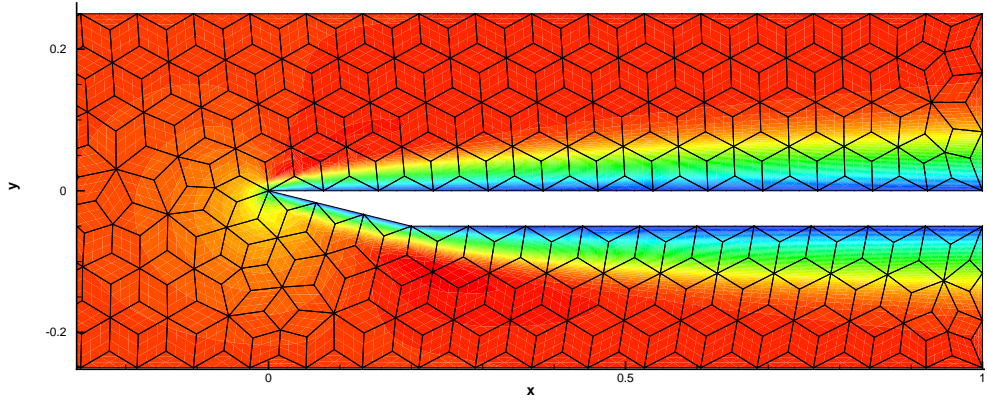}
    \caption{Computational domain used for the simulation of the Blasius boundary layer. The colors represent the horizontal velocity $u$.}
    \label{fig.B.1}
		\end{center}
\end{figure}
\begin{figure}[ht]
    \begin{center}
    \includegraphics[width=0.8\textwidth]{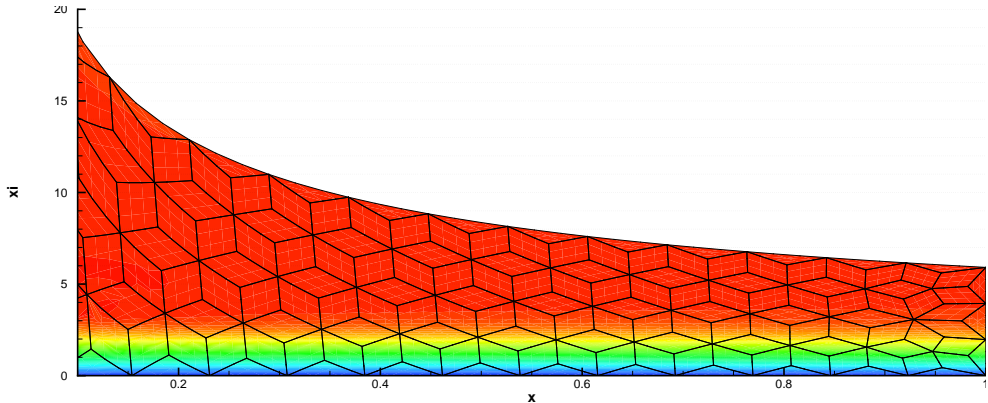}
    \caption{Velocity profile with respect to the Blasius coordinate $\xi$.}
    \label{fig.B.2}
		\end{center}
\end{figure}
\begin{figure}[ht]
    \begin{center}
    \includegraphics[width=0.65\textwidth]{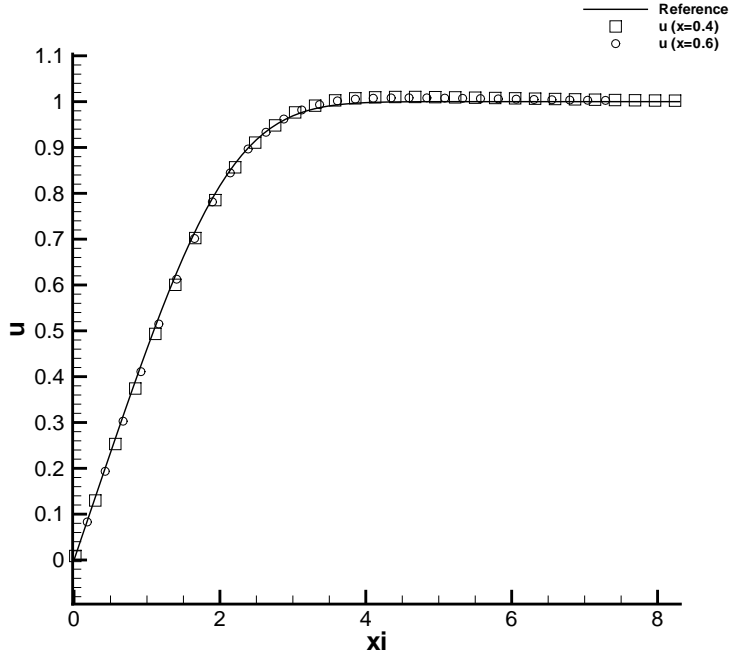}
    \caption{Numerical and reference solution for the Blasius boundary layer at $x=0.4$ and $x=0.6$.}
    \label{fig.B.3}
		\end{center}
\end{figure}
A good agreement between the reference solution and the numerical results obtained with the staggered semi-implicit DG scheme is obtained, despite the use of a very coarse
grid. Note that the solution in terms of the Blasius coordinate $\xi$ is independent from $x$. The numerical solution is also verified to maintain the self-similar Blasius
profile in the $(x,\xi)$ plane, see Fig. \ref{fig.B.2}.

\subsection{Lid-driven cavity flow}
We consider here another classical benchmark problem for the incompressible Navier-Stokes equations, namely the lid-driven cavity problem.
This test problem is solved numerically with the new staggered DG scheme on very coarse grids using a polynomial degree of $p=3$.
Let $\Omega=[-0.5,0.5]\times [-0.5, 0.5]$, set velocity boundary conditions $u=1$ and $v=0$ on the top boundary (i.e. $y=0.5$) and
impose no-slip wall boundary conditions on the other edges. As initial condition we take $u(x,y,0)=v(x,y,0)=0$. We use a grid with
$N_i=73$ triangles for $Re=100,400,1000$ and $N_i=359$ triangles for $Re=3200$. A sketch of the main and dual grid is shown in Fig.
$\ref{fig.C.2}$.

\begin{figure}[ht]
    \begin{center}
    \includegraphics[width=0.55\textwidth]{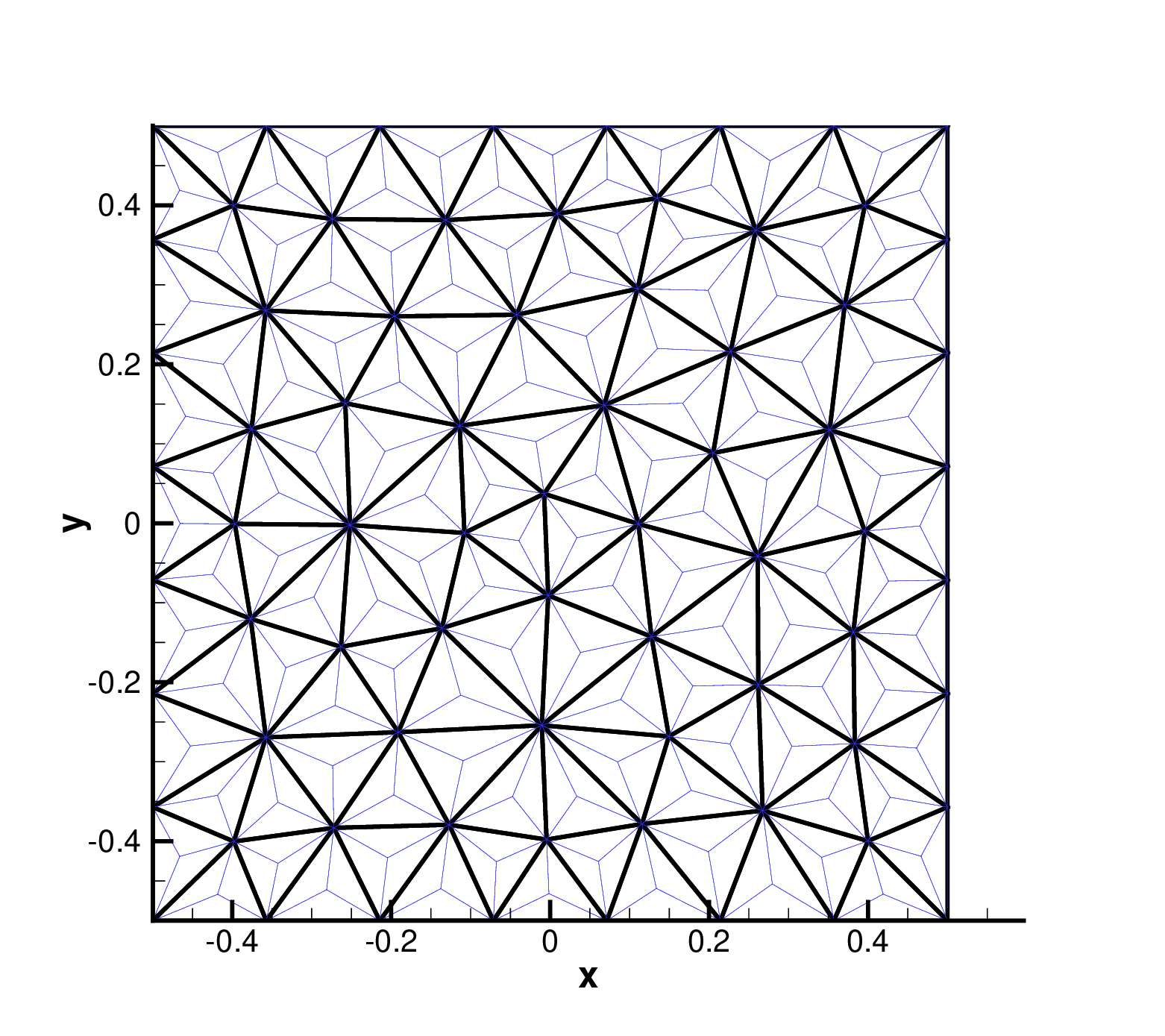}
    \caption{Main and dual grid used for the lid-driven cavity problem for $Re=100,400,1000$.}
    \label{fig.C.2}
		\end{center}
\end{figure}
For the present test $\theta=1$; $\Delta t$ is taken according to condition \eref{eq:CFLC}; and $t_{end}=150$.
According to \cite{Khurshid2014,Ghia1982}, primary and corner vortices appear from $Re=100$ to $Re=3200$, a comparison of the velocities against the data presented in \cite{Ghia1982}, as well as the streamline plots are shown in Figure $\ref{fig.C.1}$. A very good agreement is obtained in all cases, even if a very coarse grid  has been used.
\begin{figure}[ht]
    \begin{center}
    \includegraphics[width=0.45\textwidth]{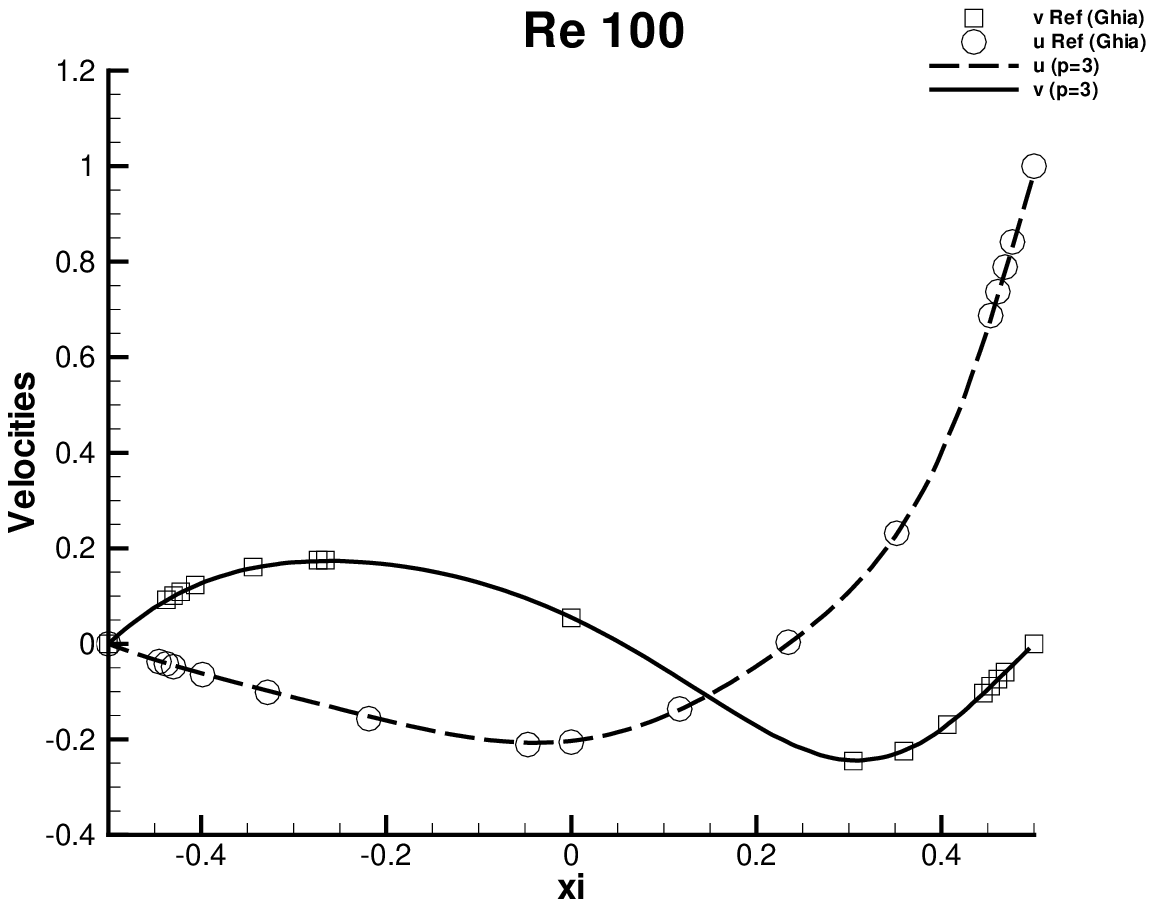}
    \includegraphics[width=0.35\textwidth]{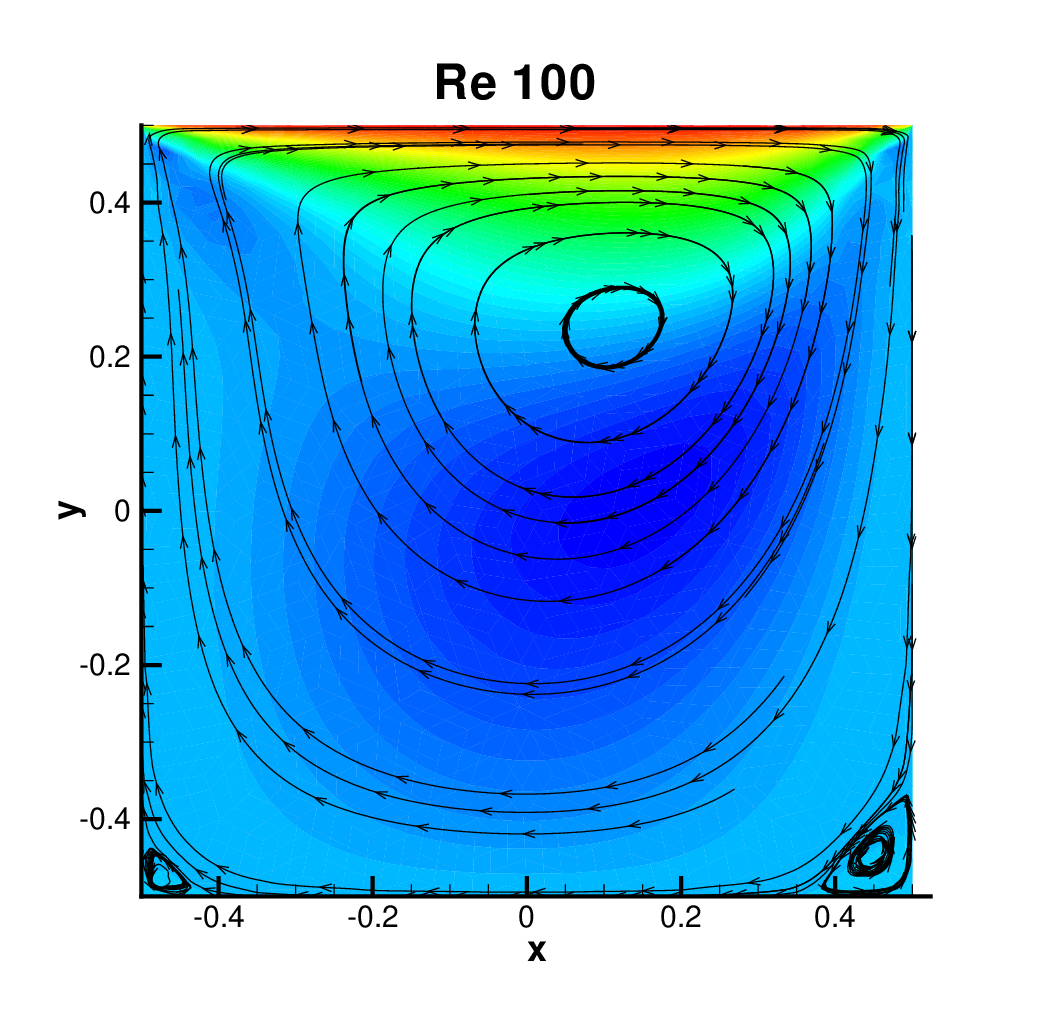} \\
    \includegraphics[width=0.45\textwidth]{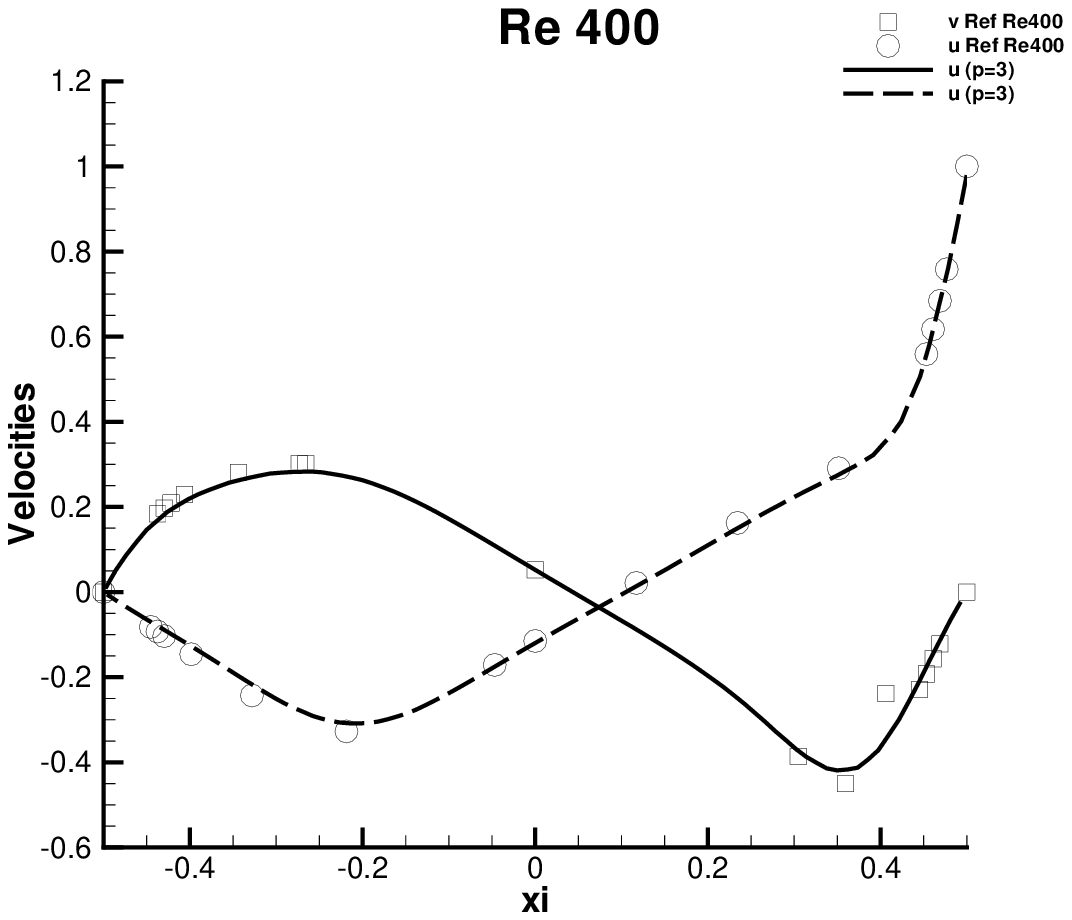}
    \includegraphics[width=0.35\textwidth]{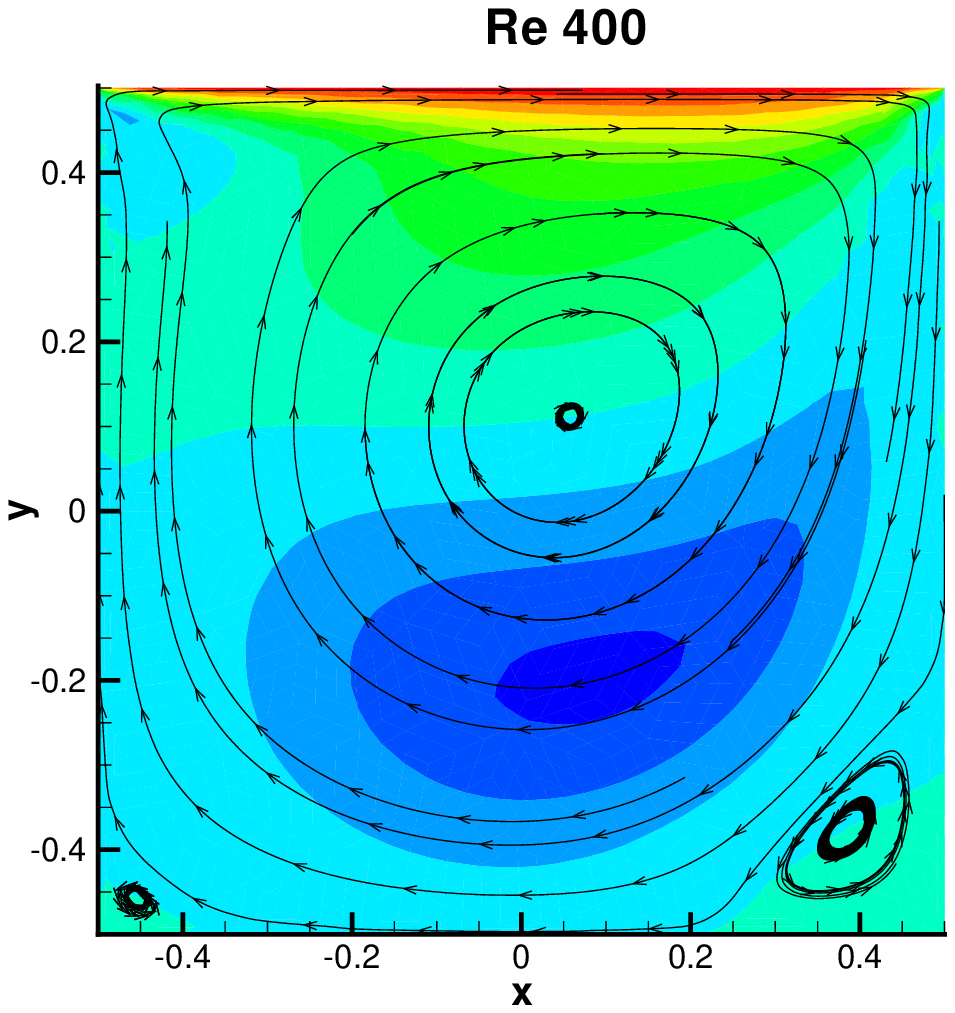} \\
    \includegraphics[width=0.45\textwidth]{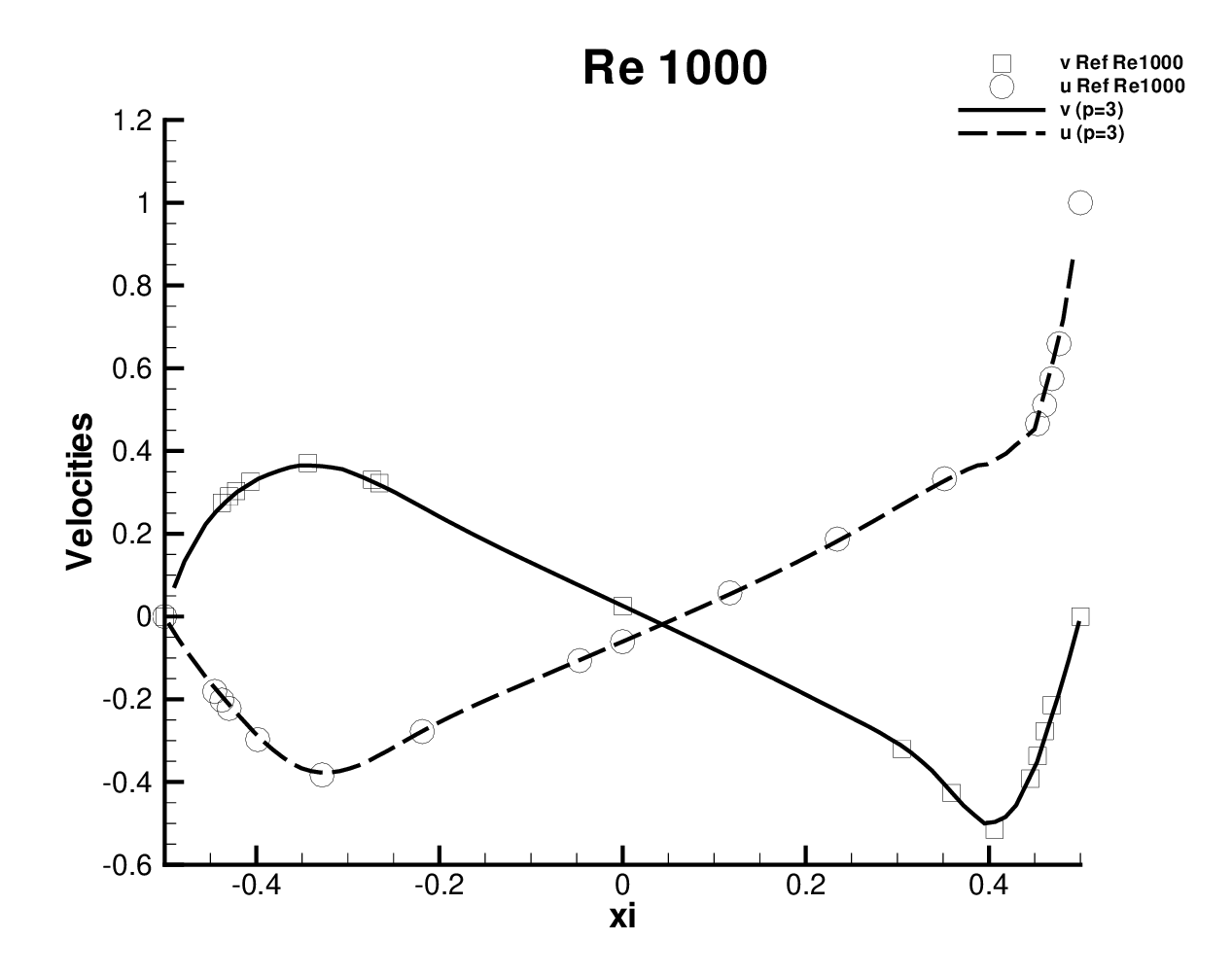}
    \includegraphics[width=0.35\textwidth]{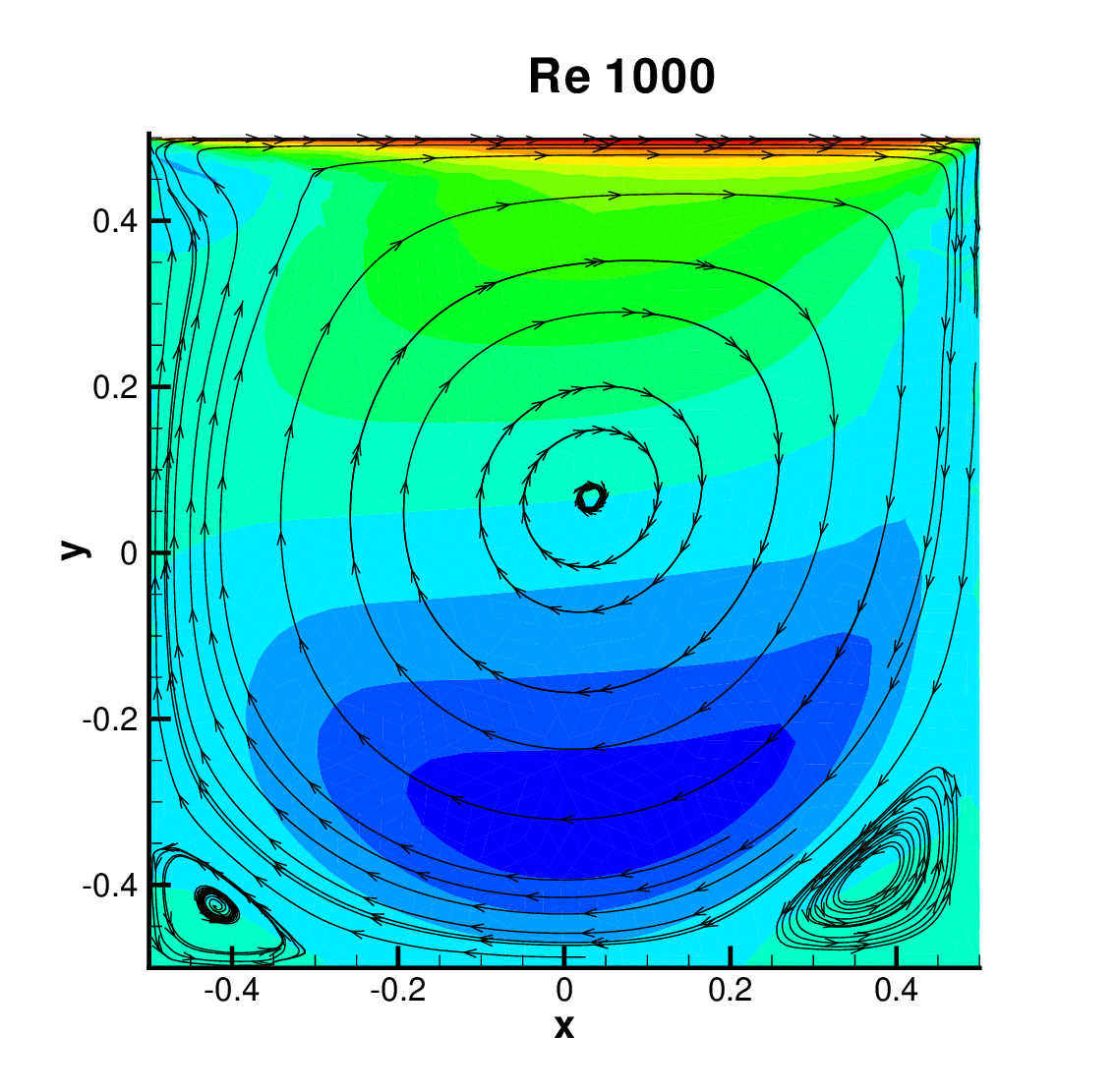} \\
    \includegraphics[width=0.45\textwidth]{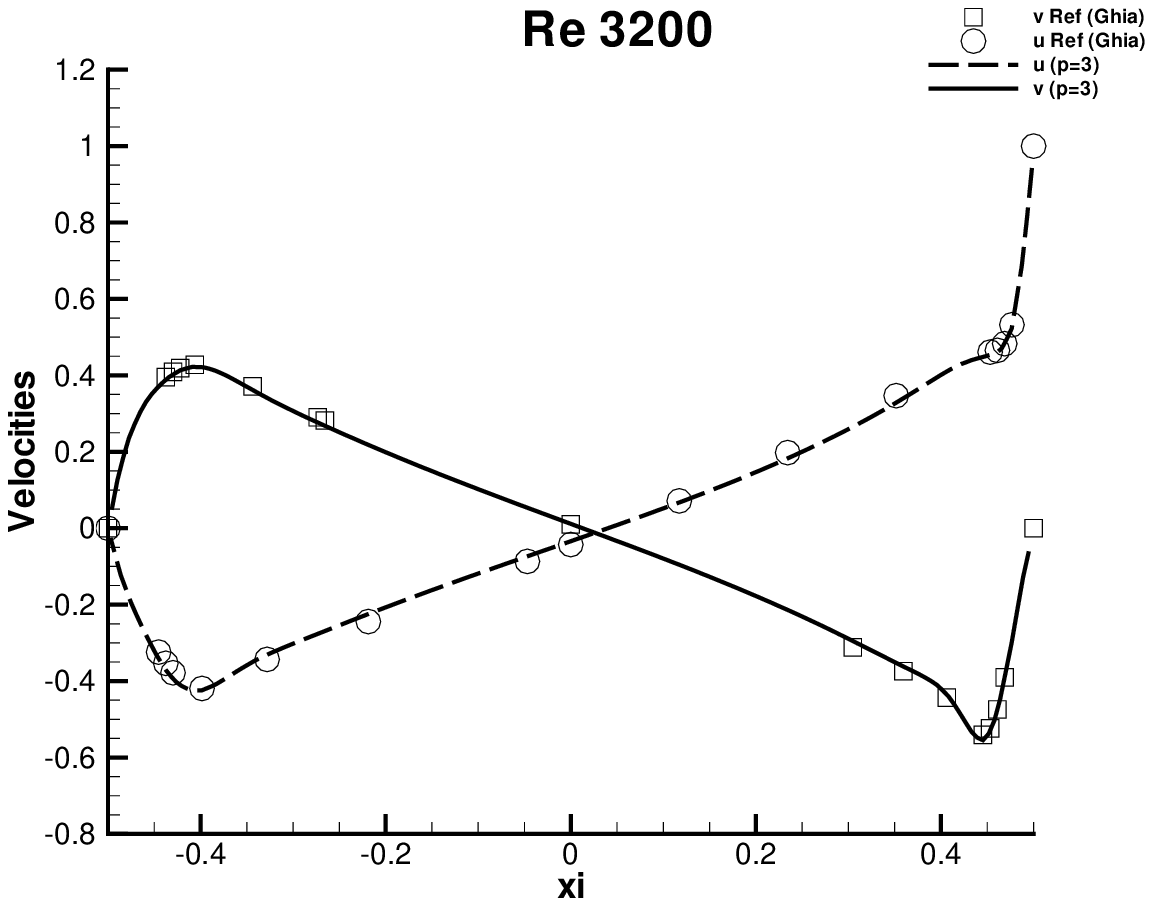}
    \includegraphics[width=0.35\textwidth]{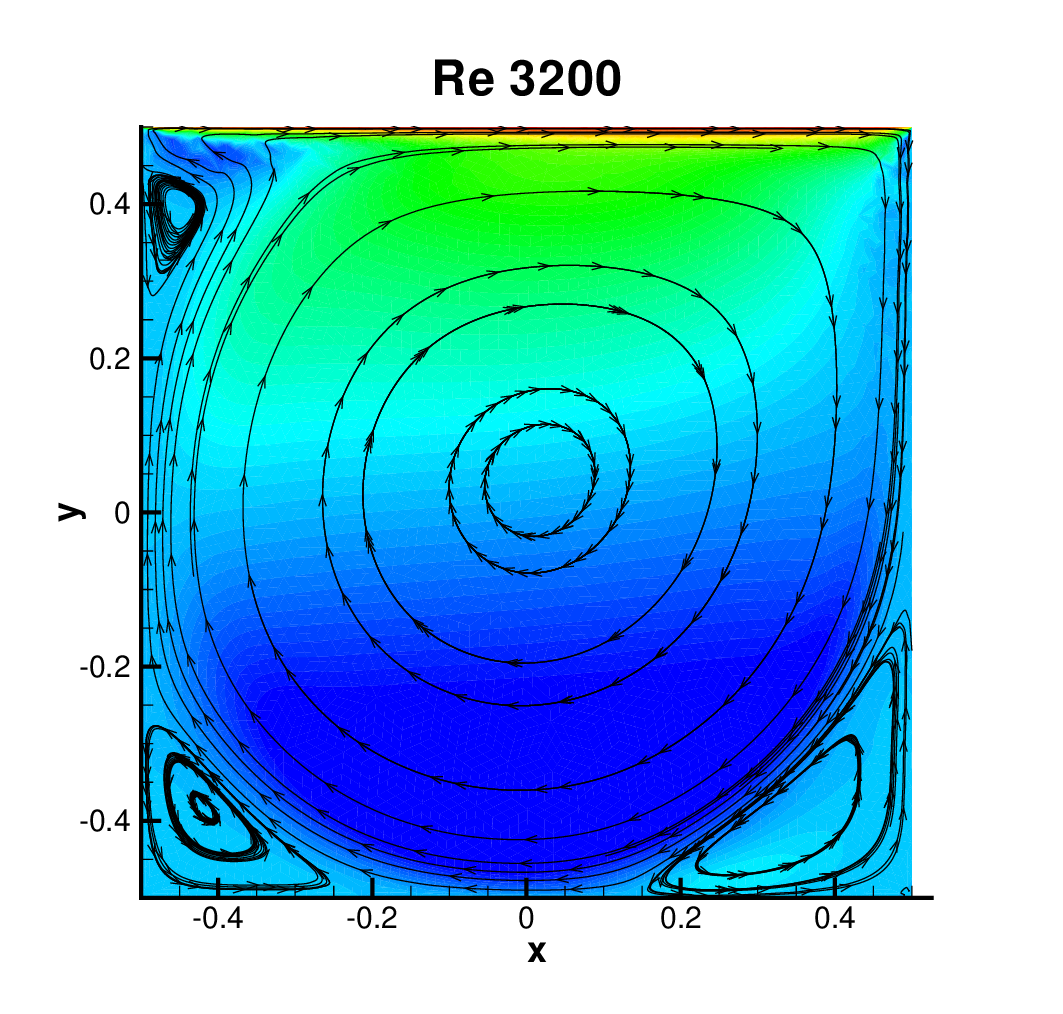} \\
    \caption{Velocity profiles (left) and streamlines (right) at several Reynolds numbers for the lid-driven cavity problem.}
    \label{fig.C.1}
		\end{center}
\end{figure}

\subsection{Backward-facing step.}
In this section, the numerical solution for the fluid flow over a backward-facing step is considered. For this test problem, both experimental and numerical results are available at several Reynolds numbers (see e.g. \cite{Armaly1983,Erturk2008}). The computational domain $\Omega$ and the main notation are reported in Figure $\ref{fig.BFS.1}$.
\begin{figure}[ht]
    \begin{center}
    \includegraphics[width=0.95\textwidth]{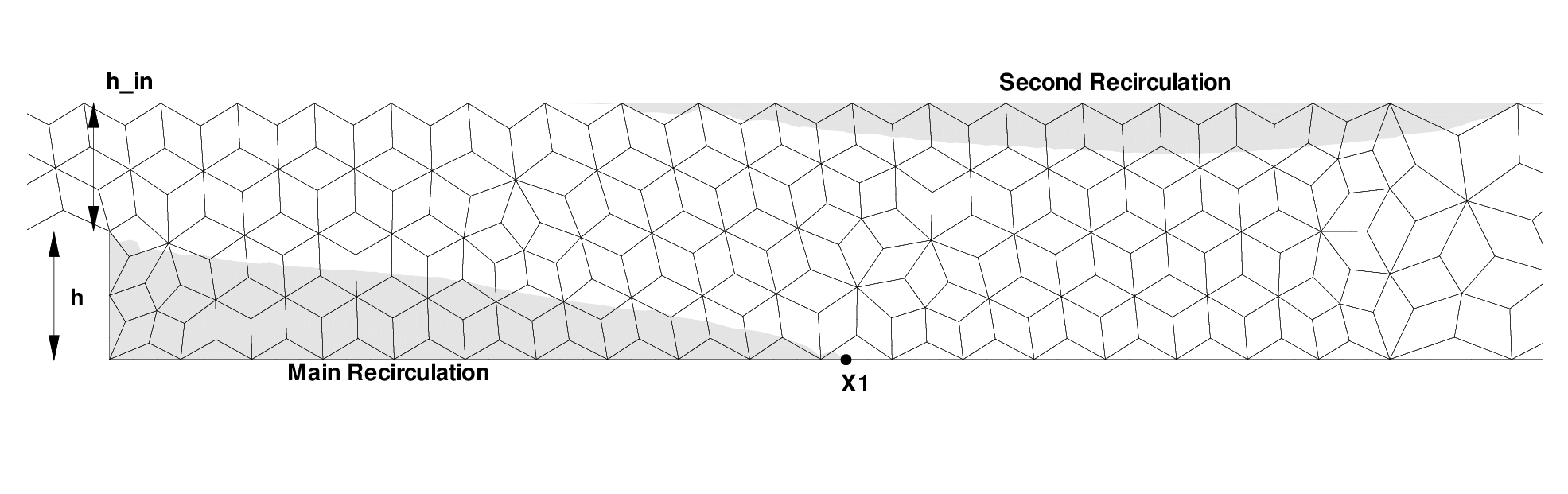}
    \caption{Grid and main notation used for the backward-facing step problem.}
    \label{fig.BFS.1}
		\end{center}
\end{figure}
The fluid flow is driven by a pressure gradient imposed at the left and the right ends of the computational domain. On all the other boundaries, no-slip
wall boundary conditions are imposed.
According to \cite{Armaly1983}, we take $Re=\frac{DU}{\nu}$ where $D=2h_{in}$; $U$ is the mean inlet velocity; $\nu$ is the kinematic viscosity.
\begin{figure}[ht]
    \begin{center}
    \includegraphics[width=0.55\textwidth]{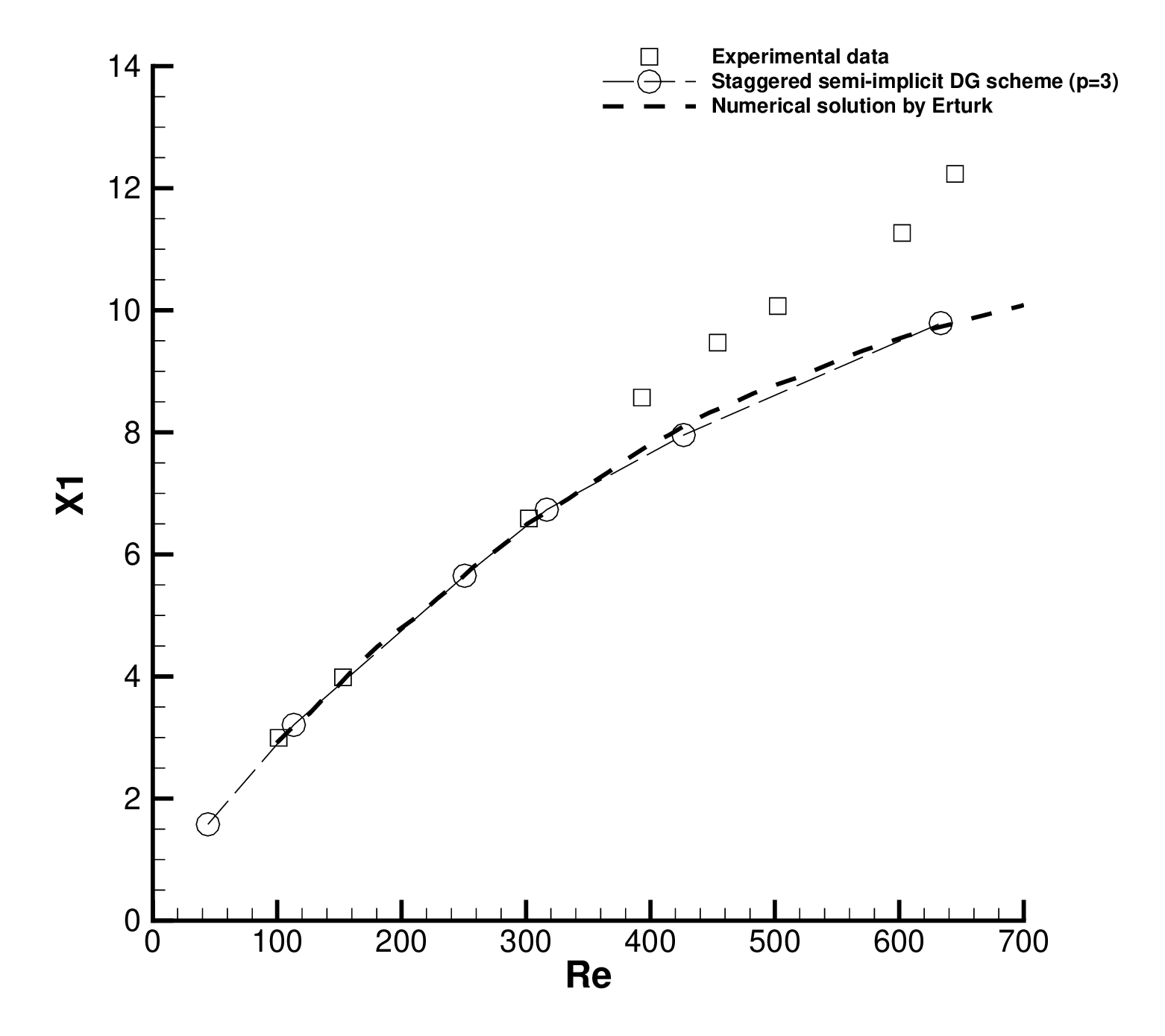} \\
    \caption{Comparison of the experimental data of Armaly et al. \cite{Armaly1983} with the numerical results obtained with the present semi-implicit
		staggered DG scheme and the numerical solution obtained in \cite{Erturk2008} for the reattachment point X1 in the backward-facing step problem.}
    \label{fig.BFS.3}
		\end{center}
\end{figure}
\begin{figure}[ht]
    \begin{center}
    \includegraphics[width=1.0\textwidth]{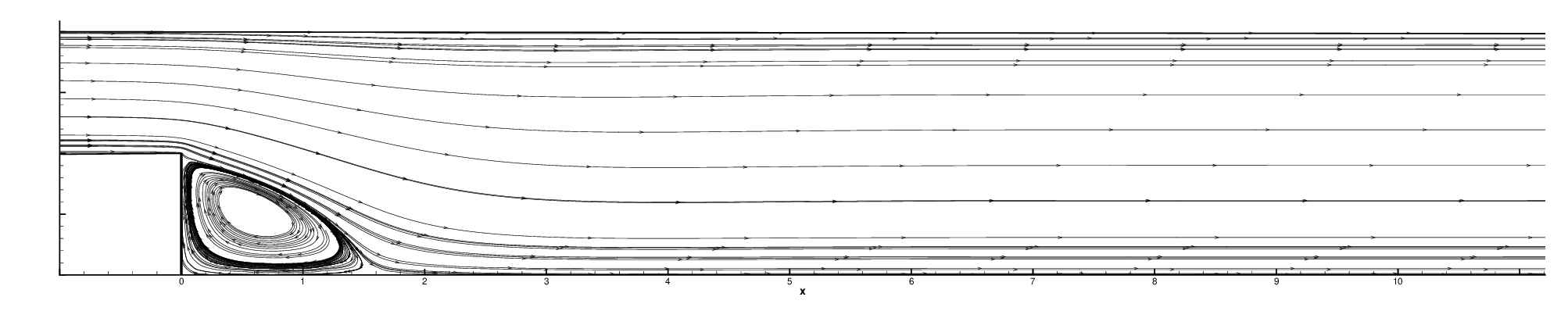} \\
    \includegraphics[width=1.0\textwidth]{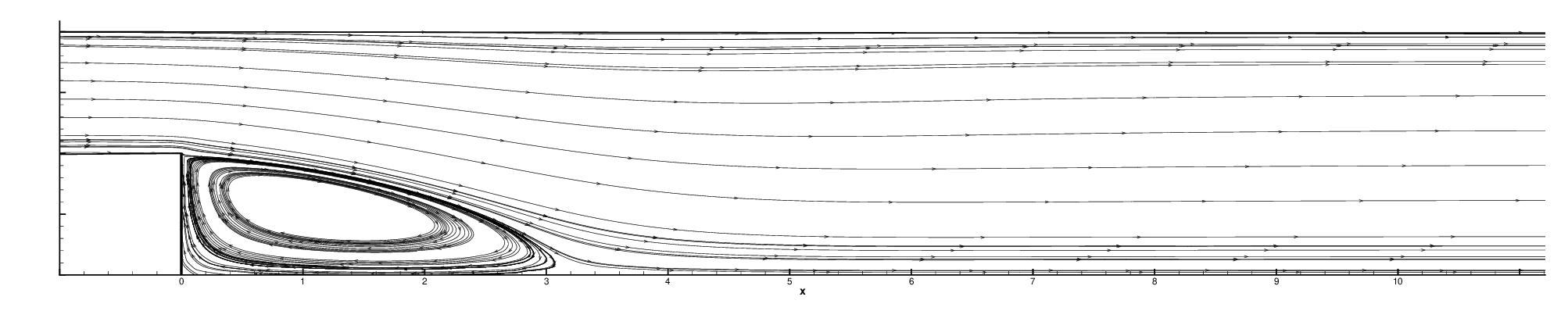} \\
    \includegraphics[width=1.0\textwidth]{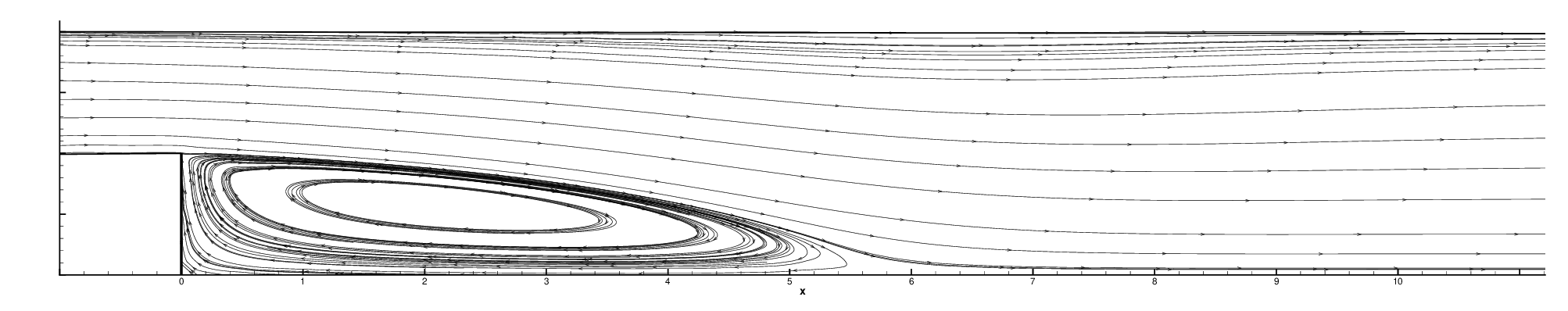} \\
    \includegraphics[width=1.0\textwidth]{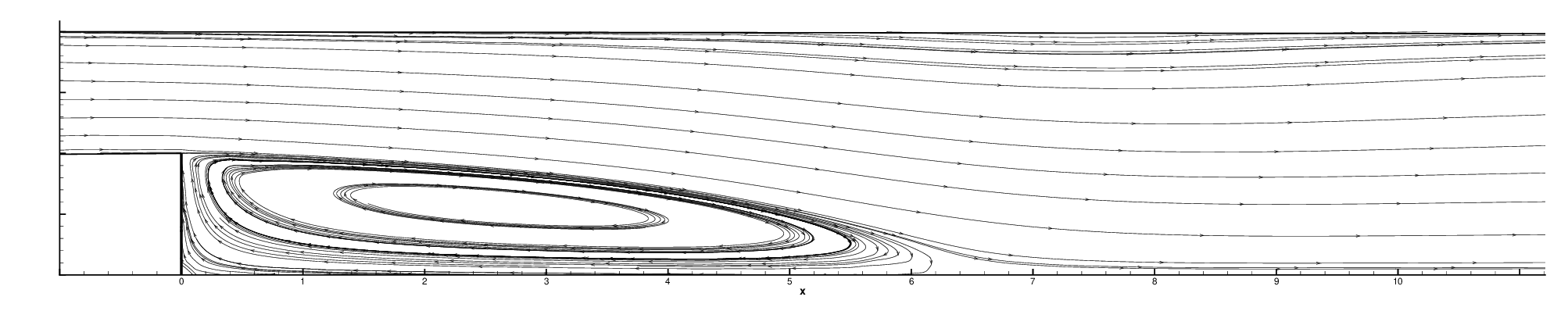} \\
    \includegraphics[width=1.0\textwidth]{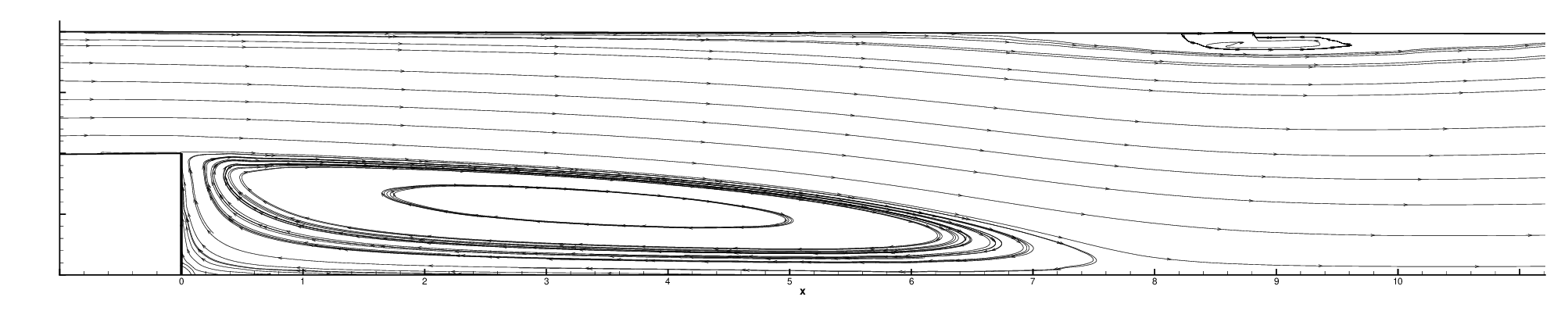} \\
    \includegraphics[width=1.0\textwidth]{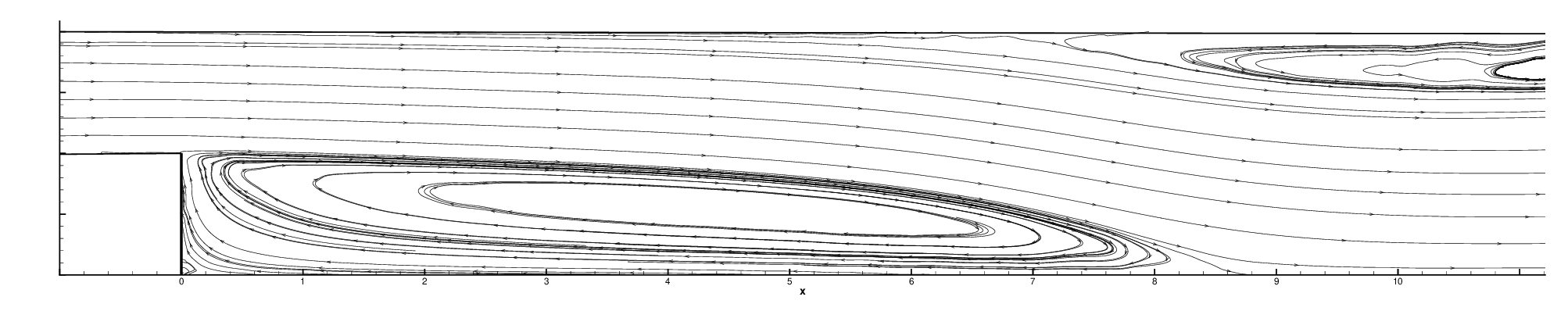}
    \caption{Streamlines at Reynolds numbers $Re=44,113,250, 316, 426$ and $633$ from top to bottom.}
    \label{fig.BFS.2}
		\end{center}
\end{figure}
The computational domain is covered with a total number of $N_i=260$ triangles with characteristic size $h=0.2$ for $x \leq 5$ and $h=0.48$ for $x>5$
(see Figure $\ref{fig.BFS.1}$). Finally we use $p=3$; $\theta=1$ and $\Delta t$ is the one given by the CFL condition for the nonlinear convective term;
$t_{end}=80 s$. Figure $\ref{fig.BFS.2}$ shows the vortices generated at different Reynolds numbers, while in Figure $\ref{fig.BFS.3}$ the main recirculation
point $X1$ is compared with experimental data given by Armaly in \cite{Armaly1983}, and the explicit second-order upwind finite difference scheme introduced
in \cite{Biagioli1998}. A good agreement with the experimental data is shown up to $Re=316$ but, according to \cite{Armaly1983}, the experiment
becomes three dimensional for $Re>400$, so the comparison can be done only up to $Re=400$. Indeed, one can see in Fig. $\ref{fig.BFS.2}$ how the secondary
vortex occurs for $Re=426$, while in the experiments it appears at higher Reynolds numbers (see e.g. \cite{Armaly1983}).
\subsection{Rotational flow past a circular half-cylinder}
Here we consider a rotational flow past a circular half-cylinder. A comparison between numerical and exact analytical solution is possible for incompressible and inviscid fluid, i.e. here we set $\nu=0$. We use the computational setup of Feistauer and Kucera \cite{Feistauer2007}, hence $\Omega=[-5,5]\times [0,5]-\{ \sqrt{x^2+y^2}\leq 0.5 \}$; as boundary conditions we impose the velocity at the left boundary; homogeneous Neumann boundary conditions on the top and right boundaries and inviscid wall at the bottom and the surface of the half-cylinder. The farfield velocity field is given by $u=y$ and $v=0$. The exact analytical solution to
this problem was found by Fraenkel in \cite{Fraenkel1961}. For the present test we choose $p=3$; $\Delta t$ is set according to \eref{eq:CFLC} and we cover
$\Omega$ with $N_i=800$ triangles, using only $6$ triangles to describe the half-cylinder. Curved isoparametric elements are considered in order to represent
the geometry of the half-cylinder properly. As initial conditions we impose $p(x,y,0)=1$; $u(x,y,0)=y$ and $v(x,y,0)=0$.
\begin{figure}[ht]
    \begin{center}
    \includegraphics[width=0.60\textwidth]{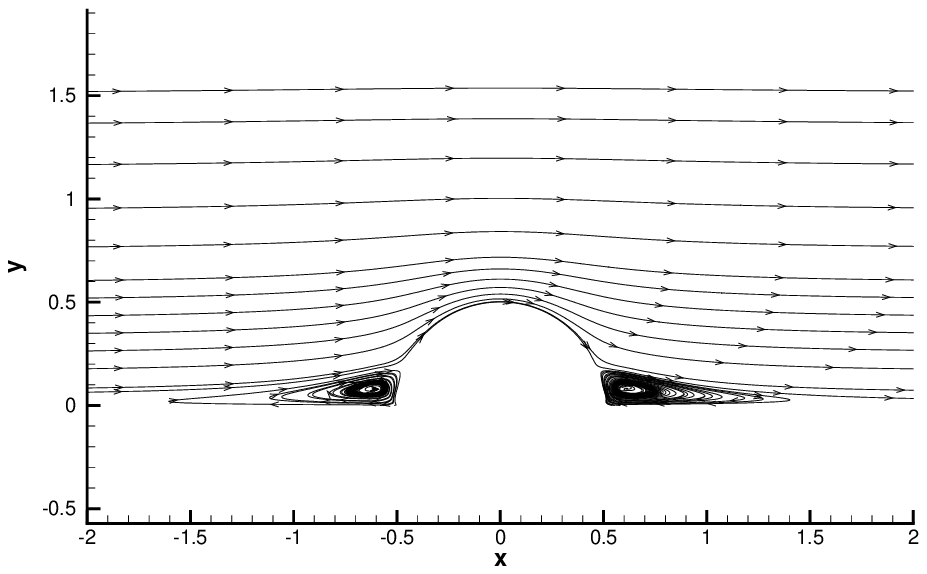}
    \includegraphics[width=0.35\textwidth]{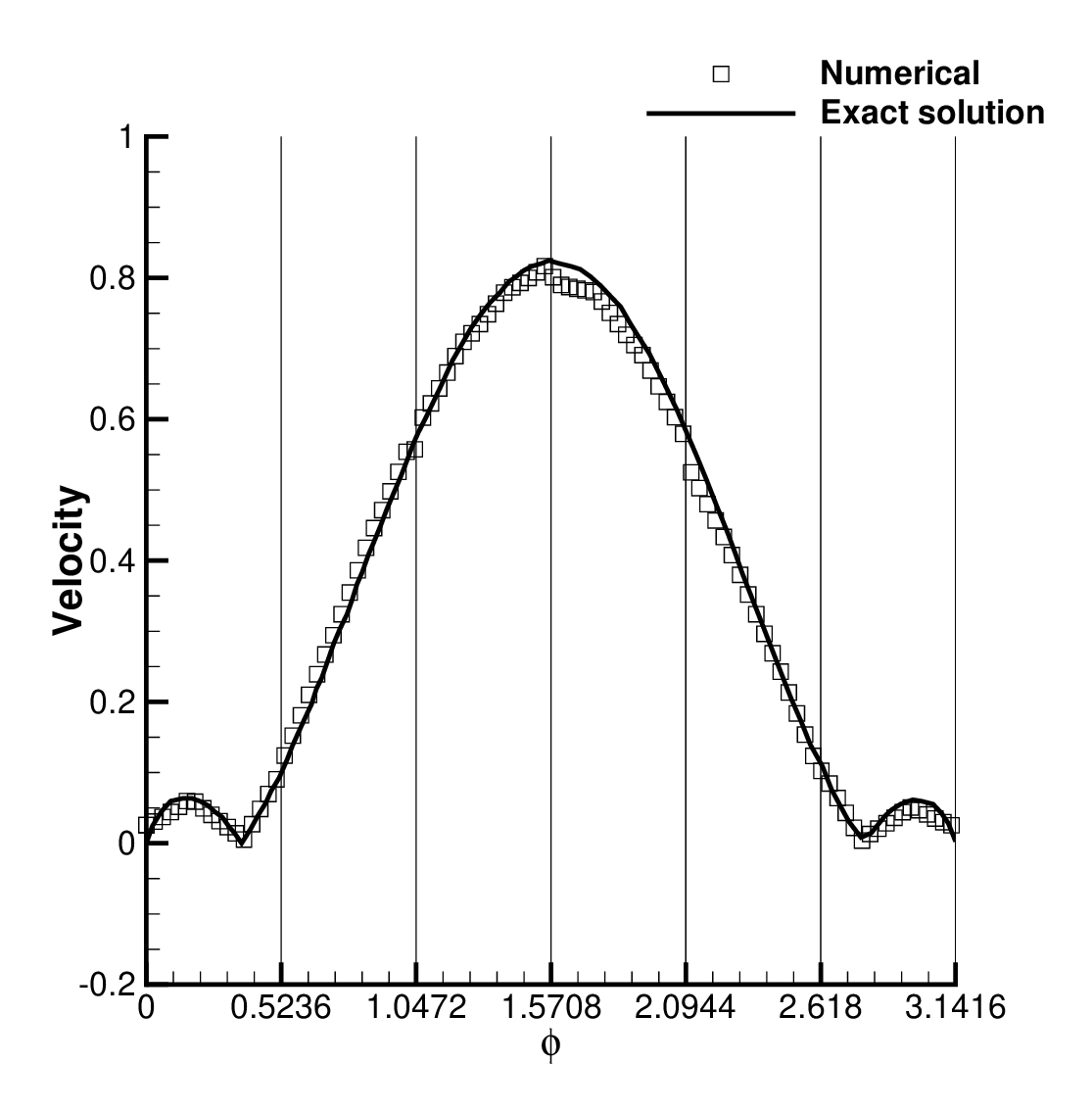}
    \caption{Rotational inviscid flow past a circular half-cylinder. Left: Streamlines. Right: Analytical and numerical results for $r=0.5$. The vertical lines show the dimension of the six curved elements that cover the half-cylinder.}
    \label{fig:HC1}
		\end{center}
\end{figure}
Two vortices appear near the half-cylinder (see Fig. \ref{fig:HC1} left), while a comparison between analytical and numerical velocity magnitude on the cylinder surface (i.e. $r=0.5$) is shown on the right of Fig. \ref{fig:HC1}. A good agreement between analytical and numerical results is obtained also with a very coarse grid. An important remark is that for this test problem the use of isoparametric elements is crucial, as previously shown for inviscid flow past a circular
cylinder by Bassi and Rebay in \cite{BassiRebay}.
\subsection{Flow over a circular cylinder}
In this section we consider the flow over a circular cylinder. Also in this case, the use of the isoparametric approach is mandatory to represent the geometry of the cylinder wall, see \cite{BassiRebay,2DSIUSW}.
In particular, two cases are considered: first, an inviscid flow around the cylinder is assumed in order to obtain a steady potential flow; finally, the complete viscous case is considered in order to get the unsteady von Karman vortex street. For the first
case a sufficiently large domain $\Omega=[-8,8]\times[-8,8]-\{ \sqrt{x^2+y^2}\leq 1 \}$ is employed. The exact solution for this case is known and reads:
\begin{equation*}
    u_r(r,\varphi)=\bar{u}\left( 1- \frac{R_c^2}{r^2} \right)\cos(\varphi), \qquad
    u_\varphi(r,\varphi)=-\bar{u}\left( 1+ \frac{R_c^2}{r^2} \right)\sin(\varphi),
\end{equation*}
\begin{equation} 		
    p=\frac{1}{2}\bar{u}^2\left( \frac{2R_c^2}{r^2}\cos(2\varphi)-\frac{R_c^4}{r^4} \right),
\label{eq:VK1}
\end{equation}
where $\bar{u}$ is the inflow velocity; $R_c$ is the cylinder radius; $u_r$ and $u_\varphi$ are the radial and angular components of the velocity, respectively. An initial condition $\vec{v}(x,y,0)=(\bar{u},0)$ is used, while the exact velocity distribution is taken as the external boundary condition. An inviscid wall  boundary condition is imposed on the cylinder. For the present test $\bar{u}=0.01$; $R_c=1$; $\nu=0$; $p=3$; $\theta=0.6$; $\Delta t$ is the one taken according to the CFL restriction \eref{eq:CFLC}; $t_{end}=10$. The domain $\Omega$ is covered with a total number of $N_i=1464$ triangles and an isoparametric approach is considered to represent the cylinder wall properly. Figure $\ref{fig.VKJ}$ shows the streamlines and the pressure contours obtained at $t=10$ as well as
the comparison between exact and numerical solution at several radii. A very good agreement between exact and numerical solution is observed.
%
%
\begin{figure}[ht]
    \begin{center}
    \includegraphics[width=0.45\textwidth]{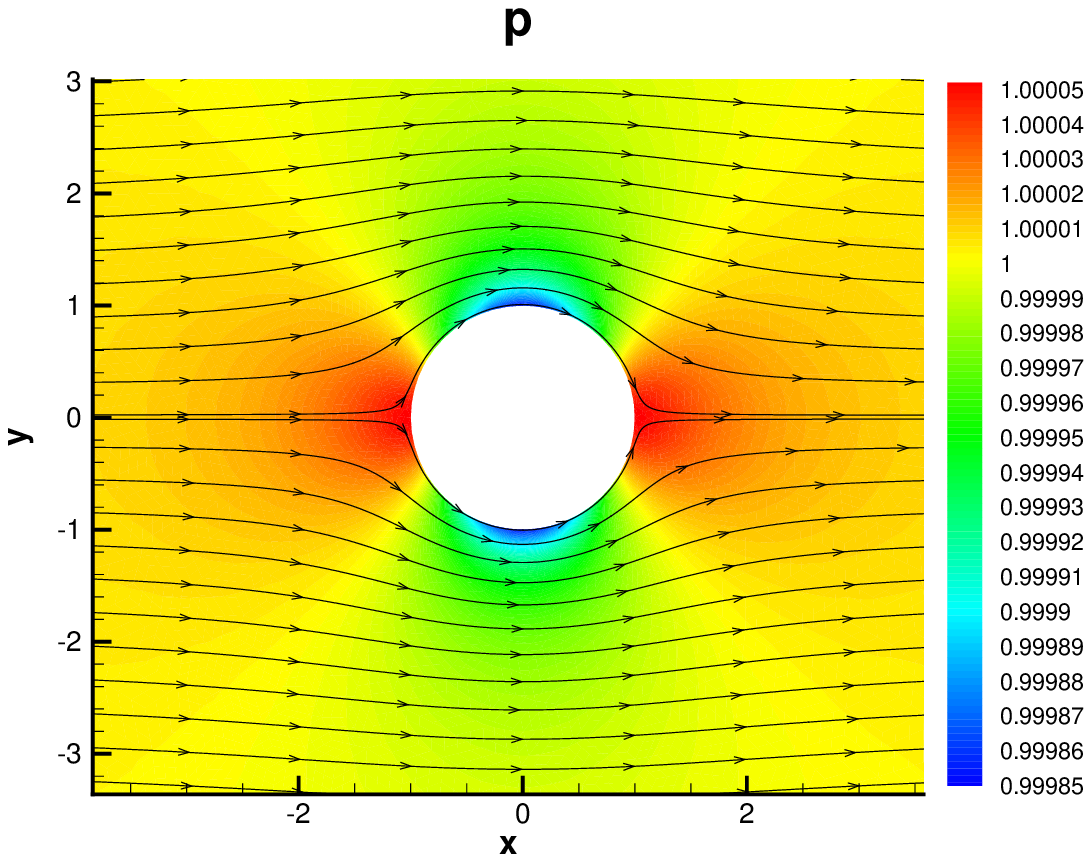}
    \includegraphics[width=0.45\textwidth]{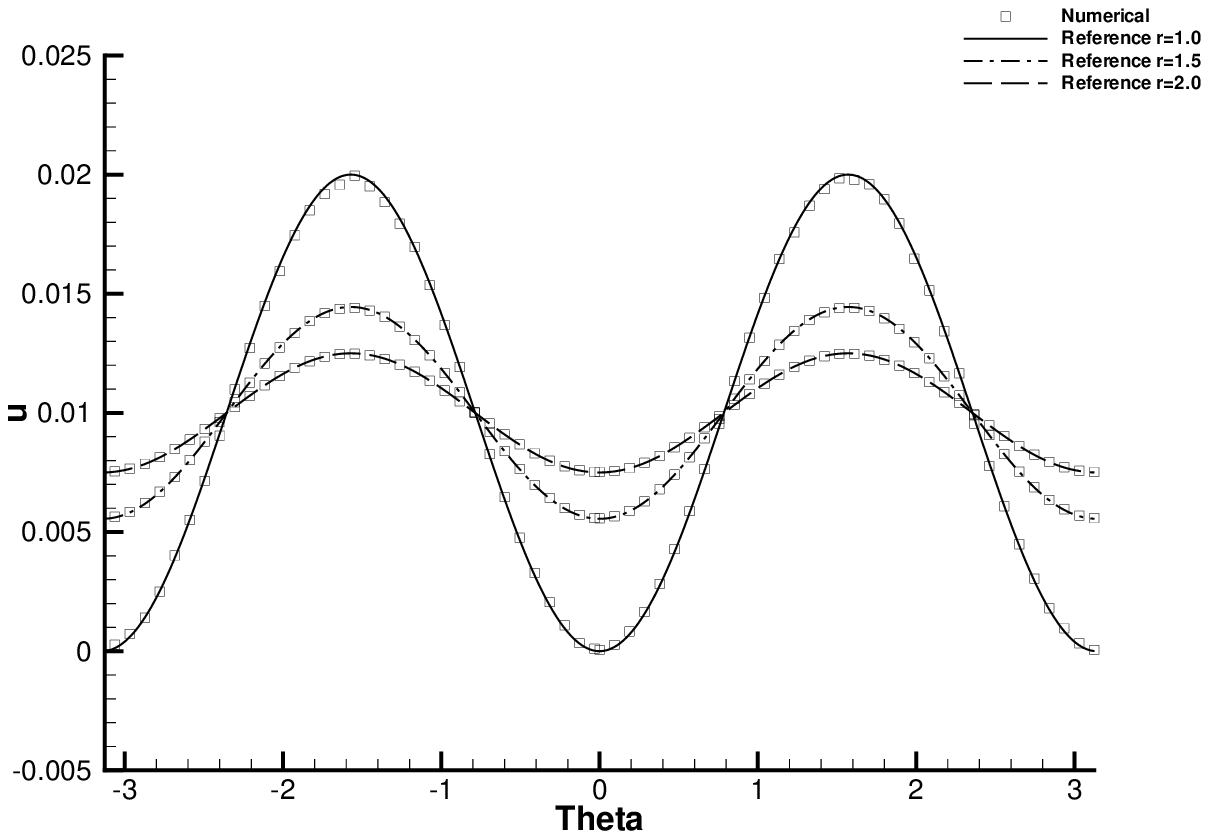}
    \includegraphics[width=0.45\textwidth]{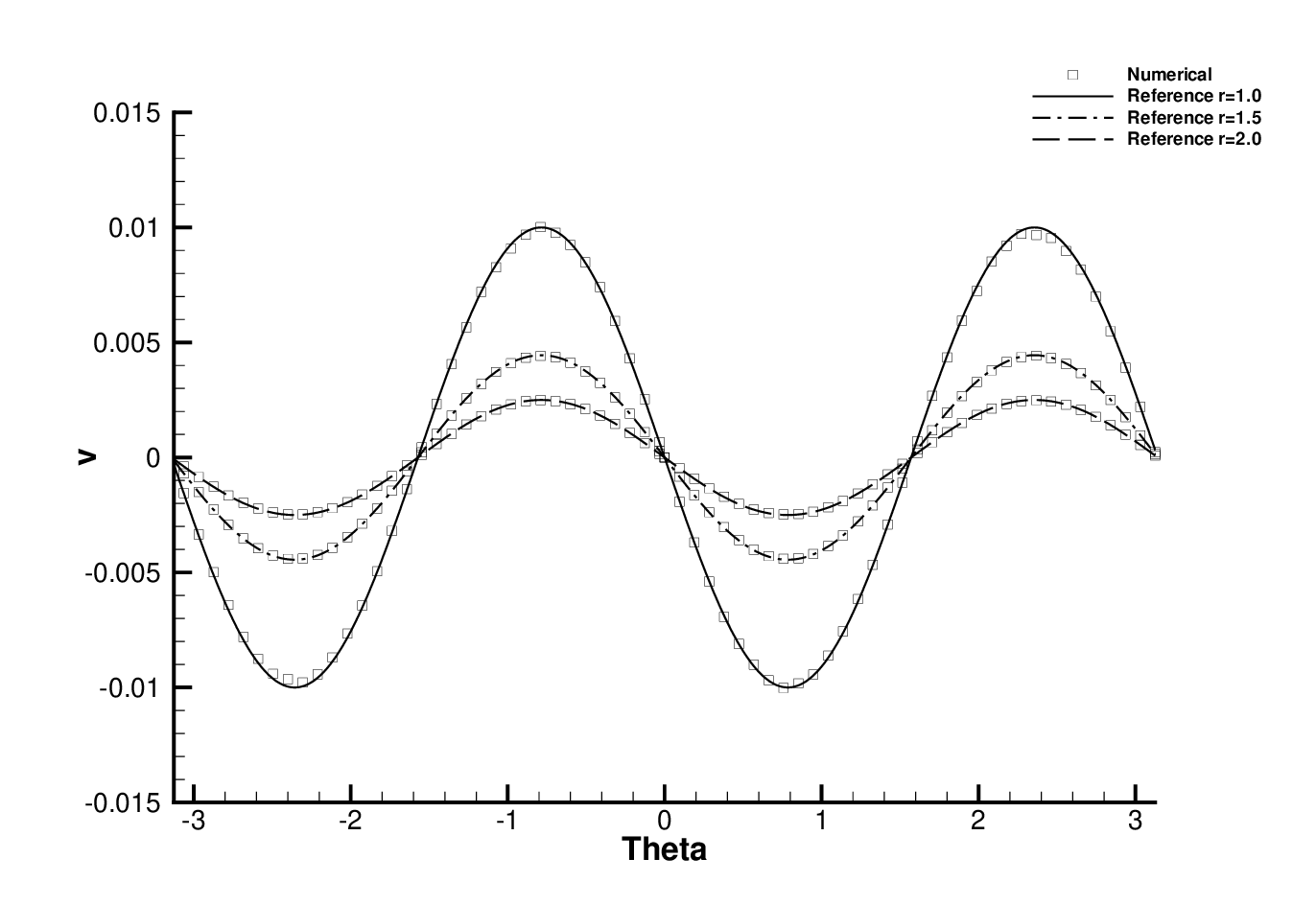}
    \includegraphics[width=0.45\textwidth]{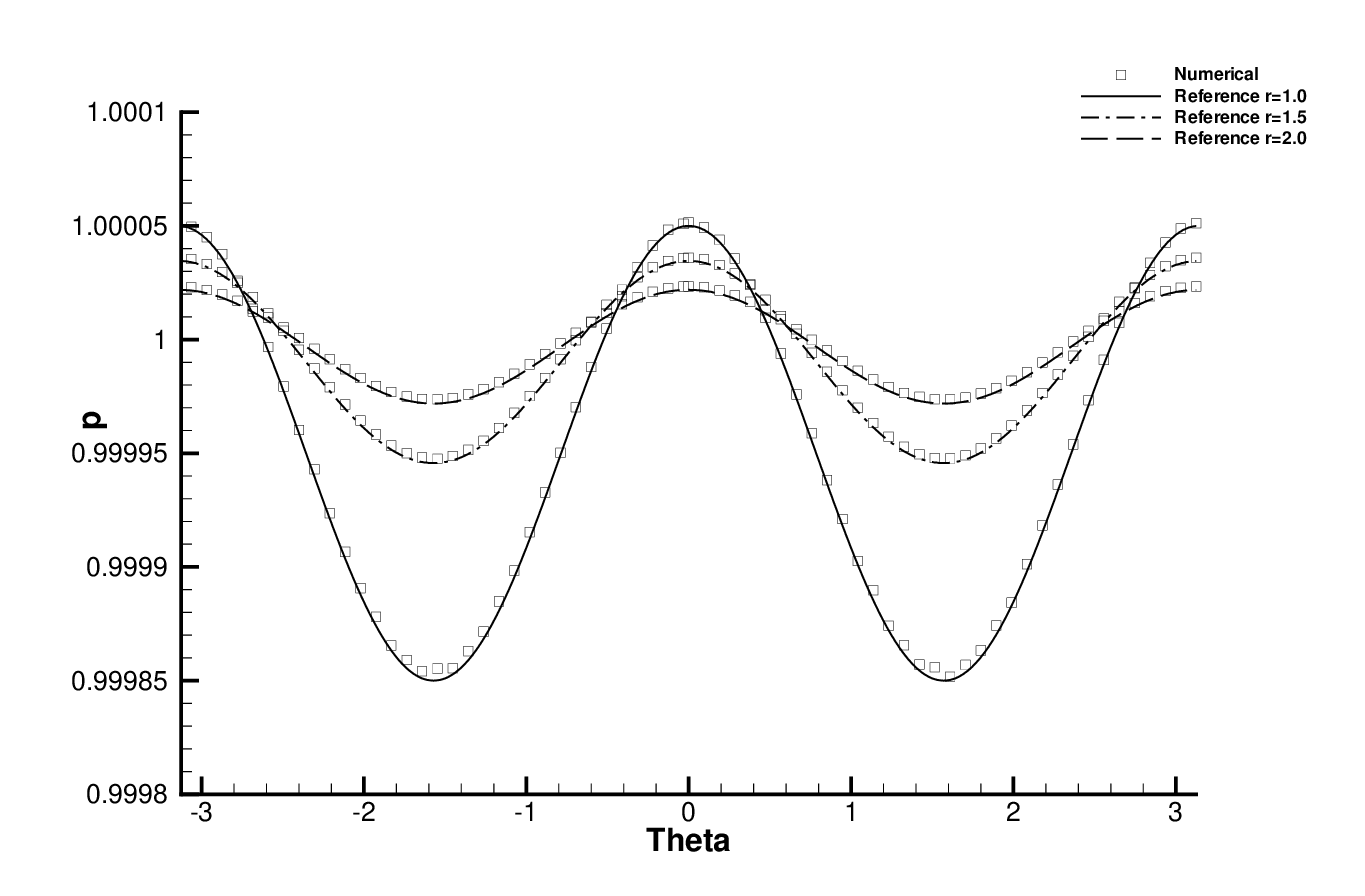}
    \caption{Steady flow of an inviscid incompressible fluid around a circular cylinder. On the top left: streamlines and pressure contours at $t_{end}=10$; Numerical and exact solution at $r=1.0$, $r=1.5$ and $r=2.0$ for the velocity components $u$, $v$ and pressure $p$ from top right to the bottom right, respectively.}
    \label{fig.VKJ}
		\end{center}
\end{figure}
We consider now the fully viscous case in order to show the formation of the von Karman vortex street. Two domains are considered here: $\Omega_1=[-20, 80]\times [-20, 20]$ covered with a $N_i=1702$ triangles; and $\Omega_2=[-5, 30]\times [-10, 10]$ covered with a $N_i=1706$ triangles. As initial condition we set $\vec{v}(x,y,0)=(\bar{u},0)$; $\theta=0.6$; and $\bar{u}=0.5$. Different viscosity coefficients are used in order to obtain different Reynolds numbers. For the present test we use $\Delta t$ according to \eref{eq:CFLC}; $p=3$; $\theta=1$. The velocity $(\bar{u},0)$ is prescribed at the left boundary while homogeneous Neumann boundary conditions are imposed on the other external edge of the domains. Finally viscous wall boundary condition is imposed on the cylinder surface.
\begin{figure}[ht]
    \begin{center}
    \includegraphics[width=0.6\textwidth]{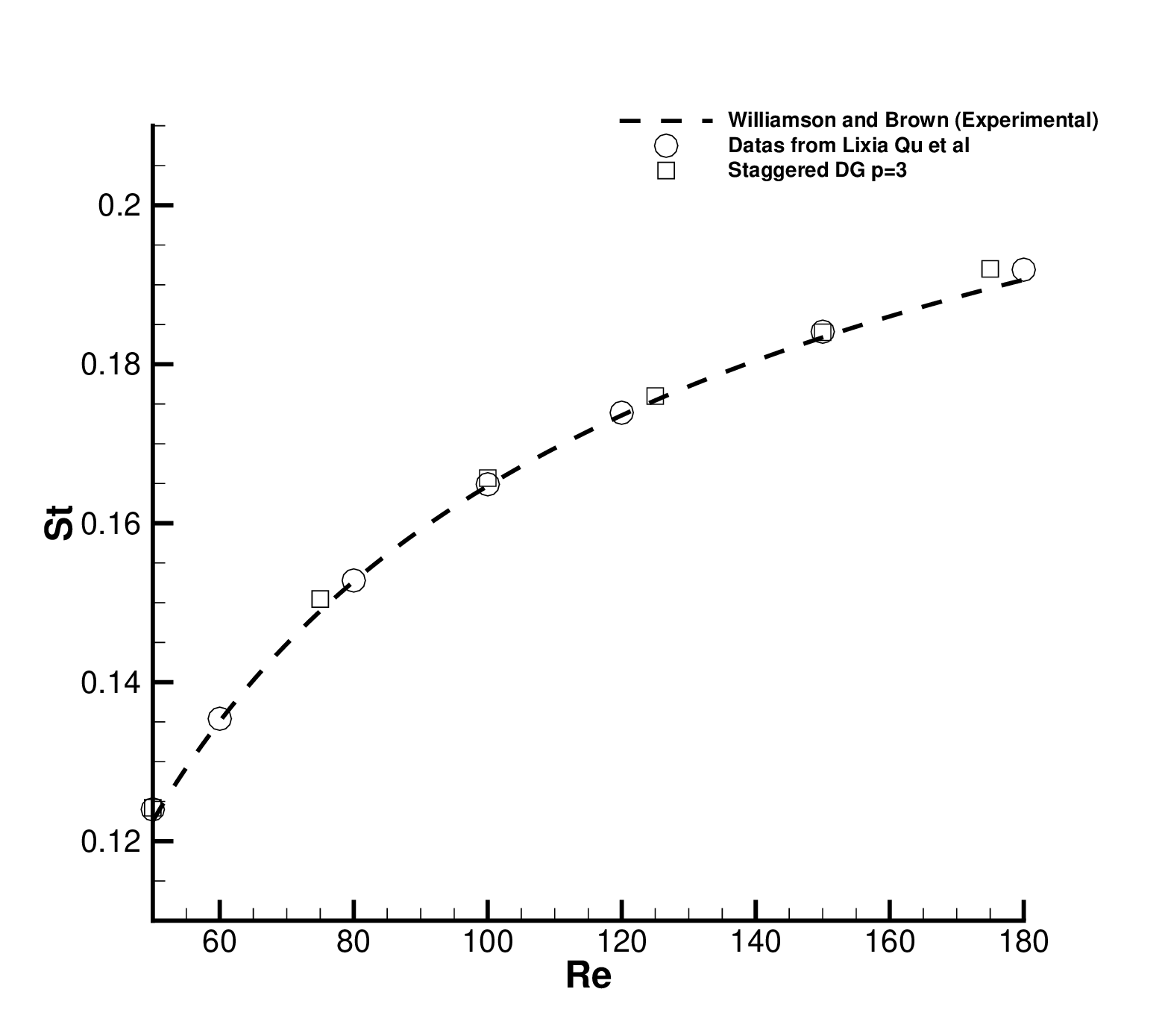} \\
    \caption{Strouhal-Reynolds number relationship for the present method,
		the method of Qu et al. \cite{Lixia2013} and experimental data of Williamson and Brown \cite{Williamson1998}. }
    \label{fig.VK4}
		\end{center}
\end{figure}
Figure $\ref{fig.VK4}$ shows the obtained relationship between the Strouhal number, computed as $St=\frac{2rf}{u_{\infty}}$, the numerical results given by Qu et al (see \cite{Lixia2013}) and the experimental law given in \cite{Williamson1998}. The simulations are performed on the domain $\Omega_1$. The numerical results fit  well the experimental data and the numerical reference solution up to $Re=150$. Better results can be obtained by further enlarging the computational domain.
\begin{figure}[ht]
    \begin{center}
    \includegraphics[width=1.05\textwidth]{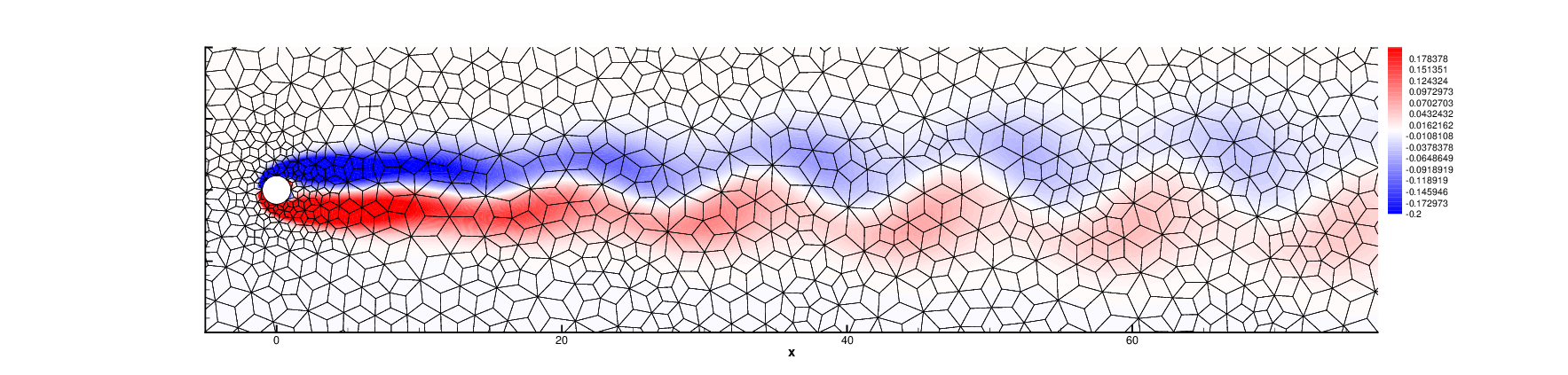}
    \includegraphics[width=1.05\textwidth]{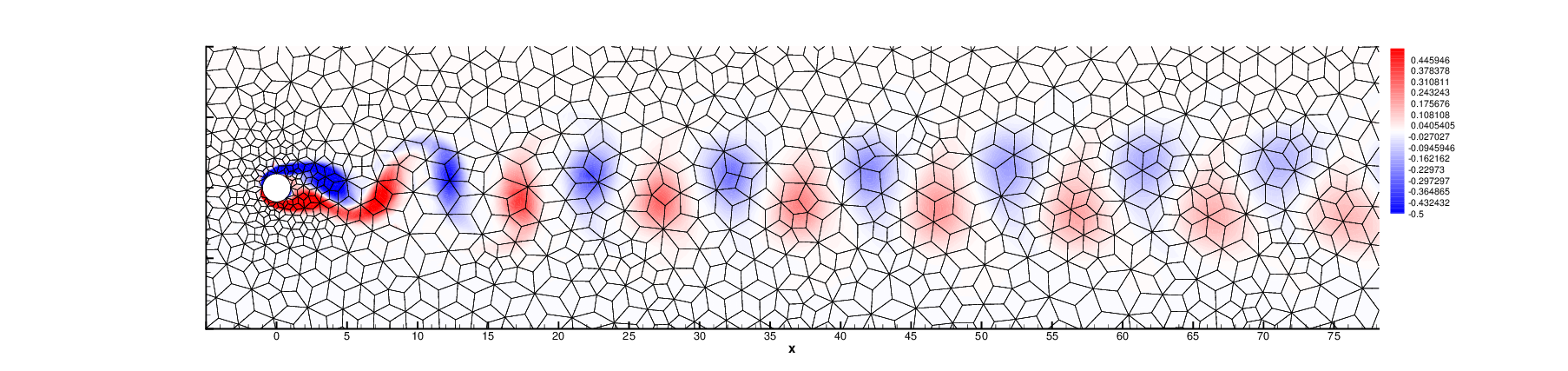}
    \caption{Dual mesh and vorticity contours of the von Karman vortex street generated at time $t=500$ for $Re=50$ (top) and $Re=125$ (bottom).}
    \label{fig.VK5}
		\end{center}
\end{figure}
The velocity field and the vorticity show different structures when low and high Reynolds numbers are considered. The vorticity contours are shown in Figure
$\ref{fig.VK5}$ for $Re=50$ and $Re=125$ at time $t=500$. In the case of $Re=125$ the von Karman vortex street is fully developed while, for $Re=50$, the two  initial vortices remain present behind the cylinder for a longer time. This is due to the low value of the Reynolds number, taken close to the limit of $Re=40$
for the generation of the vortex street.
\begin{figure}[ht]
    \begin{tabular}{cc}
    \includegraphics[width=0.45\textwidth]{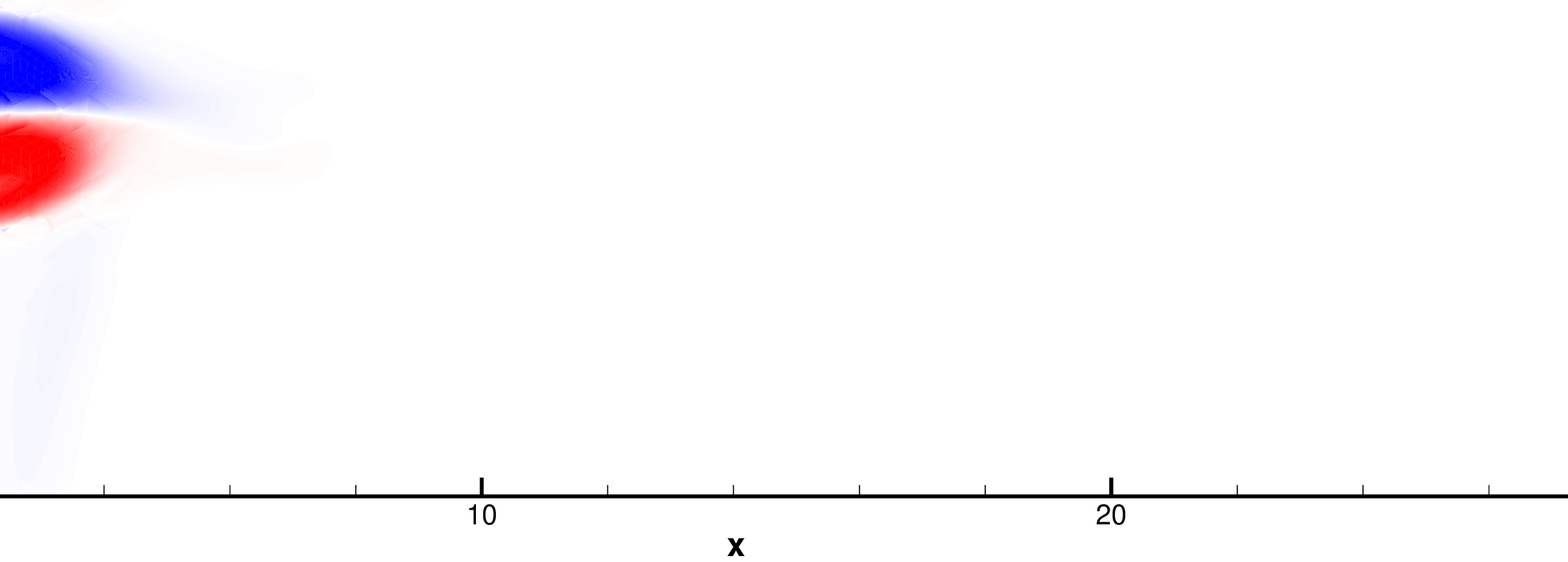} &
    \includegraphics[width=0.45\textwidth]{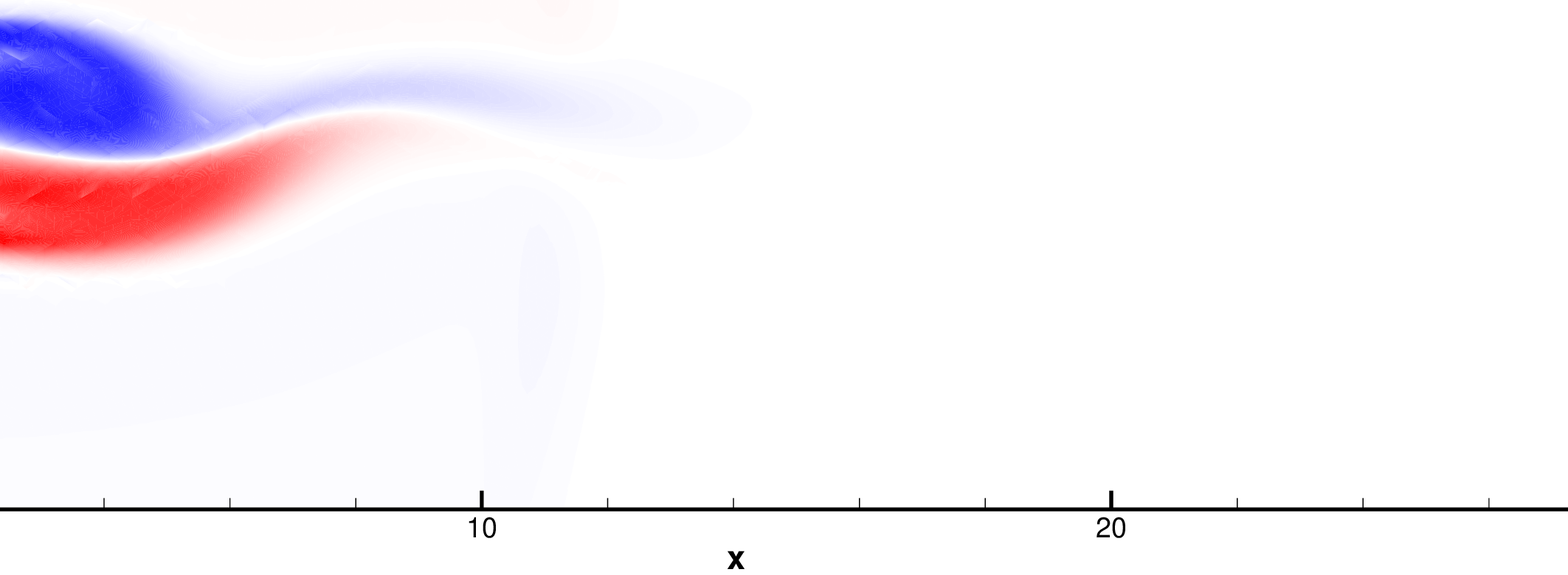} \\
    \includegraphics[width=0.45\textwidth]{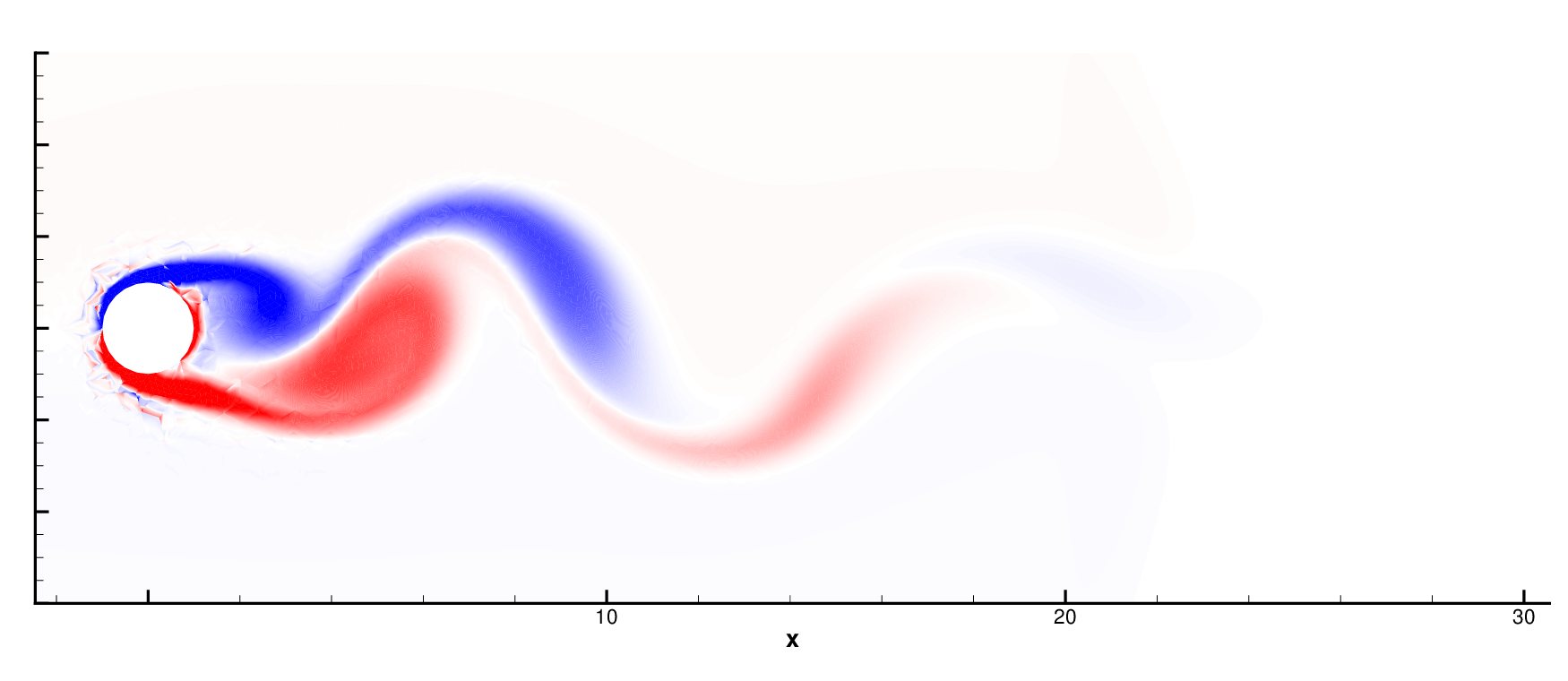} &
    \includegraphics[width=0.45\textwidth]{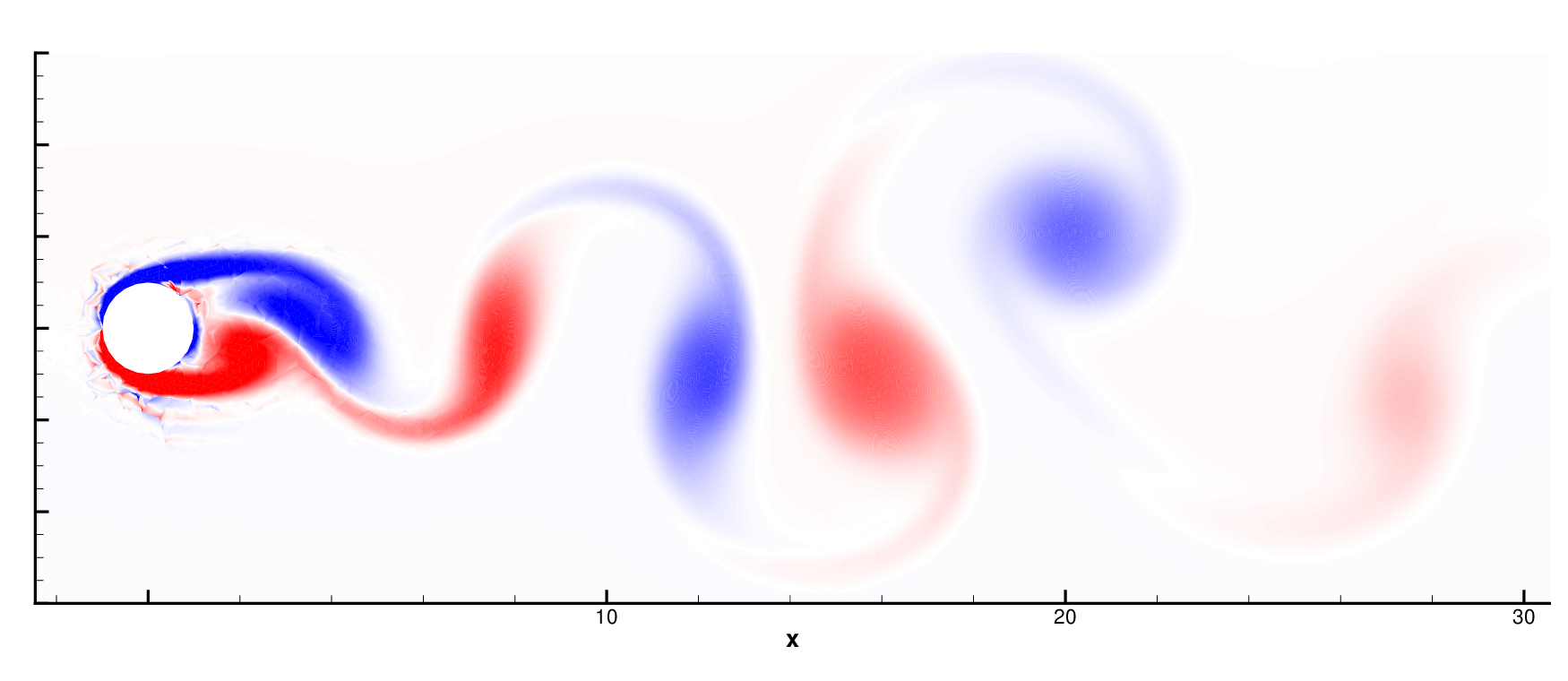}
		\end{tabular}
    \caption{Temporal evolution of the vorticity profile for $t=15$, $t=30$, $t=50$, $t=75$ from top left to bottom right at $Re=200$.}
    \label{fig.VK3}
\end{figure}

\begin{figure}[ht]
    \begin{tabular}{cc}
    \includegraphics[width=0.48\textwidth]{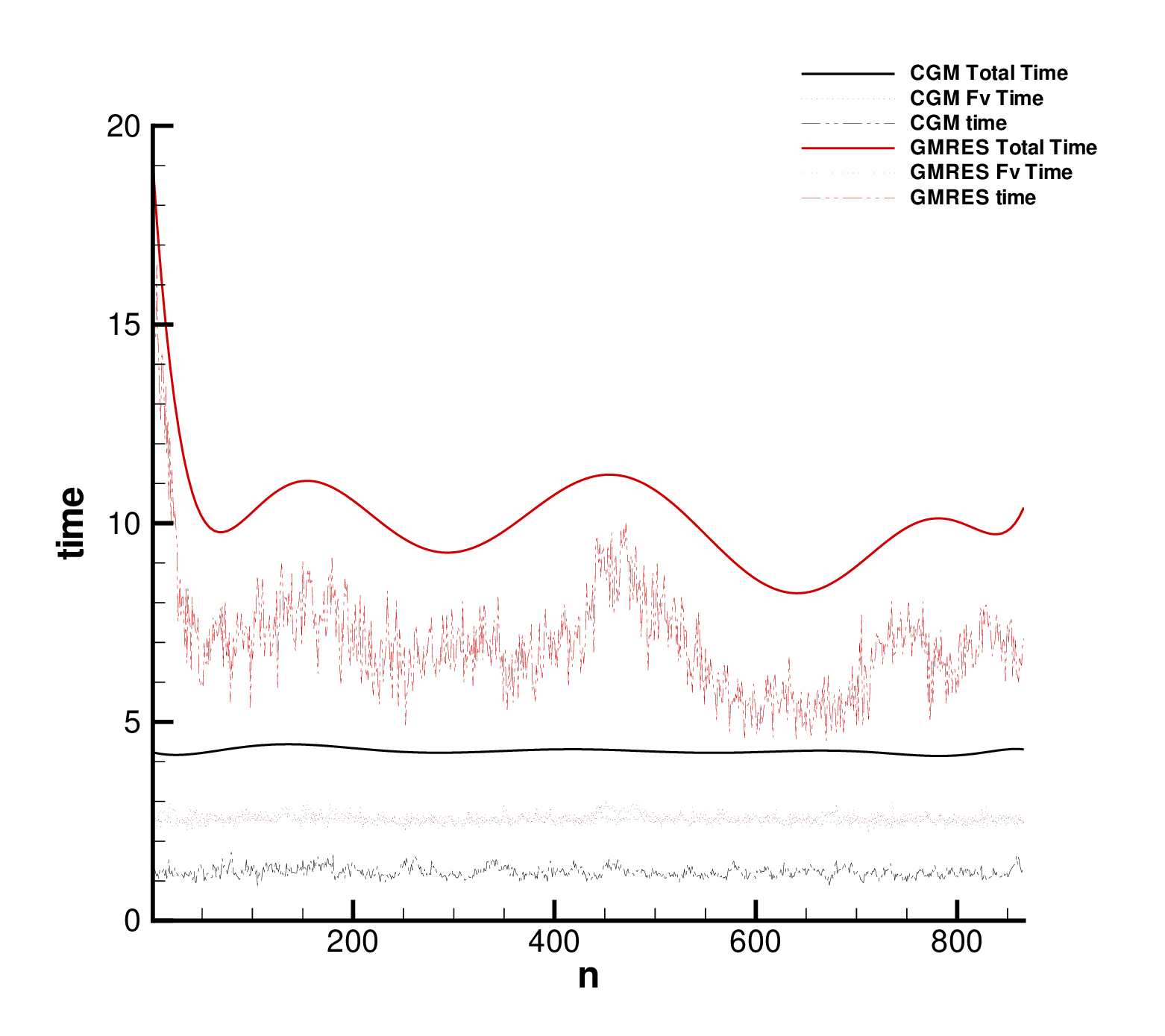} &
    \includegraphics[width=0.48\textwidth]{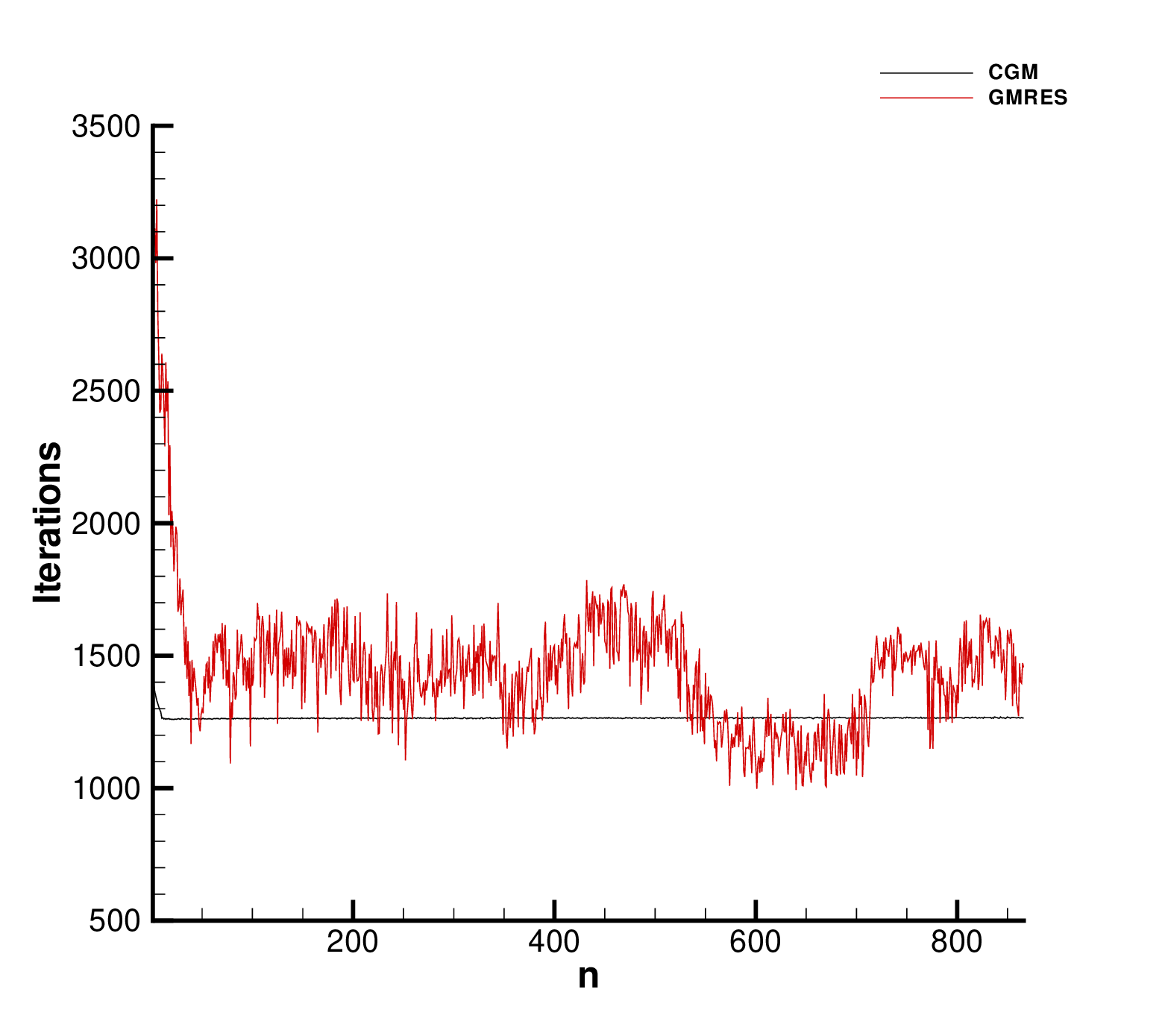} \\
		\end{tabular}
    \caption{Left: comparison between the CPU time for the GMRES method and the CG method. Right: total number of iterations for the GMRES against the number of iterations for the CG method.}
    \label{fig.VK4_2}
\end{figure}

The time evolution of the generation of the von Karman vortex street is presented at several times for $Re=200$ on $\Omega_2$ in Figure $\ref{fig.VK3}$.

Finally, in Figure $\ref{fig.VK4_2}$ we report a comparison between the computational time needed per time step for the main parts of the algorithm presented in this paper
up to the time $t=10s$ using $Re=100$ on $\Omega_1$ if we employ a GMRES method or the cheaper CG method for the solution of the linear system. Note that since our particular
semi-implicit DG discretization of the incompressible Navier-Stokes equations on staggered grids leads to a symmetric and positive-definite linear system, we can employ the CG
method. This is not always the case for DG schemes applied to the incompressible Navier-Stokes equations since some formulations may also lead to non-symmetric linear systems.

The time required to compute the convective-viscous term represents in the second case the main computational effort. Using the GMRES algorithm the computational time needed to
solve the linear system increases a lot compared to the CG method and becomes the main cost of the algorithm. In particular, the mean time to solve the system using the GMRES
algorithm is, for this test, $6.2 s$ while using the CG method is only about $1.0 s$. For all tests, the tolerance for solving the linear system was set to $tol=10^{-12}$.
We underline that for a fair comparison of the two methods, no preconditioners have been used and that faster convergence can be obtained by using a proper preconditioner for
each iterative solver.

\section{Conclusions}
\label{sec.concl}
A new, spatially high order accurate semi-implicit DG scheme for the solution of the incompressible Navier-Stokes equations on staggered unstructured
non-orthogonal curved meshes has been proposed. The high order of accuracy in space was verified and compared with reference solutions for polynomial
degrees up to $p=3$. The numerical results agree very well with the reference data for all test cases considered in this paper. The proposed numerical
method reduces to a classical semi-implicit finite-volume and finite-difference scheme on staggered meshes for $p=0$. Furthermore, the use of matrices
that depend only on the geometry and on the polynomial degree and hence can be precomputed before runtime, leads to a computationally efficient scheme.
In addition, the resulting main matrix results symmetric and positive definite for appropriate boundary conditions. This allows to use fast iterative
methods for the solution of the sparse linear system with a significant gain in terms of computational time.

Future research will concern the extension of the scheme to high order of accuracy also in time using a space-time DG approach as well as the extension
to the fully three-dimensional case on unstructured tetrahedral meshes.

\section*{Acknowledgments}
M.D. was funded by the European Research Council (ERC) under the European Union's Seventh Framework Programme (FP7/2007-2013) within the research
project \textit{STiMulUs}, ERC Grant agreement no. 278267.
\clearpage
\bibliography{SIDG}

\begin{thebibliography}{10}

\bibitem{Armaly1983}
B.F. Armaly, F.~Durst, J.C.F. Pereira, and B.~Schonung.
\newblock Experimenta and theoretical investigation on backward-facing step
  flow.
\newblock {\em Journal of Fluid Mechanics}, 127:473--496, 1983.

\bibitem{Bassi2006}
F.~Bassi, A.~Crivellini, D.A.~Di Pietro, and S.~Rebay.
\newblock On a robust discontinuous galerkin technique for the solution of
  compressible flow.
\newblock {\em Journal of Computational Physics}, 218:208--221, 2006.

\bibitem{Bassi2007}
F.~Bassi, A.~Crivellini, D.A.~Di Pietro, and S.~Rebay.
\newblock {An implicit high-order discontinuous Galerkin method for steady and
  unsteady incompressible flows}.
\newblock {\em Computers and Fluids}, 36:1529--1546, 2007.

\bibitem{BassiRebay}
F.~Bassi and S.~Rebay.
\newblock A high-order accurate discontinuous finite element method for the
  numerical solution of the compressible {Navier}-{Stokes} equations.
\newblock {\em Journal of Computational Physics}, 131:267--279, 1997.

\bibitem{Artzi}
M.~Ben-Artzi and J.~Falcovitz.
\newblock A second-order godunov-type scheme for compressible fluid dynamics.
\newblock {\em Journal of Computational Physics}, 55:1--32, 1984.

\bibitem{Bermudez1998}
A.~Bermudez, A.~Dervieux, J.A. Desideri, and M.E. Vazquez.
\newblock Upwind schemes for the two--dimensional shallow water equations with
  variable depth using unstructured meshes.
\newblock {\em Computer Methods in Applied Mechanics and Engineering},
  155:49--72, 1998.

\bibitem{Vazquez2014}
A.~Berm\'udez, J.L. Ferr\'in, L.~Saavedra, and M.E. V\'azquez-Cend\'on.
\newblock {A projection hybrid finite volume/element method for low-Mach number
  flows}.
\newblock {\em Journal of Computational Physics}, 271:360--378, 2014.

\bibitem{Biagioli1998}
F.~Biagio.
\newblock Calculation of laminar flows with second-order schemes and collocated
  variable arrangement.
\newblock {\em International Journal for Numerical Methods in Fluids},
  26:887--905, 1998.

\bibitem{Blasius1908}
H.~Blasius.
\newblock Grenzschichten in {Fl\"ussigkeiten} mit kleiner {Reibung}.
\newblock {\em Z. Math. Physik}, 56:1--37, 1908.

\bibitem{SUPG}
A.N. Brooks and T.J.R. Hughes.
\newblock {Stream-line upwind/Petrov Galerkin formulstion for convection
  dominated flows with particular emphasis on the incompressible Navier-Stokes
  equation}.
\newblock {\em Computer Methods in Applied Mechanics and Engineering},
  32:199--259, 1982.

\bibitem{BrugnanoCasulli}
L.~Brugnano and V.~Casulli.
\newblock Iterative solution of piecewise linear systems.
\newblock {\em SIAM Journal on Scientific Computing}, 30:463--472, 2007.

\bibitem{BrugnanoCasulli2}
L.~Brugnano and V.~Casulli.
\newblock Iterative solution of piecewise linear systems and applications to
  flows in porous media.
\newblock {\em SIAM Journal on Scientific Computing}, 31:1858--1873, 2009.

\bibitem{Casulli1999}
V.~Casulli.
\newblock A semi-implicit finite difference method for non-hydrostatic
  free-surface flows.
\newblock {\em International Journal for Numerical Methods in Fluids},
  30:425--440, 1999.

\bibitem{Casulli2009}
V.~Casulli.
\newblock A high-resolution wetting and drying algorithm for free-surface
  hydrodynamics.
\newblock {\em International Journal for Numerical Methods in Fluids},
  60:391--408, 2009.

\bibitem{CasulliVOF}
V.~Casulli.
\newblock {A semi--implicit numerical method for the free--surface
  Navier--Stokes equations}.
\newblock {\em International Journal for Numerical Methods in Fluids},
  74:605--622, 2014.

\bibitem{CasulliCattani}
V.~Casulli and E.~Cattani.
\newblock Stability, accuracy and efficiency of a semi-implicit method for
  three-dimensional shallow water flow.
\newblock {\em Computers \& Mathematics with Applications}, 27:99--112, 1994.

\bibitem{CasulliCheng1992}
V.~Casulli and R.~T. Cheng.
\newblock Semi-implicit finite difference methods for three--dimensional
  shallow water flow.
\newblock {\em International Journal for Numerical Methods in Fluids},
  15:629--648, 1992.

\bibitem{CasulliStelling2011}
V.~Casulli and G.~S. Stelling.
\newblock Semi-implicit subgrid modelling of three-dimensional free-surface
  flows.
\newblock {\em International Journal for Numerical Methods in Fluids},
  67:441--449, 2011.

\bibitem{CasulliWalters2000}
V.~Casulli and R.~A. Walters.
\newblock An unstructured grid, three--dimensional model based on the shallow
  water equations.
\newblock {\em International Journal for Numerical Methods in Fluids},
  32:331--348, 2000.

\bibitem{CasulliZanolli}
V.~Casulli and P.~Zanolli.
\newblock High resolution methods for multidimensional advection--diffusion
  problems in free--surface hydrodynamics.
\newblock {\em Ocean Modelling}, 10:137--151, 2005.

\bibitem{CasulliZanolli2012}
V.~Casulli and P.~Zanolli.
\newblock Iterative solutions of mildly nonlinear systems.
\newblock {\em Journal of Computational and Applied Mathematics},
  236:3937--3947, 2012.

\bibitem{chorin1}
A.J. Chorin.
\newblock A numerical method for solving incompressible viscous flow problems.
\newblock {\em Journal of Computational Physics}, 2:12--26, 1967.

\bibitem{chorin2}
A.J. Chorin.
\newblock Numerical solution of the {Navier--Stokes} equations.
\newblock {\em Mathematics of Computation}, 23:341--354, 1968.

\bibitem{StaggeredDG}
E.~T. Chung and C.~S. Lee.
\newblock {A staggered discontinuous Galerkin method for the
  convection--diffusion equation}.
\newblock {\em Journal of Numerical Mathematics}, 20:1--31, 2012.

\bibitem{StaggeredDG2}
E.T. Chung, P.~Ciarlet, and T.F. Yu.
\newblock {Convergence and superconvergence of staggered discontinuous Galerkin
  methods for the three--dimensional Maxwell's equations on Cartesian grids}.
\newblock {\em Journal of Computational Physics}, 235:14--31, 2013.

\bibitem{StaggeredDGCE1}
E.T. Chung and B.~Engquist.
\newblock {Optimal discontinuous Galerkin methods for wave propagation}.
\newblock {\em SIAM Journal on Numerical Analysis}, 44:2131--2158, 2006.

\bibitem{StaggeredDGCE2}
E.T. Chung and B.~Engquist.
\newblock {Optimal discontinuous Galerkin methods for the acoustic wave
  equation in higher dimensions}.
\newblock {\em SIAM Journal on Numerical Analysis}, 47:3820--3848, 2009.

\bibitem{StaggeredDG3}
E.T. Chung, H.H. Kim, and O.B. Widlund.
\newblock {Two--level overlapping Schwarz algorithms for a staggered
  discontinuous Galerkin method}.
\newblock {\em SIAM Journal on Numerical Analysis}, 51:47--67, 2013.

\bibitem{CBS-convection-diffusion}
B.~Cockburn and C.~W. Shu.
\newblock The local discontinuous {Galerkin} method for time-dependent
  convection diffusion systems.
\newblock {\em SIAM Journal on Numerical Analysis}, 35:2440--2463, 1998.

\bibitem{cbs4}
B.~Cockburn and C.~W. Shu.
\newblock The {Runge}-{Kutta} discontinuous {Galerkin} method for conservation
  laws {V}: multidimensional systems.
\newblock {\em Journal of Computational Physics}, 141:199--224, 1998.

\bibitem{CBS-convection-dominated}
B.~Cockburn and C.~W. Shu.
\newblock {Runge}-{Kutta} discontinuous {Galerkin} methods for
  convection-dominated problems.
\newblock {\em Journal of Scientific Computing}, 16:173--261, 2001.

\bibitem{Crivellini2013}
A.~Crivellini, V.~D'Alessandro, and F.~Bassi.
\newblock {High-order discontinuous Galerkin solutions of three-dimensional
  incompressible RANS equations}.
\newblock {\em Computers and Fluids}, 81:122--133, 2013.

\bibitem{Dolejsi1}
V.~Dolejsi.
\newblock Semi-implicit interior penalty discontinuous galerkin methods for
  viscous compressible flows.
\newblock {\em Communications in Computational Physics}, 4:231--274, 2008.

\bibitem{Dolejsi2}
V.~Dolejsi and M.~Feistauer.
\newblock A semi-implicit discontinuous galerkin finite element method for the
  numerical solution of inviscid compressible flow.
\newblock {\em Journal of Computational Physics}, 198:727--746, 2004.

\bibitem{Dolejsi3}
V.~Dolejsi, M.~Feistauer, and J.~Hozman.
\newblock Analysis of semi-implicit dgfem for nonlinear convection-diffusion
  problems on nonconforming meshes.
\newblock {\em Computer Methods in Applied Mechanics and Engineering},
  196:2813--2827, 2007.

\bibitem{ADERNSE}
M.~Dumbser.
\newblock Arbitrary high order {PNPM} schemes on unstructured meshes for the
  compressible {Navier--Stokes} equations.
\newblock {\em Computers \& Fluids}, 39:60--76, 2010.

\bibitem{DumbserCasulli}
M.~Dumbser and V.~Casulli.
\newblock A staggered semi-implicit spectral discontinuous galerkin scheme for
  the shallow water equations.
\newblock {\em Applied Mathematics and Computation}, 219(15):8057--8077, 2013.

\bibitem{Erturk2008}
E.~Erturk.
\newblock Numerical solutions of 2d steady incompressible flow over a
  backward-facing step, part i: High reynolds number solutions.
\newblock {\em Computers and Fluids}, 37:633--655, 2008.

\bibitem{FambriDumbserCasulli}
F.~Fambri, M.~Dumbser, and V.~Casulli.
\newblock {An Efficient Semi-Implicit Method for Three-Dimensional
  Non-Hydrostatic Flows in Compliant Arterial Vessels}.
\newblock {\em International Journal for Numerical Methods in Biomedical
  Engineering}.
\newblock submitted to.

\bibitem{Feistauer2007}
M.~Feistauer and V.~Kucera.
\newblock On a robust discontinuous galerkin technique for the solution of
  compressible flow.
\newblock {\em Journal of Computational Physics}, 224:208--221, 2007.

\bibitem{Ferrer2011}
E.~Ferrer and R.H.J. Willden.
\newblock A high order discontinuous galerkin finite element solver for the
  incompressible navier–stokes equations.
\newblock {\em Computer and Fluids}, 46:224--230, 2011.

\bibitem{Fortin}
M.~Fortin.
\newblock {Old and new finite elements for incompressible flows}.
\newblock {\em International Journal for Numerical Methods in Fluids},
  1:347--364, 1981.

\bibitem{Fraenkel1961}
L.~Fraenkel.
\newblock On corner eddies in plane inviscid shear flow.
\newblock {\em Journal of Fluid Mechanics}, 11:400--406, 1961.

\bibitem{MunzDiffusionFlux}
G.~Gassner, F.~L\"orcher, and C.~D. Munz.
\newblock A contribution to the construction of diffusion fluxes for finite
  volume and discontinuous {Galerkin} schemes.
\newblock {\em Journal of Computational Physics}, 224:1049--1063, 2007.

\bibitem{Ghia1982}
U.~Ghia, K.~N. Ghia, and C.~T. Shin.
\newblock High-re solutions for incompressible flow using navier-stokes
  equations and multigrid method.
\newblock {\em Journal of Computational Physics}, 48:387--411, 1982.

\bibitem{GiraldoRestelli}
F.~X. Giraldo and M.~Restelli.
\newblock High-order semi-implicit time-integrators for a triangular
  discontinuous galerkin oceanic shallow water model.
\newblock {\em International Journal for Numerical Methods in Fluids},
  63:1077--1102, 2010.

\bibitem{shu2}
S.~Gottlieb and {C. W.} Shu.
\newblock Total variation diminishing {Runge}-{Kutta} schemes.
\newblock {\em Mathematics of Computation}, 67:73--85, 1998.

\bibitem{markerandcell}
F.H. Harlow and J.E. Welch.
\newblock Numerical calculation of time-dependent viscous incompressible flow
  of fluid with a free surface.
\newblock {\em Physics of Fluids}, 8:2182--2189, 1965.

\bibitem{cgmethod}
M.~R. Hestenes and E.~Stiefel.
\newblock Methods of conjugate gradients for solving linear systems.
\newblock {\em Journal of Research of the National Bureau of Standards},
  49:409--436, 1952.

\bibitem{Rannacher1}
J.~G. Heywood and R.~Rannacher.
\newblock {Finite element approximation of the nonstationary Navier-Stokes
  Problem. I. Regularity of solutions and second order error estimates for
  spatial discretization}.
\newblock {\em SIAM Journal on Numerical Analysis}, 19:275--311, 1982.

\bibitem{Rannacher3}
J.~G. Heywood and R.~Rannacher.
\newblock {Finite element approximation of the nonstationary Navier-Stokes
  Problem. III. Smoothing property and higher order error estimates for spatial
  discretization}.
\newblock {\em SIAM Journal on Numerical Analysis}, 25:489--512, 1988.

\bibitem{HirtNichols}
C.~W. Hirt and B.~D. Nichols.
\newblock Volume of fluid ({VOF}) method for dynamics of free boundaries.
\newblock {\em Journal of Computational Physics}, 39:201--225, 1981.

\bibitem{SUPG2}
T.J.R. Hughes, M.~Mallet, and M.~Mizukami.
\newblock {A new finite element formulation for computational fluid dynamics:
  II. Beyond SUPG}.
\newblock {\em Computer Methods in Applied Mechanics and Engineering},
  54:341--–355, 1986.

\bibitem{Khurshid2014}
H.~Khurshid and K.~A Hoffmann.
\newblock A high order numerical scheme for incompressible navier-stokes
  equations.
\newblock {\em The Arabian Journal for Science and Engineering}, 0:0--0, 2014.

\bibitem{Kim2013}
H.~H. Kim, E.~T. Chung, and C.S. Lee.
\newblock A staggered discontinuous galerkin method for the stokes system.
\newblock {\em SIAM Journal of Numerical Analysis}, 51:3327--3350, 2013.

\bibitem{KleinKummerOberlack2013}
B.~Klein, F.~Kummer, and M.~Oberlack.
\newblock {A SIMPLE based discontinuous Galerkin solver for steady
  incompressible flows}.
\newblock {\em Journal of Computational Physics}, 237:235--250, 2013.

\bibitem{CentralDG1}
Y.~J. Liu, C.~W. Shu, E.~Tadmor, and M.~Zhang.
\newblock Central discontinuous galerkin methods on overlapping cells with a
  non-oscillatory hierarchical reconstruction.
\newblock {\em SIAM Journal on Numerical Analysis}, 45:2442--2467, 2007.

\bibitem{CentralDG2}
Y.~J. Liu, C.~W. Shu, E.~Tadmor, and M.~Zhang.
\newblock L2-stability analysis of the central discontinuous galerkin method
  and a comparison between the central and regular discontinuous galerkin
  methods.
\newblock {\em Mathematical Modeling and Numerical Analysis}, 42:593--607,
  2008.

\bibitem{Nguyen2011}
N.C. Nguyen, J.~Peraire, and B.~Cockburn.
\newblock An implicit high-order hybridizable discontinuous galerkin method for
  the incompressible navier-stokes equations.
\newblock {\em Journal of Computational Physics}, 230:1147--1170, 2011.

\bibitem{patankar}
V.S. Patankar.
\newblock {\em Numerical {Heat} {Transfer} and {Fluid} {Flow}}.
\newblock Hemisphere Publishing Corporation, 1980.

\bibitem{patankarspalding}
V.S. Patankar and B.~Spalding.
\newblock A calculation procedure for heat, mass and momentum transfer in
  three-dimensional parabolic flows.
\newblock {\em International Journal of Heat and Mass Transfer}, 15:1787--1806,
  1972.

\bibitem{Lixia2013}
L.~Qu, C.~Norberg, L.~Davidson, S.H. Peng, and F.~Wang.
\newblock Quantitative numerical analysis of flow past a circular cylinder at
  reynolds number between 50 and 200.
\newblock {\em Journal oFluids and Structures}, 39:347--370, 2013.

\bibitem{Rhebergen2012}
S.~Rhebergen and B.~Cockburn.
\newblock {A space–time hybridizable discontinuous Galerkin method for
  incompressible flows on deforming domains}.
\newblock {\em Journal of Computational Physics}, 231:4185--4204, 2012.

\bibitem{Rhebergen2013}
S.~Rhebergen, B.~Cockburn, and Jaap~J.W. van~der Vegt.
\newblock {A space–time discontinuous Galerkin method for the incompressible
  Navier–Stokes equations}.
\newblock {\em Journal of Computational Physics}, 233:339--358, 2013.

\bibitem{Rusanov:1961a}
V.~V. Rusanov.
\newblock {Calculation of Interaction of Non--Steady Shock Waves with
  Obstacles}.
\newblock {\em J. Comput. Math. Phys. USSR}, 1:267--279, 1961.

\bibitem{GMRES}
Y.~Saad and M.H. Schultz.
\newblock {GMRES:} a generalized minimal residual algorithm for solving
  nonsymmetric linear systems.
\newblock {\em SIAM Journal on Scientific and Statistical Computing},
  7:856–--869, 1986.

\bibitem{Shahbazi2007}
K.~Shahbazi, P.~F. Fischer, and C.~R. Ethier.
\newblock A high-order discontinuous galerkin method for the unsteady
  incompressible navier-stokes equations.
\newblock {\em Journal of Computational Physics}, 222:391--407, 2007.

\bibitem{shuosher1}
C.~W. Shu and S.~Osher.
\newblock Efficient implementation of essentially non-oscillatory shock
  capturing schemes.
\newblock {\em Journal of Computational Physics}, 77:439--471, 1988.

\bibitem{2DSIUSW}
M.~Tavelli and M.~Dumbser.
\newblock A high order semi-implicit discontinuous galerkin method for the two
  dimensional shallow water equations on staggered unstructured meshes.
\newblock {\em Applied Mathematics and Computation}, 0:0--0, 2014.

\bibitem{TavelliDumbserCasulli}
M.~Tavelli, M.~Dumbser, and V.~Casulli.
\newblock High resolution methods for scalar transport problems in compliant
  systems of arteries.
\newblock {\em Applied Numerical Mathematics}, 74:62--82, 2013.

\bibitem{TaylorHood}
C.~Taylor and P.~Hood.
\newblock {A numerical solution of the Navier-Stokes equations using the finite
  element technique}.
\newblock {\em Computers and Fluids}, 1:73--100, 1973.

\bibitem{titarevtoro}
V.~A. Titarev and E.~F. Toro.
\newblock {ADER} schemes for three-dimensional nonlinear hyperbolic systems.
\newblock {\em Journal of Computational Physics}, 204:715--736, 2005.

\bibitem{USFORCE}
E.~F. Toro, A.~Hidalgo, and M.~Dumbser.
\newblock {FORCE} schemes on unstructured meshes {I}: Conservative hyperbolic
  systems.
\newblock {\em Journal of Computational Physics}, 228:3368--–3389, 2009.

\bibitem{toro4}
{E. F.} Toro and {V. A.} Titarev.
\newblock Solution of the generalized {Riemann} problem for advection-reaction
  equations.
\newblock {\em Proc. Roy. Soc. London}, pages 271--281, 2002.

\bibitem{TumoloBonaventuraRestelli}
G.~Tumolo, L.~Bonaventura, and M.~Restelli.
\newblock {A semi-implicit, semi-Lagrangian, p-adaptive discontinuous Galerkin
  method for the shallow water equations }.
\newblock {\em Journal of Computational Physics}, 232:46--67, 2013.

\bibitem{vanKan}
J.~van Kan.
\newblock {A second-order accurate pressure correction method for viscous
  incompressible flow}.
\newblock {\em SIAM Journal on Scientific and Statistical Computing},
  7:870--891, 1986.

\bibitem{Verfuerth}
R.~Verf\"urth.
\newblock {Finite element approximation of incompressible Navier-Stokes
  equations with slip boundary condition II}.
\newblock {\em Numerische Mathematik}, 59:615--636, 1991.

\bibitem{WaltersCasulli1998}
R.~A. Walters and V.~Casulli.
\newblock A robust finite element model for hydrostatic surface water flows.
\newblock {\em Communications in Numerical Methods in Engineering},
  14:931--940, 1998.

\bibitem{Williamson1998}
C.H.K. Williamson and G.L. Brown.
\newblock A series in $1/\sqrt{Re}$ to represent the strouhal-reynolds number
  relationship of the cylinder wake.
\newblock {\em Journal oFluids and Structures}, 12:1073--1085, 1998.

\bibitem{Womersley1995}
J.~Womersley.
\newblock Method for the calculation of velocity, rate of flow and viscous drag
  in arteries when the pressure gradient is known.
\newblock {\em Journal of Physiology}, 127:553--563, 1955.

\end{thebibliography}
\bibliographystyle{plain}

\end{document}